\documentclass[11pt]{article}

\usepackage{amsmath, amsthm, amssymb, amsfonts, mathtools}
\usepackage{mathrsfs}
\usepackage{hyperref}
\usepackage{enumitem}
\usepackage{graphicx}
\usepackage{tikz}
\usepackage{bbm}
\usepackage{geometry}
\usepackage{graphicx}
\usepackage{stmaryrd} 
\usepackage[title]{appendix}
\usetikzlibrary{arrows.meta,patterns}

\newtheorem{theorem}{Theorem}[section]
\newtheorem{lemma}[theorem]{Lemma}
\newtheorem{proposition}[theorem]{Proposition}
\newtheorem{corollary}[theorem]{Corollary}
\theoremstyle{definition}
\newtheorem{definition}[theorem]{Definition}
\newtheorem{remark}[theorem]{Remark}

\newcommand{\R}{\mathbb{R}}
\newcommand{\Z}{\mathbb{Z}}

\newcommand{\HH}{\mathbb{H}}
\newcommand{\F}{\mathbb{F}}
\newcommand{\FF}{\widetilde{\mathcal F}}
\newcommand{\D}{\mathbb{D}}
\newcommand{\N}{\mathbb{N}}

\newcommand{\E}{\mathbb{E}}
\newcommand{\Prob}{\mathbb{P}}

\newcommand{\Kexp}{K^{\mathrm{exp}}}

\newcommand{\Zfull}{Z^{\mathrm{full}}}
\newcommand{\Lfull}{L^{\mathrm{full}}}
\newcommand{\Zhalf}{Z^{\mathrm{half}}}
\newcommand{\Lhalf}{L^{\mathrm{half}}}
\newcommand{\Zin}{Z^{\mathrm{in}}}
\newcommand{\Zexit}{Z^{\mathrm{exit}}}

\newcommand{\Zb}{Z^{\mathrm{b}}}

\newcommand{\I}{{\rm i}}

\title{Bulk Directed Landscape Convergence and Phase Transitions in Half-Space Models}
\author{Xinyi Zhang}
\begin{document}

\maketitle
\begin{abstract}
We study the effects of the boundary in the bulk of two half-space models. For half-space exponential last passage percolation (LPP) and the half-space log-gamma polymer, we establish convergence of their respective rescaled last passage time and free energy to the directed landscape when the observation window is at distance $f(n)n^{2/3}$ from the boundary and the attractive boundary strength is $o(f(n)n^{-1/3})$, where $f(n)\to\infty$ may diverge arbitrarily slowly. We also show that this condition on the boundary strength is sharp at the level of scales: any boundary strength of order $f(n)n^{-1/3}$ changes the limiting behavior for sufficiently large temporal separations. For exponential LPP, we prove a phase transition in the point-to-point last passage time as the boundary attraction varies on the scale \(f(n)n^{-1/3}\). As the attraction increases, the fluctuations transition from a bulk-dominated KPZ regime to a boundary-dominated Gaussian regime, with a critical window governed by the positive part of a shifted Gaussian. The limiting law in the critical window manifests the competition between a Gaussian boundary contribution and a bulk contribution whose magnitude vanishes under the Gaussian scaling.
\end{abstract}

\section{Introduction}

The Kardar--Parisi--Zhang universality class describes the large-scale
fluctuations of a broad family of random growth models, interacting
particle systems, last passage percolation models, and directed
polymers in $1+1$ dimensions. Over longitudinal distances of order $n$, its characteristic
$1{:}2{:}3$ scaling predicts fluctuations of order $n^{1/3}$ and
transversal displacements of order $n^{2/3}$. The directed landscape $\mathcal L$, constructed by
Dauvergne, Ortmann, and Vir\'ag
\cite{dauvergne2022directed}, is the universal four-parameter scaling
limit expected to govern point-to-point observables in this class.

Directed landscape convergence has now been established for several
integrable full-space models. At zero temperature, this includes
Poissonian, geometric, and exponential last passage percolation; see,
for example,
\cite{dauvergne2021scaling,dauvergne2023uniform}.
Positive-temperature analogues are more delicate: exact formulas
typically provide only one-point asymptotics; the partition
function sums over all admissible paths; and the resulting free energy
does not satisfy the metric composition law available at zero
temperature. Nevertheless, directed landscape convergence has recently
been obtained for several positive-temperature models, including the
KPZ equation, the colored ASEP and stochastic six-vertex models, and the log-gamma polymer; see
\cite{wu2023kpz,aggarwal2024scaling,zhang2025convergence}.

The presence of a boundary introduces new phenomena: paths can interact with the diagonal boundary, and
the strength of this interaction creates distinct fluctuation
regimes and depinning transitions. The rigorous study of half-space last passage models goes back to Baik and Rains, who established phase transitions from Tracy-Widom GSE to Gaussian fluctuations for point-to-point passage times with both endpoints on the diagonal and identified a nontrivial crossover distribution in the critical window \cite{BR01a,BR01b,BR01c}. The Baik-Rains phase transition has also been established for half-space exponential LPP and the half-space log-gamma polymer in \cite{Pfaffian_Schur} and \cite{imamura2022solvablemodelskpzclass}. For exponential LPP, \cite[Theorem~1.7]{Pfaffian_Schur} further establishes a two-parameter crossover between GSE, GUE, and GOE distributions by fixing one endpoint on the diagonal and simultaneously moving the other endpoint on the order $n^{2/3}$ away from the diagonal and tuning the boundary rate within order $n^{-1/3}$ of its critical value. At the level of process limits, a major recent advance is the
construction of the half-space directed landscape by Dauvergne and
Zhang \cite{dauvergne2026half}. They proved convergence of half-space exponential LPP in the critical window, where the diagonal rate is $\alpha_n=1/2+b n^{-1/3}$ for fixed $b\in\mathbb R$, and in the subcritical regime, where the diagonal rate $\alpha>1/2$ is fixed. The convergence holds under KPZ scaling around the diagonal, with transversal distances scaled by $n^{2/3}$.

This naturally raises the question of what happens beyond the KPZ transversal scale. We translate the observation window into the bulk by $f(n)n^{2/3}$, where $f(n) = O(n^{1/3})$ and $f(n)\to\infty$ may diverge arbitrarily slowly. The effect of the boundary is then governed by a competition between the distance to the diagonal and the strength of the boundary attraction. On a temporal scale of order $n$, the curvature of the bulk shape function, the deterministic first-order behavior of the last passage time or free energy, suggests that a transverse excursion of size $f(n)n^{2/3}$ incurs a cost of order $f(n)^2n^{1/3}$. In contrast, an attractive boundary perturbation of size $g(n)n^{-1/3}$ offers a potential shape gain of order $g(n)^2n^{1/3}$, following from Taylor expansions of the half-space diagonal-to-diagonal shape functions. This comparison suggests that the boundary should be asymptotically invisible when $g(n)=o(f(n))$ and dominant when $f(n)=o(g(n))$, with a transition expected when $f(n)$ and $g(n)$ are of the same order.

\begin{figure}[!htb]
\centering
{\large\bfseries Phase diagrams for half-space exponential LPP for fixed $T>0$\par}
\smallskip
\definecolor{bulkblue}{RGB}{35,85,170}
\definecolor{boundaryorange}{RGB}{198,40,40}
\definecolor{criticalpurple}{RGB}{125,91,174}
\definecolor{transitionred}{RGB}{174,83,89}


\begin{minipage}[t]{0.48\linewidth}
\vspace{0pt}\centering
\textbf{(a) Total last-passage time}
\par\smallskip
\resizebox{0.92\linewidth}{!}{%
\begin{tikzpicture}[font=\small,>=Stealth]
  \path[use as bounding box] (-0.5,-0.5) rectangle (7.35,7.35);
  \fill[bulkblue!6]
    (0,0) -- (6.4,0) -- (6.4,6.4) -- cycle;
  \fill[boundaryorange!8]
    (0,0) -- (6.4,6.4) -- (0,6.4) -- cycle;
  \draw[criticalpurple,line width=2.8pt,line cap=round]
    (0,0) -- (6.4,6.4);
  \node[text=criticalpurple,font=\normalsize]
    (competitionlabel) at (5.0,3.15) {$y=\frac{x}{2T}$};
  \draw[->,criticalpurple!75,line width=0.5pt]
    (competitionlabel.north) -- (4.55,4.55);
  \draw[->,thick] (0,0) -- (6.85,0)
    node[right,font=\large] {$x$};
  \draw[->,thick] (0,0) -- (0,6.85)
    node[above,font=\large] {$y$};
  \node[align=center,text=boundaryorange] at (2.1,5.05)
    {\textbf{Boundary-dominated}\\[5pt] Gaussian};
  \node[align=center,text=bulkblue] at (4.6,1.25)
    {\textbf{Bulk-dominated}\\[5pt] GUE Tracy--Widom};
\end{tikzpicture}%
}
\end{minipage}\hfill%
\begin{minipage}[t]{0.48\linewidth}
\vspace{0pt}\centering
\textbf{(b) Boundary-touching part}
\par\smallskip
\resizebox{0.92\linewidth}{!}{%
\begin{tikzpicture}[font=\small,>=Stealth]
  \path[use as bounding box] (-0.5,-0.5) rectangle (7.35,7.35);
  \pgfmathsetmacro{\boundarycurvebend}{0.9}
  \fill[boundaryorange!8]
    (0,0) rectangle (6.4,6.4);
  \fill[bulkblue!6]
    (0,0) -- (6.4,0)
    -- plot[domain=6.4:0,samples=81,variable=\t]
      (\t,{\t/(1+sqrt(1-(\boundarycurvebend*\t/6.4)^2))})
    -- cycle;
  \foreach \j in {0,...,40}{%
    \pgfmathsetmacro{\boundarycurvex}{0.16*\j}
    \fill[transitionred]
      (\boundarycurvex,
       {\boundarycurvex/(1+sqrt(1-(\boundarycurvebend*\boundarycurvex/6.4)^2))})
      circle[radius=1pt];
  }
  \node[text=transitionred,font=\normalsize]
    (touchlabel) at (3.95,3.85) {$y=\frac{x}{2\left(T+\sqrt{T^2-\kappa^2x^2}\right)}$};
  \draw[->,transitionred!75,line width=0.5pt]
    (touchlabel.south)
    -- (5.1,{5.1/(1+sqrt(1-(\boundarycurvebend*5.1/6.4)^2))});
  \draw[->,thick] (0,0) -- (6.85,0)
    node[right,font=\large] {$x$};
  \draw[->,thick] (0,0) -- (0,6.85)
    node[above,font=\large] {$y$};
  \node[align=center,text=boundaryorange] at (2.45,4.6)
    {\textbf{Gaussian}};
  \node[text=bulkblue] at (4.7,1.0)
    {\textbf{GUE (predicted)}};
\end{tikzpicture}%
}
\end{minipage}
\par\smallskip

\setlength{\fboxsep}{6pt}
\begin{minipage}[t]{0.48\linewidth}
\vspace{0pt}%
\colorbox{criticalpurple!7}{%
\begin{minipage}[t]{\dimexpr\linewidth-2\fboxsep\relax}
\small\raggedright
{\color{criticalpurple}\textbf{Critical window (solid line)}}
\par\smallskip
In the critical window, 
\[
  g(n)=\frac{f(n)}{2T}+\xi f(n)^{-1/2},\qquad \xi\in\mathbb R,
\]
the rescaled last-passage time converges to the positive part of a shifted Gaussian.
\end{minipage}%
}
\end{minipage}\hfill%
\begin{minipage}[t]{0.48\linewidth}
\vspace{0pt}%
\colorbox{bulkblue!6}{%
\begin{minipage}[t]{\dimexpr\linewidth-2\fboxsep\relax}
\small\raggedright
{\color{transitionred}\textbf{Earlier transition (dotted line)}}
\par\smallskip
The Gaussian phase above the dotted line is proved in Theorem \ref{thm: Gaussian} and \ref{thm: Gaussian2}.
The GUE phase below is conjectural,
and the crossover law on the line remains to be identified.
\par\smallskip
This transition occurs while the total last-passage time is still
bulk-dominated.
\end{minipage}%
}
\end{minipage}
\par\smallskip
\caption{ Let $r(n) \rightarrow \infty$ with $r(n) = O(n^{1/3})$ as the common scale where $f(n)/r(n) \rightarrow x$ and $g(n)/r(n) \rightarrow y$. The axes $x,y$ thus represent the asymptotic amplitudes of $f$ and $g$ on the common scale. Moreover, $\kappa := \lim_{n \rightarrow \infty}r(n)/n^{1/3}$. See Theorem \ref{thm:phase diagram} for the exact centering and scaling for each fluctuation regime.}
\label{fig:phase-diagram}
\end{figure}

\begin{figure}[!htb]
\centering
\definecolor{pathblue}{RGB}{35,85,170}
\definecolor{pathred}{RGB}{198,40,40}
\definecolor{pathpurple}{RGB}{125,70,170}
\resizebox{0.72\linewidth}{!}{%
\begin{tikzpicture}[
  font=\small, >=Stealth,
  bulk/.style={draw=pathpurple,line width=2.2pt,line cap=round},
  access/.style={draw=pathpurple,line width=2.2pt,line cap=round},
  pinned/.style={draw=pathred,line width=2.5pt,line cap=round}
]
  \coordinate (u) at (3.5,0.1);
  \coordinate (v) at (12,8.6);
  \coordinate (r) at (4.5,4.5);
  \coordinate (s) at (7.5,7.5);
  \coordinate (c) at (6.2,2.8);
  \coordinate (d) at (9.2,5.8);

  \draw[black,line width=1.2pt] (0.65,0.65) -- (9.1,9.1);
  \node[rotate=45,anchor=south,font=\normalsize]
    at (7.05,7.55) {\textbf{Half-space boundary}};
  \draw[pathpurple!55,line width=0.8pt,dash pattern=on 6pt off 5pt]
    (r) -- (c);
  \draw[pathpurple!55,line width=0.8pt,dash pattern=on 6pt off 5pt]
    (s) -- (d);

  \draw[bulk,draw=pathblue] (u)
    .. controls (3.75,0.28) and (3.88,0.65) .. (4.1,0.7)
    .. controls (4.45,0.78) and (4.56,1.3) .. (4.95,1.4)
    .. controls (5.3,1.6) and (5.5,2.02) .. (5.62,2.2)
    .. controls (5.85,2.6) and (6.0,2.6) .. (c);
  \draw[bulk,draw=pathred] (c)
    .. controls (6.5,2.9) and (6.6,3.35) .. (6.95,3.42)
    .. controls (7.35,3.55) and (7.42,3.92) .. (7.62,4.2)
    .. controls (7.8,4.48) and (8.06,4.51) .. (8.28,4.84)
    .. controls (8.5,5.16) and (8.75,5.62) .. (d);
  \draw[bulk,draw=pathblue] (d)
    .. controls (9.5,5.9) and (9.67,6.35) .. (9.9,6.48)
    .. controls (10.2,6.67) and (10.26,7.12) .. (10.65,7.18)
    .. controls (10.97,7.27) and (11.2,7.75) .. (11.4,7.92)
    .. controls (11.72,8.2) and (11.8,8.44) .. (v);

  \draw[access] (u)
    .. controls (3.6,0.4) and (3.7,0.6) .. (3.72,1.1)
    .. controls (3.74,1.6) and (3.95,1.65) .. (3.98,2.2)
    .. controls (4.0,2.65) and (4.16,2.8) .. (4.2,3.25)
    .. controls (4.24,3.8) and (4.44,4.05) .. (r);
  \draw[access] (s)
    .. controls (7.8,7.53) and (8.15,7.7) .. (8.6,7.72)
    .. controls (8.9,7.72) and (8.98,7.93) .. (9.3,7.95)
    .. controls (9.62,7.98) and (9.95,8.1) .. (10.22,8.18)
    .. controls (10.45,8.3) and (10.8,8.3) .. (11.15,8.35)
    .. controls (11.48,8.38) and (11.72,8.58) .. (v);

  \draw[pinned] (r)
    .. controls (4.62,4.55) and (4.78,4.55) .. (4.95,4.8)
    .. controls (5.1,4.93) and (5.15,5.1) .. (5.35,5.18)
    .. controls (5.55,5.22) and (5.65,5.52) .. (5.8,5.68)
    .. controls (6.05,5.91) and (6.08,6.05) .. (6.25,6.17)
    .. controls (6.43,6.23) and (6.56,6.5) .. (6.75,6.6)
    .. controls (7.0,6.88) and (7.2,7.15) .. (s);

  \fill (u) circle[radius=2pt] node[below right] {$(f(n)n^{2/3},0)$};
  \fill (v) circle[radius=2pt] node[above right] {$(f(n)n^{2/3} + Tn,Tn)$};

 \node[
  text=pathpurple,
  align=center,
  font=\small\bfseries\boldmath
] (cost) at (7.8,1.3)
  {Shape cost\\[4pt]
   of order $f(n)^2n^{1/3}$};
\draw[->,pathpurple,line width=0.7pt]
  (cost.west) -- (4.1,2.2);

   \node[
  text=pathpurple,
  align=center,
  font=\small\bfseries\boldmath
] (cost) at (12.8,6)
  {Shape cost\\[4pt]
   of order $f(n)^2n^{1/3}$};

  \draw[->,pathpurple,line width=0.7pt]
  (cost.north) -- (10,8);
 \node[
  text=pathred,
  align=center,
  font=\small\bfseries\boldmath
] (gain) at (2,5.5)
  {Boundary gain\\[4pt]
   of order $g(n)^2n^{1/3}$\\[4pt]
   in Gaussian phase};
\draw[->,pathred,line width=0.7pt]
  (gain.east) -- (5.25,5.5);
  \node[
    align=center,
    text width=4.1cm,
    font=\small\bfseries,
    inner sep=3pt,
    fill=white,
    rounded corners=1pt
  ] (equal) at (12.0,3.2)
    {Same bulk shape until the red boundary part enters
     the Gaussian phase};
  \draw[->,black,line width=0.6pt,shorten >=1pt]
    (equal.north west) -- (7.08,6.9);
  \draw[->,black,line width=0.6pt,shorten >=1pt]
    (equal.west) -- (7.3,3.6);

\end{tikzpicture}%
}
\caption{Competition between bulk paths and boundary-touching paths. The purple dotted lines indicate the anti-diagonal directions. The depicted paths are schematic examples of a bulk path and a boundary-touching path; the actual maximizing paths need not have these shapes.}
\label{fig:bulk-boundary-competition}
\end{figure}

Indeed, we prove in Theorems~\ref{thm:lpp-main} and \ref{thm_half_loggamma_to_DL} that when the attractive boundary strength is \(o(f(n)n^{-1/3})\),
both half-space exponential LPP and the half-space log-gamma polymer
converge to the directed landscape, as their full-space
counterparts. We also show that this condition is sharp at the level of scales in Theorem~\ref{thm:sharpness-main}. When
the attractive boundary strength is of order \(f(n)n^{-1/3}\), the corresponding one-point marginals of the last passage time and free energy diverge to $+\infty$ in probability under the full-space centering
and normalization for every sufficiently large fixed temporal separation.

Moreover, for exponential LPP, we prove a one-point phase transition (Theorem \ref{thm:phase diagram}) governed by the ratio of $g(n)$ to $f(n)$ as shown in Figure \ref{fig:phase-diagram}. The last passage time is the maximum of two competing contributions: the boundary-touching part, obtained by maximizing over all paths that intersect the diagonal, and the bulk part, obtained by maximizing over paths that avoid the diagonal; see Figure \ref{fig:bulk-boundary-competition}. 

\begin{itemize}
    \item When the boundary attraction is sufficiently weak, the boundary-touching part receives no shape gain from the boundary but still incurs a cost of order $f(n)^2n^{1/3}$ in reaching it. Geometrically, under weak boundary attraction, the two red paths in
Figure~\ref{fig:bulk-boundary-competition} share the same bulk shape, whereas the purple paths deviating from the diagonal direction induces an additional shape cost of order $f(n)^2n^{1/3}$ comparing to their blue counterparts. The bulk part therefore has a larger shape function and enforces a GUE Tracy-Widom limit for the total last passage time.

\item As the boundary attraction increases, the boundary-touching part
is expected to undergo a transition from KPZ fluctuations on the $n^{1/3}$ scale to Gaussian fluctuations of order $\sqrt{g(n)}n^{1/3}$. We establish the Gaussian asymptotics in the region above the transition threshold as indicated by the dotted line in Figure~\ref{fig:phase-diagram}\textup{(b)}. When $f(n),g(n)\asymp n^{1/3}$, a BBP crossover \cite{BBP05} is predicted in \cite[Remark~1.5]{Pfaffian_Schur} for this transition; the crossover law when $f(n),g(n)=o(n^{1/3})$ remains to be identified. Crucially, the boundary-touching part enters the Gaussian regime
while its shape function still lies below that of the bulk part
by a gap of order $f(n)^2n^{1/3}$. The boundary fluctuations of order $\sqrt{g(n)}n^{1/3}$ are too small to bridge the shape gap so the total last passage time retains its GUE limit until the shape gap shrinks to the boundary fluctuation scale.

\item A further increase in the boundary attraction brings the system into the critical window, where the shape gap is comparable to the boundary fluctuation scale. After centering by the bulk shape and rescaling by $\sqrt{g(n)}n^{1/3}$, the boundary-touching and bulk parts converge to a shifted Gaussian and zero, respectively. Their maximum
therefore yields the positive part of a shifted Gaussian as the
critical law.

\item Beyond the critical window, the boundary part has the larger shape and thus the total last passage time exhibits a Gaussian limit under the boundary-part centering and scaling. 

\end{itemize}

We can access the limiting law in the critical window because there is a separation of the thresholds as shown in Figure \ref{fig:phase-diagram} where boundary-touching part enters the Gaussian regime at a strictly lower boundary strength than that required to close its shape gap with the bulk part. These thresholds can be understood through the following heuristic of shape function optimization. 

For half-space exponential LPP from $(0,0)$ to $(\ell, \ell)$ with attractive diagonal rate
$\alpha_n=1/2-g(n)n^{-1/3}$, the corresponding shape is
\[
\frac{\ell}{\alpha_n(1-\alpha_n)} =\frac{4\ell}{1-4g(n)^2n^{-2/3}}.
\]
On the other hand, the full-space shape from $(0,0)$ to $(a,b)$
is $(\sqrt{a} + \sqrt{b})^2$. Thus, the shape function of boundary-touching part from $(f(n)n^{2/3}, 0)$ to $(f(n)n^{2/3}+Tn, Tn)$ should have the following variational form
\[
\max_{0 \leq \ell \leq Tn-f(n)n^{2/3}}\varphi_n (\ell):= \max_{0 \leq \ell \leq T n- f(n)n^{2/3}}\left\{2(Tn-\ell) +2\sqrt{(Tn-\ell)^2-f(n)^2n^{4/3}} +\frac{4\ell}{1-4g(n)^2n^{-2/3}}\right\}.
\]
Here $\ell$ represents the length of the diagonal segment. The first two terms give the combined shape of the purple entrance
and exit paths in Figure~\ref{fig:bulk-boundary-competition}, and the final term gives the shape of the red diagonal segment. Since $\varphi_n$ is strictly concave, the sign change of
$\varphi_n'(0)$ identifies the threshold at which a positive-length
diagonal segment becomes favorable at the shape level. This threshold
leads to the conjectural fluctuation transition line in
Figure~\ref{fig:phase-diagram}\textup{(b)}, above which we prove
Gaussian fluctuations. On the other hand, the boundary strength at which $\max_\ell\varphi_n(\ell)$ crosses the bulk shape $4Tn$ determines the competition threshold between the boundary-touching part and bulk part in Figure~\ref{fig:phase-diagram}\textup{(a)}. The separation of these thresholds arises because a diagonal segment can already improve
the boundary-touching shape while the associated boundary gain
remains insufficient to compensate for the shape cost of reaching the boundary.

We expect our bulk directed landscape convergence results and the competition mechanism between boundary-touching and bulk parts to extend to a broad class of half-space models in the KPZ universality class. 

Our proofs of directed landscape convergences rely on couplings between the full-space and half-space environments. For exponential LPP, we use geodesics ordering to reduce the uniform comparison of the full-space and half-space models to two boundary-avoidance events involving only a fixed pair of endpoints bracketing the whole observation window (Lemma \ref{lem:lpp-monotonicity}). For the log-gamma polymer, because every path contributes in the partition function, we have to use Holley's inequality instead to provide a stochastic ordering of polymer measures (Proposition \ref{prop:monotonicity}). This bounds the total
boundary-touching polymer mass for every pair in the observation window by a single ratio of restricted partition functions associated with one
pair of endpoints. These two reductions avoid union bounds over the observation window, which is essential for allowing $f(n)$ to diverge arbitrarily slowly. We then show that these events occur with probability tending to one using a collection of tail estimates for last passage times and polymer free energies in Section \ref{sec:4} and \ref{sec:5}. The coupling argument is probabilistic in nature and does not rely on exact integrable formulas. For models admitting analogous couplings, our approach offers a route from full-space directed landscape convergence to bulk convergence in half-space, provided suitable estimates are available.

Our proof of the one-point phase transition for exponential LPP uses additional integrable inputs: half-space shift invariance \cite{he2026shift} and exact off-diagonal formulas \cite{Pfaffian_Schur}. Shift invariance identifies the boundary-touching passage time with an off-diagonal half-space passage time, whose Gaussian asymptotics we obtain from the steepest descent analysis of the kernel formulas. In the critical window, the centered bulk contribution vanishes on the Gaussian scale, so no joint limiting distribution for the two competing contributions is needed. Taking their maximum then yields the positive part of a shifted Gaussian. Related shift-invariance results are available for several half-space integrable models \cite{he2026shift}. We also expect an analogous shift-invariance result for the log-gamma polymer. Although we are not aware of an off-diagonal Fredholm Pfaffian formula suitable for asymptotic analysis for the half-space log-gamma polymer, the Gaussian limit may still be accessible through probabilistic and  geometric methods like \cite{ginsburg2024pinning}. Together with the shift-invariance identity, Gaussian asymptotics valid in a regime before the shape gap closes and covering the critical window should yield an analogous phase transition.

\section{Models and Main Results}
\label{sec:model}

\subsection{Notations}
\label{subsec:common-notation}

Let $\prec$ denote the partial order on $\mathbb Z^2$ given by
$(x_1,y_1)\prec(x_2,y_2)$ if $x_1\le x_2$ and $y_1\le y_2$. For any
subset $\mathbb O\subset\mathbb Z^2$, let $\Pi_{\mathbb O}[p\to q]$ denote the set of up-right lattice paths from $p$ to $q$ contained in
$\mathbb O$, and abbreviate $\Pi_{\mathbb Z^2}[p\to q]$ as
$\Pi[p\to q]$. We write
\begin{equation}
\mathbb H=\{(i,j)\in\mathbb Z^2:i\ge j\},
\qquad
\mathbb F=\{(i,j)\in\mathbb Z^2:i>j\},
\qquad
\D=\{(i,i):i\in\mathbb Z\}.
\end{equation}
We also define $\mathbb R^4_+
=
\{(x,s;y,t)\in\mathbb R^4:t>s\}$
and $C(\mathbb R^4_+,\mathbb R)$ to be the space of continuous functions on
$\mathbb R^4_+$, endowed with the topology of uniform convergence on compact
sets.

Throughout the paper, we fix a deterministic function $f:\mathbb N\rightarrow(0,\infty)$ such that $
f(n)\rightarrow\infty$. Whenever
an auxiliary point $p=(x,y)\in\mathbb R^2$ is used as a lattice endpoint, it
is understood to be replaced by its coordinatewise floor. For notational
convenience, we suppress all floor symbols. For integers $a\le b$, we write $\llbracket a,b\rrbracket=\{a,a+1,\ldots,b\}.$

\subsection{Exponential last passage percolation}
\label{subsec:lpp-model}

\begin{definition}[Full-space exponential LPP]
Let $\{\nu^\mathrm{full}_{i,j}\}_{(i,j) \in \Z^2}$ be i.i.d. $\mathrm{Exp}(1)$ random variables. For $(x_1,y_1)\prec(x_2,y_2)$, define the full-space last passage time
\begin{equation}\label{eq:lpp-full-definition}
L^{\mathrm{full}}(x_1,y_1;x_2,y_2)
=
\max_{\pi\in\Pi[(x_1,y_1)\to (x_2,y_2)]}
\sum_{(i,j)\in\pi\setminus (x_1,y_1)}\nu^{\mathrm{full}}_{i,j},
\end{equation}
and set $L^{\mathrm{full}}(p;q)=0$ if the sum is empty.
\end{definition}

\begin{definition}[Half-space exponential LPP]
For $\rho > -1/2$, let $\{\nu^\mathrm{half}_{i,j}\}_{(i,j) \in \F}$ be i.i.d. $\mathrm{Exp}(1)$ random variables and, independently, let $\{\nu^\mathrm{half}_{i,i}\}_{i \in \Z}$ be i.i.d. $\mathrm{Exp}(1/2+\rho)$ random variables. For $(x_1,y_1),(x_2,y_2) \in \HH$ and $(x_1,y_1) \prec (x_2,y_2)$, define the half-space last passage time
\begin{equation}\label{eq:lpp-half-definition}
L^{\mathrm{half}}(x_1,y_1;x_2,y_2)
=
\max_{\pi\in\Pi_{\HH}[(x_1,y_1)\to (x_2,y_2)]}
\sum_{(i,j)\in\pi\setminus (x_1,y_1)}\nu^{\mathrm{half}}_{i,j},
\end{equation}
and set $L^{\mathrm{half}}(p;q)=0$ if the sum is empty.
\end{definition}

\subsection{The log-gamma polymer}
\label{subsec:loggamma-model}

\begin{definition}[Full-space log-gamma polymer]
Let $\{\omega^{\mathrm{full}}_{i,j}\}_{(i,j) \in \Z^2}$ be i.i.d.\ inverse-gamma random variables with parameter~$2\alpha>0$. For $(x_1,y_1)\prec(x_2,y_2)$, define the full-space log-gamma partition function
\[
Z^{\mathrm{full}}(x_1,y_1;x_2,y_2)
:=\sum_{\pi \in \Pi[(x_1,y_1)\to(x_2,y_2)]} 
\prod_{(i,j)\in\pi \setminus (x_1,y_1)} \omega^{\mathrm{full}}_{i,j},
\]
and set $Z^{\mathrm{full}}(x_1,y_1;x_2,y_2)=0$ if the sum is empty.
The associated full-space free energy is
\[
\log Z^{\mathrm{full}}(x_1,y_1;x_2,y_2) = \log \sum_{\pi \in \Pi[(x_1,y_1)\to(x_2,y_2)]} 
\prod_{(i,j)\in\pi \setminus (x_1,y_1)} \omega^{\mathrm{full}}_{i,j}.
\]
\end{definition}

\begin{definition}[Half-space log-gamma polymer]
For $\alpha > 0$ and $\theta > -\alpha$, let $\{\omega^\mathrm{half}_{i,j}\}_{(i,j) \in \F}$ be i.i.d. $\mathrm{Gamma}^{-1}(2\alpha)$ and, independently, let $\{\omega^\mathrm{half}_{i,i}\}_{i \in \Z}$ be i.i.d. $\mathrm{Gamma}^{-1}(\alpha + \theta)$. For $(x_1,y_1),(x_2,y_2) \in \HH$ and $(x_1,y_1) \prec (x_2,y_2)$, define the half-space point-to-point partition function
\[
Z^{\mathrm{half}}(x_1,y_1;x_2,y_2)
\;=\;\sum_{\pi \in \Pi_{\HH}[(x_1,y_1)\to(x_2,y_2)]}
\prod_{(i,j)\in\pi\setminus (x_1,y_1)} \omega^\mathrm{half}_{i,j},
\]
and set $Z^{\mathrm{half}}(x_1,y_1;x_2,y_2)=0$ if the sum is empty. The associated half-space free energy is
\[
\log Z^{\mathrm{half}}(x_1,y_1;x_2,y_2) = \log \sum_{\pi \in \Pi_{\HH}[(x_1,y_1)\to(x_2,y_2)]}
\prod_{(i,j)\in\pi\setminus (x_1,y_1)} \omega^\mathrm{half}_{i,j}.
\]
\end{definition}
Throughout this paper, the path weights defining both last passage times and partition functions include the weight at the terminal point but not at the initial point. Some results quoted below are
stated under the convention that the initial weight is also included. The two
versions differ by either an exponential random variable or a log-inverse-gamma random variable. Since both variables have exponential upper and lower tails, all the
moderate-deviation estimates used below transfer between the two
conventions, after adjusting the constants. We therefore use these estimates
under the present convention without further comment.

\subsection{Main results}
\label{subsec:main-results}
Our first two results establish directed landscape convergence for
observation windows translated into the bulk, when the attractive
boundary strength is $o(f(n)n^{-1/3})$. We begin with exponential LPP.

\begin{theorem}\label{thm:lpp-main}
For each $n$, let $\rho_n < 1/2$ and consider the half-space
environment with diagonal weights having distribution
$\operatorname{Exp}(1/2-\rho_n)$ where $\rho_n = o(f(n))n^{-1/3}$. Define the shifted and scaled last passage time in this half-space environment as
\[
\begin{split}
L^{n,f}_{\mathrm{half}}(x,s;y,t)
:=&\,2^{-4/3}n^{-1/3}\Lhalf( ns+2^{5/3}n^{2/3}x
+  f(n)n^{2/3} , ns;\,
 nt+2^{5/3}n^{2/3}y+ f(n)n^{2/3} , nt)\\
&-2^{2/3}n^{2/3}(t-s)-2^{4/3}n^{1/3}(y-x).
\end{split}
\]
Then the continuous linear interpolation of $L^{n,f}_\mathrm{half}$ converges to the directed landscape $\mathcal{L}$ in
distribution in $C(\R_+^4, \R)$.
\end{theorem}

The following theorem establishes the directed landscape convergence for the log-gamma polymer. Let $q,p, \sigma_p$ be the scaling constants defined in \cite[Definition 3.1]{zhang2025convergence} and \cite[Theorem 3.3]{zhang2025convergence} that only depend on $\alpha$. 

\begin{theorem}
\label{thm_half_loggamma_to_DL}
Fix $\alpha>0$ and for each $n$, let $\theta_n < \alpha$ and consider the half-space
environment with diagonal weights having distribution
$\operatorname{\mathrm{Gamma}}^{-1}(\alpha-\theta_n)$ where $\theta_n = o(f(n))n^{-1/3}$. We define the following scaling operators $\overline{z}_n = \lfloor n^{2/3}zq^{-2} \rfloor + 1$, $r_n = \lfloor 2nr \rfloor$. Define the shifted and scaled half-space log-gamma free energy in this half-space environment as
 \[
 \begin{aligned}
     h^{n,f}_{\mathrm{half}}(x,s;y,t) := &\frac{q\sigma_p}{\sqrt{2}n^{1/3}}\bigg[ \log Z^{\mathrm{half}}(\overline{x}_n + s_n+ f(n)n^{2/3}, s_n;\overline{y}_n + t_n + f(n)n^{2/3} -1, t_n-1)\\
     &-p\left( \overline{y}_n - \overline{x}_n + 4nt-4ns\right)\bigg].
 \end{aligned}
 \]
Then the continuous linear interpolation of $h^{n,f}_{\mathrm{half}}$ converges to the directed landscape $\mathcal{L}$ in distribution in $C(\R^4_+, \R)$.
\end{theorem}

The boundary strength assumptions in the preceding theorems are sharp at the
level of scales. The following result is proved in
Section~\ref{sec:sharpness}.

\begin{theorem}\label{thm:sharpness-main}Suppose that $f(n) = O(n^{1/3})$ and $f(n)/n^{1/3} \rightarrow c_f \in [0, \infty)$.

\begin{enumerate}[label=\textup{(\roman*)}]
\item Consider half-space exponential LPP with diagonal weights of rate $1/2 - \rho_n$ where
\[
g(n)n^{-1/3} = \rho_n\leq \rho<\frac12,
\qquad
\frac{g(n)}{f(n)}
\rightarrow
\lambda_{\mathrm{LPP}}\in(0,\infty).
\]
There exists \(T_0>0\) such that, for every fixed \(T\geq T_0\),
\[
L_{\mathrm{half}}^{n,f}(0,0;0,T)
\xlongrightarrow{\Prob}+\infty.
\]

\item Consider the half-space log-gamma polymer with diagonal weights of rate $\alpha - \theta_n$ where
\[
h(n)n^{-1/3}=\theta_n\leq \theta <\alpha,
\qquad
\frac{h(n)}{f(n)}
\rightarrow
\lambda_{\mathrm{LG}}\in(0,\infty).
\]
There exists \(T_0>0\) such that, for every fixed \(T\geq T_0\),
\[
h_{\mathrm{half}}^{n,f}(0,0;0,T)
\xlongrightarrow{\Prob}+\infty.
\]
\end{enumerate}

Consequently, in either model, the four-parameter half-space field
cannot converge to the directed landscape under the full-space
centering and normalization.
\end{theorem}

For exponential LPP, we can go beyond scale-level sharpness and
identify the point-to-point transition at a fixed temporal separation
$T$. The following theorem describes the bulk-dominated and
boundary-dominated regimes and identifies the limiting law in the
critical window; see Figure~\ref{fig:phase-diagram}.
Its proof is given in Section~\ref{sec:9}.
\begin{theorem}\label{thm:phase diagram}
Fix $T>0$ and suppose that one of the following holds:
    \[
     f(n),g(n) = o(n^{1/3}), \qquad \frac{g(n)}{f(n)} \rightarrow \lambda \in \left(0, \infty\right],
    \]
    or
    \[
    \frac{f(n)}{n^{1/3}} \rightarrow c_f  \in (0, T), \qquad \frac{g(n)}{n^{1/3}} \rightarrow c_g \in \left(0,\frac12\right), \qquad \lambda := \frac{c_g}{c_f} \in \left(0,\infty\right).
    \]
    In the first case, set $c_f = c_g=0$. Consider half-space exponential LPP with diagonal weights of rate $1/2-g(n)n^{-1/3}$.

\begin{enumerate}[label=\textup{(\roman*)}]
\item If $\lambda<\frac{1}{2T}$, then
\[
L_{\mathrm{half}}^{n,f}(0,0;0,T)
\xlongrightarrow{d}
\mathcal L(0,0;0,T).
\]

\item If $\lambda>\frac{1}{2T}$, then
\[
L_{\mathrm{half}}^{n,f}(0,0;0,T)
\xlongrightarrow{\Prob}+\infty.
\]
Moreover, setting
\[
m_{\mathrm{half}}^{n,f}(T)
=
\frac{2^{5/3}g(n)(2Tg(n)-f(n))}
{1-4g(n)^2n^{-2/3}},
\]
we have
\[
\frac{L_{\mathrm{half}}^{n,f}(0,0;0,T)
-m_{\mathrm{half}}^{n,f}(T)}{\sqrt{g(n)}}
\xlongrightarrow{d}
\mathcal N \left(0,\,
\frac{4T-(1+4c_g^2)/\lambda}{2^{-1/3}(1-4c_g^2)^2}
\right).
\]

\item If $g(n)=\frac{f(n)}{2T}+\xi f(n)^{-1/2}$ for some
$\xi\in\mathbb R$, then
\[
\frac{L_{\mathrm{half}}^{n,f}(0,0;0,T)}{\sqrt{f(n)}}
\xlongrightarrow{d}
\max\left\{
\frac{2^{5/3}\xi}{1-c_f^2/T^2}
+\frac{2^{1/6}}{\sqrt{1-c_f^2/T^2}}\mathcal N(0,1),\,
0
\right\}.
\]
\end{enumerate}
\end{theorem}

The width of the critical window follows by balancing the shape gap between the bulk part and the boundary-touching part with the boundary fluctuations. Near $\lambda = \frac{1}{2T} $, a perturbation $\Delta g(n)$ changes the shape gap by order $f(n)n^{1/3}\Delta g(n)$, while the boundary fluctuation scale is $\sqrt{f(n)}n^{1/3}$. Balancing these quantities gives $\Delta g(n)\asymp f(n)^{-1/2}$.

\section{Couplings and Reduction to Boundary Avoidance}
\label{sec:couplings-reduction}
The proofs of the first two main results begin with couplings under which the full-space and half-space environments agree strictly inside the half-space. We then seek to compare the corresponding passage times or free energies uniformly over endpoint pairs in a compact observation window. Because $f(n)$ may diverge arbitrarily slowly, a direct union bound over these pairs would generally be insufficient.

Instead, we reduce the uniform comparison in a compact observation window to boundary-avoidance estimates for a single pair of outer endpoints. For exponential LPP, planarity and the ordering of maximizing paths provide a geometric barrier controlling all endpoint pairs in the window. For the log-gamma polymer, no analogous pathwise argument is available:
the partition function
receives contributions from every path that touches the
boundary. Our replacement for geodesic ordering is a stochastic
ordering of polymer measures obtained from Holley's inequality. This bounds the quenched probability of touching the diagonal uniformly over endpoint pairs in the window by a ratio of restricted partition functions for a single pair of endpoints.

\subsection{Coupling}
\label{subsec:lpp-coupling}

On a common probability space, sample two independent families
\begin{equation}
\{\nu^{\mathrm{bulk}}_{i,j}:(i,j)\in\mathbb Z^2\},
\qquad
\{\nu^{\partial}_{i,i}:i\in\mathbb Z\},
\end{equation}
where the bulk weights are i.i.d.\ $\mathrm{Exp}(1)$ random
variables and the boundary weights are i.i.d.\ $\mathrm{Exp}(1/2-\rho_n)$ random
variables. Set
\begin{equation}\label{eq:lpp-coupling}
\nu^{\mathrm{full}}_{i,j}=\nu^{\mathrm{bulk}}_{i,j},
\qquad
\nu^{\mathrm{half}}_{i,j}
=
\begin{cases}
\nu^{\mathrm{bulk}}_{i,j},&i>j,\\
\nu^{\partial}_{i,i},&i=j.
\end{cases}
\end{equation}
Similarly, on a common probability space, sample two
independent families
\begin{equation}
\{\omega^{\mathrm{bulk}}_{i,j}:(i,j)\in\mathbb Z^2\},
\qquad
\{\omega^{\partial}_{i,i}:i\in\mathbb Z\},
\end{equation}
where the bulk weights are i.i.d.\ $\mathrm{Gamma}^{-1}(2\alpha)$ random variables and the boundary weights are i.i.d.\ $\mathrm{Gamma}^{-1}(\alpha-\theta_n)$ random variables. Set
\begin{equation}\label{eq:loggamma-coupling}
\omega^{\mathrm{full}}_{i,j}=\omega^{\mathrm{bulk}}_{i,j},
\qquad
\omega^{\mathrm{half}}_{i,j}
=
\begin{cases}
\omega^{\mathrm{bulk}}_{i,j},&i>j,\\
\omega^{\partial}_{i,i},&i=j.
\end{cases}
\end{equation}

We recall the full-space directed landscape convergence for exponential
LPP and the log-gamma polymer from
\cite[Theorem~12.1]{dauvergne2021scaling} and
\cite[Theorem~1.6]{zhang2025convergence}, respectively. By translation invariance and the full-space convergence theorem, the continuous interpolation of the following translated last passage time and free energy converge to $\mathcal{L}$ in distribution in $C(\R_+^4, \R)$:
\begin{equation}\label{eq:lpp-translated-full}
\begin{aligned}
L^{n,f}_{\mathrm{full}}(x,s;y,t)
:=&\,2^{-4/3}n^{-1/3}\Lfull( ns+2^{5/3}n^{2/3}x
+  f(n)n^{2/3} , ns;\,
 nt+2^{5/3}n^{2/3}y+ f(n)n^{2/3} , nt)\\
&-2^{2/3}n^{2/3}(t-s)-2^{4/3}n^{1/3}(y-x),\\
     h^{n,f}_{\mathrm{full}}(x,s;y,t) := &\frac{q\sigma_p}{\sqrt{2}n^{1/3}}\bigg[ \log Z^{\mathrm{full}}(\overline{x}_n + s_n+ f(n)n^{2/3}, s_n;\overline{y}_n + t_n + f(n)n^{2/3} -1, t_n-1)\\
     &-p\left( \overline{y}_n - \overline{x}_n + 4nt-4ns\right)\bigg].
 \end{aligned}
    \end{equation}
Under the above coupling, Theorem \ref{thm:lpp-main} and Theorem \ref{thm_half_loggamma_to_DL} follow readily from the following proposition and the corresponding full-space convergence.
\begin{proposition}\label{prop:difference}
Fix any $\alpha > 0$. Let $ \rho_n$ and $\theta_n $ be defined as in Theorem \ref{thm:lpp-main} and Theorem \ref{thm_half_loggamma_to_DL}. For $b > 0$, define $Q_b = [-b/2,b/2]^4 \cap \{(x,s;y,t) \in \R_+^4: t-s > b^{-1}\}$. Then, for every $b \in \N$,
\begin{equation}\label{eq:lpp-difference}
\sup_{(x,s;y,t)\in Q_b}
\left|
L^{n,f}_{\mathrm{half}}(x,s;y,t)
-L^{n,f}_{\mathrm{full}}(x,s;y,t)
\right|
\xlongrightarrow{\Prob}0,
\end{equation}
\begin{equation}\label{eq:loggamma-difference}
\sup_{(x,s;y,t)\in Q_b}
\left|
h^{n,f}_{\mathrm{half}}(x,s;y,t)
-h^{n,f}_{\mathrm{full}}(x,s;y,t)
\right|
\xlongrightarrow{\Prob}0.
\end{equation}
\end{proposition}

Fix $b \in \mathbb N$. For the remainder of the proof, it suffices to
assume that
\begin{equation}\label{eq:f-upper-reduction}
f(n)\leq 4bn^{1/3} \text{ for all } n\geq 1.
\end{equation}
Otherwise no up-right path
whose endpoints are indexed by $Q_b$ can intersect the diagonal $\D$.
Consequently, under the couplings above, the full-space and half-space
fields agree identically on $Q_b$, and both assertions of
Proposition~\ref{prop:difference} are immediate.

\subsection{Reduction for exponential LPP}
\label{subsec:lpp-reduction}
We will use the following deterministic monotonicity of geodesics to reduce
the LPP assertion \eqref{eq:lpp-difference} in
Proposition~\ref{prop:difference} to two events involving a single pair of
endpoints. This monotonicity is a standard consequence of planarity; see, for example,
\cite[Lemma~11.2]{basu2014last} and
\cite[Lemma~3.6]{dauvergne2022directed} for closely related formulations. Define
\begin{equation}\label{eq:outer-pair}
(x_1,y_1)
:=
\left(-2bn+\frac12f(n)n^{2/3},-2bn\right),
\qquad
(x_2,y_2)
:=
\left(2bn+\frac12f(n)n^{2/3},2bn\right),
\end{equation}
\begin{equation}\label{eq:common-barrier-region}
R
=
\left\{(x,y)\in\mathbb Z^2:
y>x-\frac34f(n)n^{2/3}\right\}.
\end{equation}
\begin{lemma}\label{lem:lpp-monotonicity}
    Fix $\mathbb O \in \{\Z^2, \HH\}$ and any assignment of non-negative weights $(\nu_{x})_{x \in \mathbb O}$. Let $\Gamma$ be a maximizing path from $(x_1,y_1)$ to $(x_2,y_2)$ constrained to $\mathbb O \cap R$. For any $(u_1,u_2), (v_1,v_2) \in \mathbb O \cap R^c$ where $(u_1, u_2) \prec (v_1,v_2)$ and $y_1 \leq u_2 \leq v_2 \leq y_2$, the rightmost maximizing path in $\mathbb O$ from $(u_1,u_2)$ to $(v_1, v_2)$ cannot cross from the right of $\Gamma$ to its left.    
\end{lemma}

\begin{proof}
Otherwise the two paths have a first and a last crossing. Exchanging the
segments between these crossings and using maximality shows that the two
segments have equal weight. Replacing the corresponding segment of the path
from $(u_1,u_2)$ to $(v_1, v_2)$ by the segment of $\Gamma$ then produces a maximizing path
lying strictly farther to the right, a contradiction.
\end{proof}

For $\mathbb O\in\{\mathbb Z^2,\mathbb H\}$, let
\begin{equation}\label{eq:lpp-barrier-event}
\mathcal A_n(\mathbb O)
=
\left\{
\begin{array}{c}
\text{no maximizing path from $(x_1,y_1)$ to $(x_2,y_2)$}\\
\text{constrained to $\mathbb O\cap R$ is contained in $\mathbb F$}
\end{array}
\right\}.
\end{equation}
Thus, on $\mathcal A_n(\mathbb O)^c$, there exists a maximizing
path $\Gamma^{\mathbb O}$ constrained to $\mathbb O\cap R$ that avoids
$\D$.

For $(x,s;y,t)\in Q_b$, let $(u_1, u_2)$ and $(v_1, v_2)$ be the corresponding lattice
endpoints in the definition of $L^{n,f}_{\mathrm{full}}$ and
$L^{n,f}_{\mathrm{half}}$. Since $f(n)\to\infty$, we know that $(u_1,u_2),(v_1, v_2) \in R^c$ for all sufficiently large $n$. Moreover, since $s,t\in[-b/2,b/2]$, we have $y_1\leq u_2\leq v_2\leq y_2.$ Thus, by Lemma~\ref{lem:lpp-monotonicity}, on
\[
\bigl(\mathcal A_n(\mathbb Z^2)\cup\mathcal A_n(\mathbb H)\bigr)^c,
\]
the rightmost maximizing paths from $(u_1,u_2)$ to $(v_1, v_2)$ in both models are contained in
$\mathbb F$. Since the full-space and half-space environments agree on
$\mathbb F$, it follows that
\[
L^{n,f}_{\mathrm{full}}(x,s;y,t)
=
L^{n,f}_{\mathrm{half}}(x,s;y,t)
\qquad
\text{for all }(x,s;y,t)\in Q_b.
\]
Thus, the LPP assertion \eqref{eq:lpp-difference} in Proposition \ref{prop:difference} follows from the following proposition.
\begin{proposition}\label{prop:lpp-one-pair-goal}
Let $\rho_n$ be defined as in Theorem \ref{thm:lpp-main}. For every $b \in \N$,
\[
\Prob(\mathcal A_n(\mathbb Z^2))
+
\Prob(\mathcal A_n(\mathbb H))
\rightarrow0 \qquad\text{as }n \rightarrow \infty.
\]
\end{proposition}
\subsection{Reduction for the log-gamma polymer}
\label{subsec:loggamma-reduction}

We next reduce the log-gamma assertion \eqref{eq:loggamma-difference} in
Proposition~\ref{prop:difference} to two events involving the same pair of
endpoints $(x_1,y_1)$ and $(x_2,y_2)$. Unlike in LPP, the comparison is not pathwise, since the partition
function receives contributions from all admissible paths. We instead use
Holley's inequality to compare the corresponding polymer measures. The
main input is the following stochastic ordering.

For $\mathbb O \in \{\Z^2, \HH\}$, let
\[
\mathfrak m(\mathbb O)
=
\begin{cases}
\mathrm{full},&\mathbb O=\mathbb Z^2,\\
\mathrm{half},&\mathbb O=\mathbb H.
\end{cases}
\]
For $A\subseteq\mathbb Z^2$ and $p\prec q$, define
\[
Z_A^{\mathfrak m (\mathbb O)}(p;q)
:=
\sum_{\pi\in\Pi_{\mathbb O \cap A}[p\to q]}
\prod_{(i,j) \in \pi \setminus p} \omega^{\mathfrak{m}(\mathbb O)}_{i,j}
, \qquad Z_{A,\mathbb D}^{\mathfrak m (\mathbb O)}(p;q)
:=
\sum_{\substack{
\pi\in\Pi_{\mathbb O \cap A}[p\to q]\\
\pi\cap\mathbb D\neq\varnothing
}}
\prod_{(i,j) \in \pi \setminus p} \omega^{\mathfrak{m}(\mathbb O)}_{i,j}.
\]
We define
\[
\Zexit=Z_{\Z^2, \mathbb D}^{\mathrm{full}},
\qquad
\Zexit_R=Z_{R,\mathbb D}^{\mathrm{full}},
\qquad
\Zb=Z_{\HH, \mathbb D}^{\mathrm{half}},
\qquad
\Zb_R=Z_{R,\mathbb D}^{\mathrm{half}}.
\]
\[
    Q_b(n) = \{(\overline x_n + s_n + f(n)n^{2/3},s_n ; \overline{y}_n +t_n+f(n)n^{2/3}-1 ,t_n-1):(x,s;y,t) \in Q_b \}.
\]
\begin{proposition}\label{prop:monotonicity}
    Fix any assignment of strictly positive weights $(\omega_{x})_{x \in \Z^2}$ and $(\omega_x)_{x \in \HH}$. There exists $n_0 > 0$ that depends on $b$ and $\alpha$ such that for any $n \geq n_0$ and $(x,s;y,t) \in Q_b(n)$, we have
    \[
    \frac{\Zexit(x,s;y,t)}{\Zfull(x,s;y,t)} \leq \frac{\Zexit_R(x_1,y_1;x_2,y_2)}{\Zfull_R(x_1,y_1;x_2,y_2)}, \quad \frac{\Zb(x,s;y,t)}{\Zhalf(x,s;y,t)} \leq \frac{\Zb_R(x_1,y_1;x_2,y_2)}{\Zhalf_R(x_1,y_1;x_2,y_2)}
    \]
    
\end{proposition}

\begin{proof}
Fix
\(\mathbb O\in\{\mathbb Z^2,\mathbb H\}\). For $z=(z_1,z_2)\in\mathbb O$, define
\[
\tau(z)=z_1+z_2,
\qquad
\Delta(z)=z_1-z_2.
\]
Each up-right step increases $\tau$ by one. Hence an up-right path
$\pi$ that starts at $p$ and ends at $q$ intersects every anti-diagonal $\{z:\tau(z)=k\}$ where $k\in\llbracket\tau(p),\tau(q)\rrbracket$ at exactly one vertex. We denote the transverse coordinate of this
vertex by $\Delta_\pi(k)$.

For $(x,s;y,t)\in Q_b(n)$, let
\[
I_L=\llbracket\tau(x_1,y_1),\tau(x_2,y_2)\rrbracket,
\qquad
I_S=\llbracket\tau(x,s),\tau(y,t)\rrbracket.
\] 
If $f(n)n^{2/3}>3bn$, no path with endpoints in $Q_b(n)$ can intersect
$\D$, and the comparison is immediate. We may therefore assume
that $f(n)n^{2/3}\leq3bn$. Hence, for all sufficiently large $n$, we have $I_S \subseteq I_L$.

We identify every up-right path with its transverse-coordinate
profile. For $\Gamma \in \Pi_{\mathbb O\cap R}[(x_1,y_1)\to (x_2,y_2)]$, we identify it as $\widehat{\Gamma}: I_L \rightarrow \Z \cup \{\infty\}$ such that $\widehat{\Gamma}(k) = \Delta_\Gamma(k)$. For $\gamma \in \Pi_{\mathbb O}[(x,s) \to (y,t)]$,  we identify it as $\widehat{\gamma}: I_L \rightarrow \Z \cup \{\infty\}$ such that 
\[
\widehat{\gamma}(k)
=
\begin{cases}
\Delta_\gamma(k),&k\in I_S,\\
+\infty,&k\in I_L\setminus I_S.
\end{cases}
\]

Define
\[
\mathbb L
=
\left\{
\widehat\Gamma:
\Gamma\in\Pi_{\mathbb O\cap R}[(x_1,y_1)\to (x_2,y_2)]
\right\}
\cup
\left\{
\widehat\gamma:
\gamma\in\Pi_{\mathbb O}[(x,s)\to (y,t)]
\right\}.
\]
We 
equip $\mathbb L$ with the pointwise order. For
$\eta,\zeta\in\mathbb L$, define
\[
(\eta\wedge\zeta)(k)
=
\min\{\eta(k),\zeta(k)\},
\qquad
(\eta\vee\zeta)(k)
=
\max\{\eta(k),\zeta(k)\},
\qquad k\in I_L.
\]
By the definitions of $Q_b(n)$ and $R$, it is not difficult to check that for any $\Gamma\in\Pi_{\mathbb O\cap R}[(x_1,y_1)\to (x_2,y_2)]$, $\widehat{\Gamma}(\tau(x,s)) < \Delta(x,s)$ and $\widehat{\Gamma}(\tau(y,t)) < \Delta(y,t)$ for all $n$ sufficiently large. Hence $\mathbb L$ is a distributive lattice. We define the following polymer measures on their respective supports and set them equal to zero elsewhere
\[
\mathbf P_L^{\mathbb O}(\widehat{\Gamma})
:=
\frac{\prod_{(i,j) \in \Gamma \setminus (x_1,y_1)}\omega^{\mathfrak m(\mathbb O)}_{i,j}}
{Z_R^{\mathfrak m(\mathbb O)}(x_1,y_1;x_2,y_2)},
\qquad
\Gamma\in
\Pi_{\mathbb O\cap R}[(x_1,y_1)\to(x_2,y_2)],\]
\[
\mathbf P_S^{\mathbb O}(\widehat{\gamma})
:=
\frac{\prod_{(i,j) \in \gamma \setminus (x,s)}\omega_{i,j}^{\mathfrak m(\mathbb O)}}
{Z^{\mathfrak m(\mathbb O)}(x,s;y,t)},
\qquad
\gamma\in
\Pi_{\mathbb O}[(x,s)\to(y,t)].
\]
Whenever $\eta=\widehat\Gamma$ and $\zeta=\widehat\gamma$ for $\Gamma\in
\Pi_{\mathbb O\cap R}[(x_1,y_1)\to(x_2,y_2)]$ and $\gamma\in
\Pi_{\mathbb O}[(x,s)\to(y,t)]$, we have
\[
\mathbf P_L^{\mathbb O}(\eta\wedge\zeta)
\mathbf P_S^{\mathbb O}(\eta\vee\zeta)
=
\mathbf P_L^{\mathbb O}(\eta)
\mathbf P_S^{\mathbb O}(\zeta).
\]
In all other cases, the right-hand side vanishes, and hence the
inequality is trivial. Therefore, for all $\eta,\zeta\in\mathbb L$,
\[
\mathbf P_L^{\mathbb O}(\eta\wedge\zeta)
\mathbf P_S^{\mathbb O}(\eta\vee\zeta)
\ge
\mathbf P_L^{\mathbb O}(\eta)
\mathbf P_S^{\mathbb O}(\zeta).
\]
Therefore, by Holley's inequality, since the function $\mathbbm 1
\left\{
\min_{k\in I_L}\eta(k)\le0
\right\}$ is decreasing with respect to the pointwise order on $\mathbb L$, we have

\[
\frac{\Zexit(x,s;y,t)}{\Zfull(x,s;y,t)}
\le
\frac{\Zexit_R(x_1,y_1;x_2,y_2)}
{\Zfull_R(x_1,y_1;x_2,y_2)}, \qquad 
\frac{\Zb(x,s;y,t)}{\Zhalf(x,s;y,t)}
\le
\frac{\Zb_R(x_1,y_1;x_2,y_2)}
{\Zhalf_R(x_1,y_1;x_2,y_2)}.
\]
\end{proof}

For $p\prec q$ in $\mathbb H$, define the common
contributions from paths avoiding $\D$ by
\begin{equation}
\Zin(p;q)
=
\sum_{\pi\in\Pi_{\mathbb F}[p\to q]}
\prod_{z\in\pi\setminus p}\omega_z^{\mathrm{bulk}},
\qquad
\Zin_R(p;q)
=
\sum_{\pi\in\Pi_{\mathbb F\cap R}[p\to q]}
\prod_{z\in\pi\setminus p}\omega_z^{\mathrm{bulk}}.
\end{equation}
Under the coupling \eqref{eq:loggamma-coupling},
\begin{equation}\label{eq:partition-function-decomposition}
\Zfull(p;q)
=
\Zin(p;q)+\Zexit(p;q),
\qquad
\Zhalf(p;q)
=
\Zin(p;q)+\Zb(p;q),
\end{equation}
\begin{equation}\label{eq:restricted-partition-function-decomposition}
\begin{aligned}
\Zfull_R(x_1,y_1;x_2,y_2)
&=
\Zin_R(x_1,y_1;x_2,y_2)
+
\Zexit_R(x_1,y_1;x_2,y_2),\\
\Zhalf_R(x_1,y_1;x_2,y_2)
&=
\Zin_R(x_1,y_1;x_2,y_2)
+
\Zb_R(x_1,y_1;x_2,y_2).
\end{aligned}
\end{equation}

For $\varepsilon>0$, define
\begin{align}
\mathcal B_{n,\varepsilon}(\mathbb Z^2)
&=
\left\{
\log\left(
1+
\frac{\Zexit_R(x_1,y_1;x_2,y_2)}
{\Zin_R(x_1,y_1;x_2,y_2)}
\right)
\geq
\frac{\varepsilon n^{1/3}}{\sqrt{2}q\sigma_p}
\right\},
\label{eq:full-loggamma-bad-event}\\
\mathcal B_{n,\varepsilon}(\mathbb H)
&=
\left\{
\log\left(
1+
\frac{\Zb_R(x_1,y_1;x_2,y_2)}
{\Zin_R(x_1,y_1;x_2,y_2)}
\right)
\geq
\frac{\varepsilon n^{1/3}}{\sqrt{2}q\sigma_p}
\right\}.
\label{eq:half-loggamma-bad-event}
\end{align}

For $(x,s;y,t)\in Q_b$, let $p$ and $q$ be the corresponding lattice
endpoints in the definitions of $h^{n,f}_{\mathrm{full}}$ and
$h^{n,f}_{\mathrm{half}}$. Since the deterministic centerings agree,
\begin{equation}
\begin{aligned}
\left|
\log\Zfull(p;q)-\log\Zhalf(p;q)
\right|
&\leq
\log\left(
1+\frac{\Zexit(p;q)}{\Zin(p;q)}
\right)+
\log\left(
1+\frac{\Zb(p;q)}{\Zin(p;q)}
\right).
\end{aligned}
\end{equation}
Hence,
Proposition~\ref{prop:monotonicity} implies that
\begin{equation}
\log\left(
1+\frac{\Zexit(p;q)}{\Zin(p;q)}
\right)
\leq
\log\left(
1+
\frac{\Zexit_R(x_1,y_1;x_2,y_2)}
{\Zin_R(x_1,y_1;x_2,y_2)}
\right),
\end{equation}
\begin{equation}
\log\left(
1+\frac{\Zb(p;q)}{\Zin(p;q)}
\right)
\leq
\log\left(
1+
\frac{\Zb_R(x_1,y_1;x_2,y_2)}
{\Zin_R(x_1,y_1;x_2,y_2)}
\right).
\end{equation}
Consequently,
\begin{equation}\label{eq:loggamma-one-pair-reduction}
\begin{aligned}
&\Prob\left(
\sup_{(x,s;y,t)\in Q_b}
\left|
h^{n,f}_{\mathrm{full}}(x,s;y,t)
-
h^{n,f}_{\mathrm{half}}(x,s;y,t)
\right|
\geq\varepsilon
\right)\leq
\Prob\bigl(\mathcal B_{n,\varepsilon}(\mathbb Z^2)\bigr)
+
\Prob\bigl(\mathcal B_{n,\varepsilon}(\mathbb H)\bigr).
\end{aligned}
\end{equation}
Thus, the log-gamma assertion \eqref{eq:loggamma-difference} in
Proposition~\ref{prop:difference} follows from the following proposition.
\begin{proposition}\label{prop:loggamma-one-pair-goal}
Fix any $\alpha >0$ and $\varepsilon >0$. Let $\theta_n$ be defined as in Theorem \ref{thm_half_loggamma_to_DL}. For every $b \in \N$, 
\[
\Prob\bigl(\mathcal B_{n,\varepsilon}(\mathbb Z^2)\bigr)
+
\Prob\bigl(\mathcal B_{n,\varepsilon}(\mathbb H)\bigr)
\rightarrow0
\qquad\text{as }n\to\infty.
\]
\end{proposition}

\section{Proof of Proposition~\ref{prop:lpp-one-pair-goal}}\label{sec:4}

In this section, we prove Proposition~\ref{prop:lpp-one-pair-goal}. We decompose each boundary-touching path at its first and last visits to the diagonal into two bulk legs and a middle half-space segment. The deterministic shape inequality quantifies the cost of reaching the boundary, while full-space and half-space tail estimates control the three pieces uniformly over the contact locations. Together, these estimates show that the shape cost outweighs the possible boundary gain and random fluctuations.

\subsection{Shape function analysis}
The full-space exponential LPP shape function is defined, for
$(p_1,p_2)\prec(q_1,q_2)$, by
\begin{equation}\label{eq:lpp-shape}
g(p_1,p_2;q_1,q_2)
:=
\left(\sqrt{q_1 - p_1}+\sqrt{q_2 - p_2}\right)^2.
\end{equation}
It describes the leading-order last passage time between the two
endpoints. In the non-attractive half-space regime, the shape along the
diagonal agrees with the full-space shape. The following lemma quantifies the shape loss caused by forcing a path
from $(x_1,y_1)$ to $(x_2,y_2)$ to visit the diagonal. Recall that
\begin{equation}
(x_1,y_1)
=
\left(-2bn+\frac12f(n)n^{2/3},-2bn\right),
\qquad
(x_2,y_2)
=
\left(2bn+\frac12f(n)n^{2/3},2bn\right).
\end{equation}

\begin{lemma}\label{lem:lpp-shape-function} Let $I=\llbracket x_1,y_2\rrbracket.$
There exist $c_0,n_0>0$, depending only on $b$, such that, for all
$n\geq n_0$ and $i,j\in I$ satisfying $i\leq j$,
\begin{equation}\label{eq:lpp-shape-gap}
\begin{aligned}
&g(x_1,y_1;i + n^{1/3},i + n^{1/3})
+4(j-i)
+g(j-n^{1/3},j-n^{1/3};x_2,y_2)
\leq
16bn-5c_0f(n)^2n^{1/3}.
\end{aligned}
\end{equation}
\end{lemma}

\begin{proof}
One can easily check that for $A > |B|$, 
\[
    \sqrt{A^2 - B^2} \leq A - \frac{B^2}{2A}.
    \]
For $i\in I$, we have
\begin{align*}
g(x_1,y_1;i+n^{1/3},i+n^{1/3})
&=
2i+2n^{1/3}-x_1-y_1+
\sqrt{
(2i+2n^{1/3}-x_1-y_1)^2-(x_1-y_1)^2
}\\
&\leq 2(2i+2n^{1/3}-x_1-y_1) - \frac{n^{4/3}f(n)^2}{8(2i+2n^{1/3}-x_1-y_1)}.
\end{align*}
For all sufficiently large $n$,
\[
2i+2n^{1/3}-x_1-y_1
=
2i+2n^{1/3}+4bn-\frac12 f(n)n^{2/3}
\leq 8bn.
\]
By absorbing the $n^{1/3}$ term into the constant, it follows that
\[
g(x_1,y_1;i+n^{1/3},i+n^{1/3})
\leq
2(2i-x_1-y_1)
-cf(n)^2n^{1/3}.
\]
Similarly, 
\[
g(j-n^{1/3},j-n^{1/3};x_2,y_2)
\leq
2(x_2+y_2-2j)
-cf(n)^2n^{1/3}.
\]
\end{proof}

\subsection{Last passage time estimates}

Before we collect the last passage time estimates needed below, let us first introduce some notations.
\begin{definition}\label{def:parallelogram}
    For $n>0$, $m>0$, and $\ell>0$, let $U_{n,m,\ell}$ be the
parallelogram with vertices
\begin{equation}
(0,-\ell n^{2/3}),
\qquad
(0,\ell n^{2/3}),
\qquad
(n,mn-\ell n^{2/3}),
\qquad
(n,mn+\ell n^{2/3}).
\end{equation}
We identify $U_{n,m,\ell}$ with its intersection with $\mathbb Z^2$. For a parallelogram $U$ and $\psi >1$, define
\begin{equation}
\mathcal S_\psi(U)
=
\left\{
(p,q)\in (U\cap\mathbb Z^2)^2:
p\prec q,\ 
\frac{q_2-p_2}{q_1-p_1}\in[\psi^{-1},\psi]
\right\}.
\end{equation}
\end{definition}

The following two estimates control constrained and exiting paths in the
full-space model. Let $p \in \Z^2$ and $q = p + (n,n)$. Let
$L^{\mathrm{in},\,\ell n^{2/3}}(p;q)$ and
$L^{\mathrm{exit},\,\ell n^{2/3}}(p;q)$ denote the full-space
last passage times over paths from $p$ to $q$ that, respectively,
remain entirely inside the parallelogram $p+U_{n,1,\ell}$ and exit it through a side parallel to the
diagonal. 

\begin{proposition}[{\cite[Proposition~3.7]{basu2022interlacing}}]
\label{prop:lpp-constrained-moderate}
Fix $\ell>0$. There exist $c,\theta_0,n_0>0$ such that, for all
$n\geq n_0$ and $\theta\geq\theta_0$,
\begin{equation}
\Prob\left(
L^{\mathrm{in},\,\ell n^{2/3}}(0,0;n,n)-4n
\leq-\theta n^{1/3}
\right)
\leq e^{-c\ell\theta}.
\end{equation}
\end{proposition}

\begin{proposition}[{\cite[Proposition~4.7]{basu2019temporal}}]
\label{prop:lpp-exit}
There exist $c_1,c_2>0$ such that, for all sufficiently large $\ell$
and all $n\geq1$,
\begin{equation}
\Prob\left(
L^{\mathrm{exit},\,\ell n^{2/3}}(0,0;n,n)-4n
\geq-c_1\ell^2n^{1/3}
\right)
\leq e^{-c_2\ell^3}.
\end{equation}
\end{proposition}

Moreover, we need a uniform upper-tail estimate for normalized full-space last passage times over a family of admissible endpoint pairs in a parallelogram. We record \cite[Proposition A.3]{Basu2021}, which is based on \cite[Proposition 10.5]{basu2014last} for Poissonian last passage percolation. 

\begin{proposition}\label{prop:tree_sup}
Let $U = U_{n,1,1}.$ There exist positive constants $C,c,n_0$ such that for all $n \geq n_0$ and $x>0$,
\begin{equation}
\Prob\left(
\sup_{(u,v)\in \mathcal S_2(U)}
\Bigl(
\Lfull(u;v)-\E[\Lfull(u;v)]
\Bigr)
\geq x n^{1/3}
\right)
\leq C e^{-cx}.
\end{equation}
\end{proposition}

Because of the superadditivity structure of last passage times, translation invariance, and the \(L^1\) shape
theorem \cite[Proposition~2.1(i)]{Martin2004}, it follows that for any $u=(u_1,u_2), v = (v_1,v_2) \in \Z^2$ where $u \prec v$,
\[
\E[\Lfull(u;v)]
\leq g(u;v)
=
\left(\sqrt{v_1-u_1}+\sqrt{v_2-u_2}\right)^2.
\]
As a result, we can apply Proposition \ref{prop:tree_sup} with $\E[\Lfull(u;v)]$ replaced by $g(u;v)$.

We also need an upper-tail estimate for full-space last passage times between points with steep slope, which are not covered by the preceding uniform estimate. We derive this estimate from the connection between exponential last passage percolation and the Laguerre unitary ensemble. By \cite[Proposition 1.4]{Johansson_2000}, we know that for $m \geq \ell$, the exponential last passage time from $(1,1)$ to $(m,\ell)$ in the full space is equal in distribution with the largest eigenvalue in the Laguerre ensemble:
\begin{equation}\label{eq:LUE}
    \Lfull(1,1;m,\ell)\stackrel{d}{=}\lambda_{\max}(AA^*),
\end{equation}
where \(A\) is an \(\ell\times m\) matrix with independent standard complex Gaussian
entries. 

We will use the following result \cite[Theorem 2]{Small_Deviation} to control the the LPP value between points whose slope is steep:

\begin{theorem}\label{thm:small_deviation}
    Let \(A\) be an \(\ell\times m\) matrix with independent standard complex Gaussian
entries and $\lambda_{\max}(AA^*)$ denote the largest eigenvalue of $AA^*$. There exists some $C,c > 0$ such that for all $m \geq \ell \geq 1$, $0 < \varepsilon \leq 1$, 
\[
\Prob\left( \lambda_{\max}(AA^*) - (\sqrt{m} +\sqrt{\ell})^2 \geq \varepsilon(\sqrt{m} +\sqrt{\ell})^2\right) \leq Ce^{-c\sqrt{\ell m}\varepsilon^{3/2} \min\{1/\sqrt\varepsilon, (m/\ell)^{1/4}\}}
\]
\end{theorem}

\begin{corollary}\label{cor:cheap tail}
    
There exists some $C,c > 0$ such that for all $m \geq \ell \geq 1$ and $0<x \leq 4m^{1/2}\ell^{1/6}$,
    \[
\Prob\left(\Lfull(1,1;m,\ell)-(\sqrt{m} + \sqrt{\ell})^2 \geq xm^{1/2}\ell^{-1/6}\right) \leq  Ce^{-cx}
\] 
\end{corollary}

\begin{proof}
    We obtain this by using \eqref{eq:LUE} and setting $4\varepsilon = xm^{-1/2}\ell^{-1/6}$ in Theorem \ref{thm:small_deviation} as
    \[
    \Prob\left( \lambda_{\max}(AA^*) - (\sqrt{m} +\sqrt{\ell})^2 \geq 4\varepsilon m \right) \leq \Prob\left( \lambda_{\max}(AA^*) - (\sqrt{m} +\sqrt{\ell})^2 \geq \varepsilon(\sqrt{m} +\sqrt{\ell})^2\right).
    \]
\end{proof}

Lastly, we need the following uniform half-space control along the
diagonal. For nonnegative boundary parameter this follows from \cite[Lemma~4.11]{schmid2023mixing}. The proof reduces the uniform estimate to a one-point half-space bound and an FKG concatenation argument. The same FKG concatenation argument applies under our assumptions of $\rho_n$ in Theorem \ref{thm:lpp-main}. We only need to adapt the one-point half-space upper bound. For $\rho_n = g(n)n^{-1/3}$ where $g(n)>0$ and $g(n) = o(f(n)),$ 
comparison with stationary exponential LPP at its characteristic
endpoint and \cite[Theorem~2.5]{bhatia2020moderate} gives the required
upper-tail estimate. Under \eqref{eq:f-upper-reduction},
$g(n)=o(f(n))$ implies that $\rho_n\to0$. Hence the corresponding
stationary densities remain in a fixed compact subset of $(0,1)$.
Inspecting the proof of \cite[Theorem~2.5]{bhatia2020moderate}, its
constants can therefore be chosen uniformly in $n$. 

Since the characteristic endpoint differs from
the diagonal by $O(g(n)n^{2/3})$, the corresponding change in centering
is $O(g(n)^2n^{1/3})=o(f(n)^2n^{1/3})$. Moreover, a standard
random-walk estimate shows that the fluctuation scale of the stationary
correction is $g(n)^{1/2}n^{1/3},$
which is negligible compared with $f(n)^2n^{1/3}$. A fully detailed positive-temperature analogue is given in the proof of Proposition~\ref{prop:hslg_moderate}.

\begin{proposition}\label{prop:lpp-hs-moderate}
Let $\rho_n$ be defined in Theorem \ref{thm:lpp-main}. For every fixed $A > 0$ and $x>0$,
\begin{equation}
\Prob\left(
\max_{\substack{i,j\in\llbracket0,An\rrbracket\\i<j}}
\left\{
\Lhalf(i,i;j,j)-4(j-i)
\right\}
\geq xf(n)^2n^{1/3}
\right) \rightarrow 0 \qquad \text{ as } n \rightarrow \infty. 
\end{equation}
\end{proposition}

\subsection{Proof of Proposition~\ref{prop:lpp-one-pair-goal}}

\begin{proof}[Proof of Proposition~\ref{prop:lpp-one-pair-goal}]
For all sufficiently large $n$, the parallelogram of width
$n^{2/3}$ is contained in $\mathbb F\cap R$, while every path touching
$\D$ exits the parallelogram of width
$\frac14f(n)n^{2/3}$. Consequently,
\begin{align}
\Prob\bigl(\mathcal A_n(\mathbb Z^2)\bigr)
&\leq
\Prob\left(
L^{\mathrm{in},\,n^{2/3}}(x_1,y_1;x_2,y_2)
\leq
L^{\mathrm{exit},\,\frac14f(n)n^{2/3}}
(x_1,y_1;x_2,y_2)
\right)\notag\\
&\leq
\Prob\left(
L^{\mathrm{in},\,n^{2/3}}(x_1,y_1;x_2,y_2)
-16bn
\leq
-\kappa f(n)^2n^{1/3}
\right)\\
&\quad +
\Prob\left(
L^{\mathrm{exit},\,\frac14f(n)n^{2/3}}
(x_1,y_1;x_2,y_2)
-16bn
\geq
-\kappa f(n)^2n^{1/3}
\right)\notag\\
&\leq
Ce^{-cf(n)^2}.
\end{align}
Here $\kappa>0$ is chosen so that we can apply
Propositions~\ref{prop:lpp-constrained-moderate} and
\ref{prop:lpp-exit} with $4bn$ in place of $n$ to the final inequality.

On the other hand, let us consider the half space event $\mathcal A_n(\HH)$. Let $I = \llbracket x_1,y_2 \rrbracket$, the probability of $\mathcal{A}_n(\HH)$ is bounded by the following probability

\begin{align}
\Prob\left(L^{\mathrm{in},\,n^{2/3}}(x_1,y_1;x_2,y_2) \leq
\max_{\substack{i,j\in I\\i\leq j}}
\Bigl\{
\Lfull(x_1,y_1;i,i)
+\nu^\partial_{i,i}
+\Lhalf(i,i;j,j)
+\Lfull(j,j;x_2,y_2)
\Bigr\} \right).
\label{eq:lpp-half-decomposition}
\end{align}
By positivity of the weights,
\[
\Lfull(x_1,y_1;i,i)
\leq
\Lfull(x_1,y_1;i+n^{1/3},i+n^{1/3}), \quad 
\Lfull(j,j;x_2,y_2)
\leq
\Lfull(j-n^{1/3},j-n^{1/3};x_2,y_2).
\]
Therefore, let $c_0$ be the constant from Lemma \ref{lem:lpp-shape-function}, $\Prob(\mathcal A_n(\HH))$ is bounded by
\begin{equation}\label{eq:lpp-half-first-split}
\begin{aligned}
&\Prob\left(
L^{\mathrm{in},\,n^{2/3}}(x_1,y_1;x_2,y_2)-16bn
\leq
-c_0f(n)^2n^{1/3}
\right)\\
&+
\Prob\Bigg(
\max_{\substack{i,j\in I\\i\leq j}}
\Bigl\{
\Lfull(x_1,y_1;i+n^{1/3},i+n^{1/3})
+\nu^\partial_{i,i}
+\Lhalf(i,i;j,j) 
+\Lfull(j-n^{1/3},j-n^{1/3};x_2,y_2)
\Bigr\}\\
&\qquad \qquad
-16bn
\geq
-c_0f(n)^2n^{1/3}
\Bigg).
\end{aligned}
\end{equation}
The first probability is bounded by $Ce^{-cf(n)^2}$ by
Proposition~\ref{prop:lpp-constrained-moderate}, applied with $4bn$ in
place of $n$. By Lemma~\ref{lem:lpp-shape-function}, the second probability in
\eqref{eq:lpp-half-first-split} is at most
\begin{align}
&\Prob\left(
\max_{i\in I}
\left\{
\Lfull(x_1,y_1;i+n^{1/3},i+n^{1/3})
-g(x_1,y_1;i+n^{1/3},i+n^{1/3})
\right\}
\geq
c_0f(n)^2n^{1/3}
\right)\notag\\
&\quad+
\Prob\left(
\max_{j\in I}
\left\{
\Lfull(j-n^{1/3},j-n^{1/3};x_2,y_2)
-g(j-n^{1/3},j-n^{1/3};x_2,y_2)
\right\}
\geq
c_0f(n)^2n^{1/3}
\right)\notag\\
&\quad+
\Prob\left(
\max_{\substack{i,j\in I\\i<j}}
\left\{
\Lhalf(i,i;j,j)-4(j-i)
\right\}
\geq
c_0f(n)^2n^{1/3}
\right)+
\Prob\left(
\max_{i\in I}\nu^\partial_{i,i}
\geq
c_0f(n)^2n^{1/3}
\right).
\label{eq:lpp-four-terms}
\end{align}
By Proposition~\ref{prop:lpp-hs-moderate}, the third probability is $o(1)$. Moreover, since $\nu^\partial_{i,i}\sim
\operatorname{Exp}(1/2-\rho_n)$ where $1/2-\rho_n$ is uniformly bounded away from zero for all sufficiently
large $n$ and $|I|\leq Cn$,
\[
\Prob\left(
\max_{i\in I}\nu^\partial_{i,i}
\geq
c_0f(n)^2n^{1/3}
\right)
\leq
Cn e^{-cf(n)^2n^{1/3}}.
\]
It remains to control the two full-space terms. Their estimates are
identical by symmetry, so we treat only the first one. Let $J
=\llbracket x_1,x_1+f(n)n^{2/3}\rrbracket \subseteq I$. We split
\begin{align}\label{eq:splittt}
&\Prob\left(
\max_{i\in I}
\left\{
\Lfull(x_1,y_1;i+n^{1/3},i+n^{1/3})
-g(x_1,y_1;i+n^{1/3},i+n^{1/3})
\right\}
\geq
c_0f(n)^2n^{1/3}
\right)\notag\\
&\quad\leq
\sum_{i\in J}
\Prob\left(
\Lfull(x_1,y_1;i+n^{1/3},i+n^{1/3})
-g(x_1,y_1;i+n^{1/3},i+n^{1/3})
\geq
c_0f(n)^2n^{1/3}
\right)\notag\\
&\qquad+
\Prob\left(
\max_{i\in I\setminus J}
\left\{
\Lfull(x_1,y_1;i+n^{1/3},i+n^{1/3})
-g(x_1,y_1;i+n^{1/3},i+n^{1/3})
\right\}
\geq
c_0f(n)^2n^{1/3}
\right).
\end{align}
For $i\in J$, the corresponding endpoint pairs may have steep slope. Corollary~\ref{cor:cheap tail} gives
\begin{align}
&\sum_{i\in J}
\Prob\left(
\Lfull(x_1,y_1;i+n^{1/3},i+n^{1/3})
-g(x_1,y_1;i+n^{1/3},i+n^{1/3})
\geq
c_0f(n)^2n^{1/3}
\right)\notag\\
&\quad\leq
C\sum_{k=1}^{f(n)n^{2/3}+1}
\exp\left(
-c\min\left\{
f(n)^{3/2}k^{1/6},
f(n)^{1/2}n^{1/3}
\right\}
\right)\notag\\
&\quad\leq
Cn^{2/3}f(n)e^{-cf(n)^{1/2}n^{1/3}}
+
C\sum_{k=1}^{\infty}
e^{-cf(n)^{3/2}k^{1/6}}\notag\\
&\quad\leq
Cn^{2/3}f(n)e^{-cf(n)^{1/2}n^{1/3}}
+
Cf(n)^{-9}.
\label{eq:lpp-steep-final}
\end{align}
For $i\in I\setminus J$, we apply Proposition~\ref{prop:tree_sup} with longitudinal scale
$Cf(n)^{3/2}n$ and bound the second term in \eqref{eq:splittt} by $Ce^{-cf(n)^{3/2}}$. Combining the estimates above yields Proposition~\ref{prop:lpp-one-pair-goal}.
\end{proof}

\section{Proof of Proposition~\ref{prop:loggamma-one-pair-goal}}\label{sec:5}

In this section, we prove Proposition~\ref{prop:loggamma-one-pair-goal}. As in the LPP case, the proof combines a shape inequality with tail estimates, several of which we establish in this paper for the log-gamma polymer.

\subsection{Shape function analysis for log-gamma polymer}
Let $\Psi_0,\Psi_1,\Psi_2$ denote the polygamma functions, defined by
\[
\Psi_0(x) = \frac{d}{dx}\log \Gamma(x), \qquad
\Psi_1(x) = \frac{d}{dx}\Psi_0(x), \qquad
\Psi_2(x) = \frac{d}{dx}\Psi_1(x).
\]
Define the function $g: (0,2\alpha) \rightarrow (0,\infty)$ as
\[
g(x) = \frac{\Psi_1(2\alpha - x)}{\Psi_1(x)}.
\]
It is not difficult to check that $g$ is a smooth, increasing bijection from $(0,2\alpha)$ to $(0, \infty)$. Let $g^{-1}$ denote its inverse function, which is a smooth increasing bijection from $(0, \infty)$ to $(0,2\alpha)$. Now for all $x \in (0, \infty)$, we define the function
\[
\phi(x) = x\Psi_0(g^{-1}(x)) + \Psi_0(2\alpha - g^{-1}(x)).
\]

\begin{definition}
    We define the full-space log-gamma polymer shape function $d: \R_{>0} \times \R_{>0} \rightarrow \R$ by $d(x,y) = -y\phi(x/y)$. 
\end{definition}

\begin{remark}
For $x, y \in \R_{>0}$, it is not difficult to check that $g^{-1}(x/y) = 2\alpha-g^{-1}(y/x)$ and hence $d(x,y) = d(y,x)$.
\end{remark}

Now we can state the inequality for the shape function, which will be proven in Section \ref{sec:shape}.
\begin{proposition}\label{prop:shape_function}
Fix $\alpha > 0, b \geq 1$, and $\xi \in (0,1)$. Let $I = \llbracket x_1 + \xi f(n)n^{2/3}, y_2 - \xi f(n)n^{2/3} \rrbracket$. There exist constants $D>0$ and $n_0\ge1$ such that for all $n\ge n_0$, 
\begin{equation}\label{eq:key_bound}
-8bn\Psi_0(\alpha)
\;\ge\;
5D f(n)^2n^{1/3}
+\max_{\substack{i, j \in I \\ i \leq j}}
\Bigl\{
d(i-x_1,\, i-y_1)
- 2(j-i)\Psi_0(\alpha)
+ d(x_2-j,\, y_2-j)
\Bigr\}.
\end{equation}
\end{proposition}

\subsection{Tail estimates for log-gamma polymer}

We next collect the log-gamma polymer estimates needed below. These are
the positive-temperature analogues of the LPP estimates used in the
preceding section: they control free energies over paths that
remain inside or exit a parallelogram, uniform point-to-point free
energies, and the complementary steep-slope regime.

Let $p \in \Z^2$ and $q = p + (n,n)$. Let $Z^{\mathrm{in},\,\ell n^{2/3}}(p;q)$ and
$Z^{\mathrm{exit},\,\ell n^{2/3}}(p;q)$ denote the full-space
partition functions over paths from $p$ to $q$ that, respectively,
remain entirely inside the parallelogram $p+U_{n,1,\ell}$ and exit it through a side parallel to the
diagonal. We can now state versions of
\cite[Corollary~3.14]{basu2024temporal} and
\cite[Theorem~3.16]{basu2024temporal} adapted to our setting; see
\cite{basu2024temporal} for the full generality.
\begin{proposition}\label{prop:exit}
    There exist constants $C, c, n_0 > 0$ such that for any $n \geq n_0$, $t \geq 1$, we have
    \begin{equation}
        \Prob \left( \log Z^{\mathrm{exit}, tn^{2/3}}\left(0, 0; n,n\right) + 2n\Psi_0(\alpha) \geq -Ct^2n^{1/3} \right)\leq  e^{-c t^3}.
    \end{equation}
\end{proposition}
\begin{proposition}\label{prop:in}
    There exist constants $c, t_0, n_0 > 0$ such that for any $n \geq n_0, t \geq t_0$, we have
    \begin{equation}
        \Prob \left(\log Z^{\mathrm{in}, n^{2/3}}(0,0;n,n)+ 2n\Psi_0(\alpha) \leq -tn^{1/3} \right) \leq \sqrt{t}e^{-ct}.
    \end{equation}
\end{proposition}

We next state three more tail estimates needed to control the free energy in the decomposition of the boundary-touching paths, whose proofs
are postponed to Section \ref{sec: tail}. The first two are log-gamma analogues of the zero-temperature estimates
\cite[Propositions~10.5]{basu2014last} and \cite[Lemma 4.11]{schmid2023mixing}, while the third treats the
complementary steep-slope regime.

\begin{proposition}\label{prop:loggamma-uniform-upper}
Fix $\psi > 1$ and let $U=U_{n,1,1}$. There exist constants $C, n_0,c>0$, depending only on
$\alpha$ and $\psi$, such that, for every $n\ge n_0$ and $x\ge0$,
\begin{equation}
\mathbb P\left(
\max_{(p,q)\in\mathcal S_\psi(U)}
\left\{
\log\Zfull(p;q)-d(q_1-p_1,q_2-p_2)
\right\}
\ge xn^{1/3}
\right)
\le Ce^{-cx}.
\end{equation}
\end{proposition}

\begin{proposition}
\label{prop:hslg_moderate}
Fix $\alpha > 0$. For $\theta_n$ defined in Theorem \ref{thm_half_loggamma_to_DL}, fix any $A > 0$ and $x>0$, 
\begin{equation}
\Prob\left(
\max_{\substack{i,j\in \llbracket0,An\rrbracket \\ i<j}}
\left\{\log \Zhalf(i,i;j,j)+2(j-i)\Psi_0(\alpha) \right\}
\geq x f(n)^2 n^{1/3}
\right) \rightarrow 0 \qquad \text{ as }n \rightarrow \infty.
\end{equation}
\end{proposition}

\begin{proposition}\label{prop:thin-rect}
    For any $\varepsilon \in (0,2\alpha)$, there exists a $\xi \in (0,1)$ and positive constants $n_0, c$ such that for all $n \geq n_0$, $m \geq 1$ satisfying $m/n \in (0, \xi]$ and all $x \geq 0$, 
    \begin{equation}\label{eq:thin_rect}
        \Prob\left(\log \Zfull(1,1;n,m) + (n+m-2)\Psi_0(2\alpha-\varepsilon)\geq x \right) \leq e^{-cx}.
    \end{equation}  
\end{proposition}

\subsection{Proof of Proposition~\ref{prop:loggamma-one-pair-goal}}
\begin{proof}[Proof of Proposition~\ref{prop:loggamma-one-pair-goal}]
By Proposition \ref{prop:exit}, we know that there exists some $C,c > 0$ such that for all $n$ large enough
\[
\begin{aligned}
&\Prob\left(\log \Zexit_R(x_1,y_1;x_2,y_2) + 8bn\Psi_0(\alpha) \geq -Cf(n)^2n^{1/3}\right)\\
&\leq \Prob\left(\log Z^{\mathrm{exit},\frac14 f(n) n^{2/3}}(x_1,y_1;x_2,y_2) + 8bn\Psi_0(\alpha) \geq -Cf(n)^2n^{1/3}\right)\\
&\leq e^{-cf(n)^{3}}.
\end{aligned}
\]
Similarly, by Proposition \ref{prop:in}, for all sufficiently large $n$,
\[
\begin{aligned}
&\Prob\left(\log \Zin_R(x_1,y_1;x_2,y_2) + 8bn\Psi_0(\alpha) \leq -f(n)n^{1/3}\right)\\
&\leq \Prob\left(\log Z^{\mathrm{in}, (4bn)^{2/3}}(x_1,y_1;x_2,y_2) + 8bn\Psi_0(\alpha) \leq -f(n)n^{1/3}\right)\\
&\leq Cf(n)^{1/2}e^{-cf(n)}.
\end{aligned}
\]
Thus, 
\begin{equation}
    \Prob \left(\frac{\Zexit_R(x_1,y_1;x_2,y_2)}{\Zin_R(x_1,y_1;x_2,y_2)} \geq e^{-C f(n)^2 n^{1/3}+ f(n)n^{1/3}}\right) \leq Cf(n)^{1/2}e^{-cf(n)} + e^{-cf(n)^3}.
\end{equation}
Take $n_0$ large enough such that $\frac{\epsilon n^{1/3}}{\sqrt{2}q\sigma_p} > \log 2$ and $-Cf(n)^2 +f(n) \leq 0$ for all $n \geq n_0$, we know
\[
\begin{aligned}
    \Prob (\mathcal B_{n,\varepsilon}(\Z^2)) \leq\Prob \left(  \frac{\Zexit_R(x_1,y_1;x_2,y_2)}{\Zin_R(x_1,y_1;x_2,y_2)} \geq e^{-C f(n)^2n^{1/3}+f(n)n^{1/3}}\right) \leq Cf(n)^{1/2}e^{-cf(n)} + e^{-cf(n)^3}.
\end{aligned}
\]

For the upper tail of $\Zb_R$, we start with the following inequality
\[
\Zb_R(x_1,y_1;x_2,y_2) \leq \sum_{\substack{i,j \in \llbracket x_1,y_2 \rrbracket\\ i \leq j}}\omega^{\partial}_{i,i}\Zfull(x_1, y_1;i,i-1)\Zhalf(i,i;j,j)\Zfull(j,j;x_2,y_2).
\]

Thus,
\begin{equation}\label{eq:final_half}
    \begin{aligned}
        &\Prob\left(\log \Zb_R(x_1,y_1;x_2,y_2) + 8bn\Psi_0(\alpha) \geq -Cf(n)^{3/2}n^{1/3}\right)\\
        &\leq \Prob\Bigg(\max_{x_1 \leq i \leq j \leq y_2}\left\{ \log \omega_{i,i}^{\partial}+\log \Zfull(x_1,y_1;i,i-1) + \log \Zhalf(i,i;j,j) + \log \Zfull(j,j;x_2,y_2) \right\}\\
        & \qquad \qquad +2\log(4bn)+ 8bn\Psi_0(\alpha) \geq -Cf(n)^{3/2}n^{1/3}\Bigg)
    \end{aligned}
\end{equation}
Since $2\log(4bn)$ is negligible comparing to $f(n)^{3/2}n^{1/3}$, we can absorb this term into the constant $C$. Recall Proposition \ref{prop:thin-rect}. We know that there exists a $\xi \in (0, 1)$ and $n_0, c >0$ such that for all $n \geq n_0$, $m \geq 1$ satisfying $m/n \in (0, \xi]$ and all $x \geq 0$, 
\[
\Prob\left(\log \Zfull(1,1;n,m) + (n+m-2)\Psi_0(3/2\alpha)\geq x \right) \leq e^{-cx}.
\]

For
$x_1\le i\le j\le y_2$, set
\begin{equation}
\begin{aligned}
\mathcal A_{i,j}
=
\log \omega_{i,i}^{\partial}
+\log\Zfull(x_1,y_1;i,i-1)
+\log\Zhalf(i,i;j,j)+\log\Zfull(j,j;x_2,y_2).
\end{aligned}
\end{equation}
We split the maximum according to whether both full-space segments
have longitudinal length at least $\xi f(n)n^{2/3}$. Let $I = \llbracket x_1 + \xi f(n)n^{2/3}, y_2 - \xi f(n)n^{2/3} \rrbracket$. By a union bound,
the probability in \eqref{eq:final_half} is at most
\begin{align}\label{eq:final-two}
&\Prob\left(
\max_{\substack{
i, j \in I \\ i\leq j
}}
\mathcal A_{i,j}
+8bn\Psi_0(\alpha)
\ge-Cf(n)^{3/2}n^{1/3}
\right)+
\Prob\left(
\max_{\substack{
x_1\le i\le j\le y_2\\
i \notin I \text{ or } j \notin I
}}
\mathcal A_{i,j}
+8bn\Psi_0(\alpha)
\ge-Cf(n)^{3/2}n^{1/3}
\right).
\end{align}
For the first term we apply Proposition \ref{prop:shape_function} and upper bound it by
\begin{equation}
\begin{aligned}
    &\Prob\left( \max_{i\in I}\left\{\log \Zfull(x_1,y_1;i,i-1) -d(i-x_1, i-y_1)\right\} \geq Df(n)^2n^{1/3} \right)\\
    &+\Prob\left( \max_{j\in I}\left\{\log \Zfull(j,j;x_2,y_2) -d(x_2-j, y_2-j)\right\} \geq Df(n)^2n^{1/3} \right)\\
    &+\Prob\left( \max_{\substack{i, j\in I\\i \leq j}}\left\{\log \Zhalf(i,i;j,j) -d(j-i, j-i)\right\} \geq Df(n)^2n^{1/3} \right)\\
    &+ \Prob\left( \max_{i \in I}\left\{\log \omega_{i,i}^\partial \right\} \geq Df(n)^2n^{1/3} \right).
\end{aligned}    
\end{equation}
The first two terms are upper bounded by $Ce^{-cf(n)^{3/2}}$ by Proposition \ref{prop:loggamma-uniform-upper}, applied to parallelograms with longitudinal scale
$f(n)^{3/2}n$ and transversal width $f(n)n^{2/3}$. The third term is $o(1)$ by Proposition \ref{prop:hslg_moderate}. Lastly, the fourth term is bounded by $Cne^{-cf(n)n^{1/3}}$ because $\log \omega_{i,i}^\partial$ is subexponential. Combining the preceding four estimates, the contribution from
$i,j\in I$ is bounded by
\begin{equation}\label{eq:bulk-half-bound}
Ce^{-cf(n)^{3/2}}
+
o(1)
+
Cne^{-cf(n)n^{1/3}}
= o(1).
\end{equation}

It remains to consider $i<x_1+\xi f(n)n^{2/3}
$ or $
j>y_2-\xi f(n)n^{2/3}$.
We will only treat the case where $i<x_1+\xi f(n)n^{2/3}
$ and $
j \in I$ as the rest follows analogously. Set
\begin{equation}
\Delta
=
\Psi_0(3\alpha/2)-\Psi_0(\alpha)>0.
\end{equation}
Since we have assumed $f(n) \leq 4bn^{1/3}$ in \eqref{eq:f-upper-reduction}, we have $f(n)^{3/2}n^{1/3} \leq 2\sqrt{b}f(n)n^{1/2} =o(  f(n)n^{2/3})$ for all $n$ large enough. We then upper bound the second term in \eqref{eq:final-two} by
\begin{equation}
\begin{aligned}
    &\Prob\left( \max_{x_1 \leq i < x_1 + \xi f(n)n^{2/3}}\left\{\log \Zfull(x_1,y_1;i,i-1) + (2i-x_1-y_1)\Psi_0(3/2\alpha)\right\} \geq \frac{\Delta}{10}f(n)n^{2/3} \right)\\
    &+\Prob\left( \max_{j\in I}\left\{\log \Zfull(j,j;x_2,y_2) + (x_2+y_2-2j)\Psi_0(\alpha)\right\} \geq \frac{\Delta}{10}f(n)n^{2/3}  \right)\\
    &+\Prob\left( \max_{\substack{x_1 \leq i < x_1 + \xi f(n)n^{2/3}\\ j \in I,i \leq j}}\left\{\log \Zhalf(i,i;j,j) +2(j-i)\Psi_0(\alpha)\right\} \geq\frac{\Delta}{10}f(n)n^{2/3}  \right)\\
    &+ \Prob\left( \max_{x_1 \leq i < x_1 + \xi f(n)n^{2/3}}\left\{\log \omega_{i,i}^\partial \right\} \geq \frac{\Delta}{10}f(n)n^{2/3}  \right).
\end{aligned}    
\end{equation}
We upper bound the first term by $Cne^{-cf(n)n^{2/3}}$ using Proposition \ref{prop:thin-rect}. By Proposition~\ref{prop:peak},
\begin{equation}
d(a,b)\le-(a+b)\Psi_0(\alpha), \qquad d(j-i,j-i)=-2(j-i)\Psi_0(\alpha).
\end{equation}
Proposition~\ref{prop:loggamma-uniform-upper} bounds the second term by
$Ce^{-cf(n)^{1/2}n^{1/3}}$, while
Proposition~\ref{prop:hslg_moderate} shows that the third term is
$o(1)$. Lastly, the fourth term is bounded by $Cne^{-cf(n)n^{2/3}}$. Thus,
\begin{equation}
    \Prob \left(\frac{\Zb_R(x_1,y_1;x_2,y_2)}{\Zin_R(x_1,y_1;x_2,y_2)} \geq e^{-C f(n)^{3/2} n^{1/3}+ f(n)n^{1/3}} \right) \leq Cf(n)^{1/2}e^{-cf(n)} + o(1) = o(1).
\end{equation}
Since $-C f(n)^{3/2} n^{1/3}+ f(n)n^{1/3} \leq 0$ for all $n$ large enough, we get
\[
\begin{aligned}
    \Prob(\mathcal B_{n,\varepsilon}(\HH)) =\Prob \left( \log\left(1+ \frac{\Zb_R(x_1,y_1;x_2,y_2)}{\Zin_R(x_1,y_1;x_2,y_2)}\right) \geq \frac{\epsilon n^{1/3}}{\sqrt{2}q\sigma_p}\right) \rightarrow 0 \qquad \text{ as }n \rightarrow \infty.
\end{aligned}
\]
\end{proof}

\section{Shape function analysis for log-gamma polymer}\label{sec:shape}
We will devote this section to prove Proposition \ref{prop:shape_function} and an auxiliary lemma needed for the proof of Proposition \ref{prop:loggamma-uniform-upper}. Recall that 
\[
g(x) = \frac{\Psi_1(2\alpha - x)}{\Psi_1(x)}, \quad \phi(x) = x\Psi_0(g^{-1}(x)) + \Psi_0(2\alpha - g^{-1}(x)).
\]
Suppose that $x = m/n$ for some positive integers $n \geq m$ and let $\zeta = g^{-1}(m/n)$, we can rewrite $-n\phi(m/n)$ as the following
\[
\begin{aligned}
    -n\phi(m/n) &= -n\left[\Psi_1(2\alpha - \zeta)\Psi_1(\zeta)^{-1}\Psi_0(\zeta) + \Psi_0(2\alpha - \zeta)\right]\\
    &= -n\left(1+\frac{\Psi_1(2\alpha - \zeta)}{\Psi_1(\zeta) }\right)\left[ \frac{\Psi_1(2\alpha - \zeta)}{\Psi_1(\zeta) + \Psi_1(2\alpha - \zeta)}\Psi_0(\zeta) + \frac{\Psi_1(\zeta)}{\Psi_1(\zeta) + \Psi_1(2\alpha-\zeta)}\Psi_0(2\alpha - \zeta)\right]\\
    &= \left(n+m\right)\left[ -  \frac{\Psi_1(2\alpha - \zeta)}{\Psi_1(\zeta) + \Psi_1(2\alpha - \zeta)}\Psi_0(\zeta) - \frac{\Psi_1(\zeta)}{\Psi_1(\zeta) + \Psi_1(2\alpha-\zeta)}\Psi_0(2\alpha - \zeta)\right].
\end{aligned}
\]
Define $F: (0,2\alpha) \rightarrow \R$ to be
\[
F\left(\zeta \right) = -  \frac{\Psi_1(2\alpha - \zeta)}{\Psi_1(\zeta) + \Psi_1(2\alpha - \zeta)}\Psi_0(\zeta) - \frac{\Psi_1(\zeta)}{\Psi_1(\zeta) + \Psi_1(2\alpha-\zeta)}\Psi_0(2\alpha - \zeta).
\]
Then the shape function can be written as $(n+m)F(\zeta)$, which separates the total path length from its direction. The diagonal direction corresponds to $\zeta=\alpha$. The local expansion below from \cite{basu2024temporal} quantifies the quadratic cost of a small transversal deviation. To control the shape farther from the diagonal direction, we also establish the monotonicity of $F$.

\begin{proposition}[{\cite[Proposition~3.2]{basu2024temporal}}]\label{prop:slope}
    There exists $C_1, C_2, \epsilon > 0$ such that for any $m \in [1-\epsilon,1+ \epsilon]$, we have
    \begin{equation}
        \left| g^{-1}(m)-\alpha- C_1(m-1)\right| \leq C_2(m-1)^2.
    \end{equation}
\end{proposition}

\begin{proposition}[{\cite[Proposition 3.4]{basu2024temporal}}]\label{prop:shape}
    There exists constants $C_3, C_4, \epsilon > 0$ such that for any $z \in [-\epsilon, \epsilon]$, we have
    \begin{equation}
        \left|F(\alpha + z) - F(\alpha) + C_3z^2 \right| \leq C_4z^4.
    \end{equation}
\end{proposition}

\begin{proposition}\label{prop:peak}
The function $F$ is symmetric about $\alpha$, strictly increasing on $(0,\alpha)$, and strictly decreasing on $(\alpha,2\alpha)$.
\end{proposition}

\begin{proof}
The symmetry is immediate. Direct differentiation gives
\[
F'(\zeta)
=
-\frac{(\Psi_1(\zeta)\Psi_2(2\alpha -\zeta)+\Psi_1(2\alpha-\zeta)\Psi_2(\zeta))\bigl(\Psi_0(2\alpha-\zeta)-\Psi_0(\zeta)\bigr)}
{(\Psi_1(2\alpha-\zeta)+\Psi_1(\zeta))^2}.
\]
 Since \(\Psi_0\) is strictly increasing, for \(\zeta\in(0,\alpha)\), $\Psi_0(2\alpha -\zeta)-\Psi_0(\zeta)>0.$
Also \(\Psi_1>0\) and \(\Psi_2<0\) on \((0,\infty)\), so $F'(\zeta)>0$ for $\zeta\in(0,\alpha)$. Thus \(F\) is strictly increasing on \((0,\alpha)\). Since \(F(2\alpha-\zeta)=F(\zeta)\), it follows by symmetry that \(F\) is strictly decreasing on \((\alpha,2\alpha)\).
\end{proof}

Now we are ready to prove Proposition \ref{prop:shape_function}.
\begin{proof}
By Propositions~\ref{prop:slope}, \ref{prop:shape}, and
\ref{prop:peak}, there exists \(\delta \in (0,1), c_0>0\) such that
\begin{equation}
F(\alpha)-F\bigl(g^{-1}(1-u)\bigr)= F(g^{-1}(1))-F\bigl(g^{-1}(1-u)\bigr)\geq c_0u^2,
\qquad |u| < \delta.
\end{equation}
By Proposition~\ref{prop:peak},
\begin{equation*}
    a_\delta:=
    F(\alpha)
    -
    \max\left\{
        F\bigl(g^{-1}(1-\delta)\bigr),
        F\bigl(g^{-1}(1+\delta)\bigr)
    \right\}
    >0.
\end{equation*}
Monotonicity of $g^{-1}$ and Proposition~\ref{prop:peak} imply that, for
$\delta\leq |u|<1$,
\begin{equation*}
    F(\alpha)-F\bigl(g^{-1}(1-u)\bigr)
    \geq a_\delta
    \geq a_\delta u^2.
\end{equation*}
Taking $c=\min\{c_0,a_\delta\}$, we have
\begin{equation}\label{eq:F-quadratic-gap}
F(\alpha)-F\bigl(g^{-1}(1-u)\bigr)= F(g^{-1}(1))-F\bigl(g^{-1}(1-u)\bigr)\geq cu^2,
\qquad |u| < 1.
\end{equation}
Using
\[
d(p,q)=(p+q)F\bigl(g^{-1}(p/q)\bigr),
\]
for any $i,j \in \llbracket x_1 + \xi f(n)n^{2/3},y_2 - \xi f(n)n^{2/3} \rrbracket$ where $i \leq j$, we obtain
\begin{align*}
&-8bn\Psi_0(\alpha)
-\Bigl[
d(i-x_1,i-y_1)
-2(j-i)\Psi_0(\alpha)
+d(x_2-j,y_2-j)
\Bigr] \\
&\quad=
(2A-r)\left[
F(\alpha)-F\left(g^{-1}\left(1-\frac rA\right)\right)
\right] +
(2B+r)\left[
F(\alpha)-F\left(g^{-1}\left(1+\frac rB\right)\right)
\right].
\end{align*}
where \[
A=i+2bn,\qquad B=2bn-j,\qquad  r = \frac{1}{2}f(n)n^{2/3}.
\]
Since $r < A \leq 4bn$, for all $n$ large enough, \eqref{eq:F-quadratic-gap} gives
\[
(2A-r)\left[
F(\alpha)-F\left(g^{-1}\left(1-\frac rA\right)\right)
\right]
\geq c\frac{r^2}{A}
\geq \frac{c}{16b}f(n)^2n^{1/3},
\]
and Proposition \ref{prop:peak} imply that 
\[
(2B+r)\left[
F(\alpha)-F\left(g^{-1}\left(1+\frac rB\right)\right)
\right] \geq 0.
\]
\end{proof}

The multiscale construction in the next section concatenates paths through intermediate points that remain close to a straight line. To compare the resulting free energy with the shape function between the original endpoints, we need to control the accumulated shape function error. The following lemma shows that transverse displacements on the KPZ scale produce shape errors on the corresponding fluctuation scale.

\begin{lemma}\label{lem:shape-additivity}
Fix $\psi >1$ and $h > 0$. Let $(w,z), (w', z') \in \Z^2$ be such that $(w, z) \prec (w', z')$ and $m=\frac{z'-z}{w'-w} \in (\psi^{-1}, \psi)$. Consider a sequence of points in $\Z^2$ 
\begin{equation}
(w,z) = (w_0,z_0)\prec(w_1,z_1)\prec\cdots\prec(w_\ell,z_\ell) = (w',z')
\end{equation}
where for $1\le i\le\ell-1$,
\begin{equation}
z_i=z_0+m(w_i-w_0)+h_i, \qquad |h_i|
\le
h\min\left\{
(w_i-w_{i-1})^{2/3},
(w_{i+1}-w_i)^{2/3}
\right\}.
\end{equation}
Then there exists some $n_0 > 0$ and $C > 0$ that depends on $\psi$ and $h$ such that if $\min_{i} |w_{i+1}-w_i| \geq n_0$, then
\begin{equation}\label{eq:shape-additivity}
\left|
\sum_{i=0}^{\ell-1}
d(w_{i+1}-w_i,z_{i+1}-z_i)
-
d(w_\ell-w_0,z_\ell-z_0)
\right|
\le
C\sum_{i=0}^{\ell-1}(w_{i+1}-w_i)^{1/3}.
\end{equation}
\end{lemma}

\begin{proof}
    Let $d_i = w_{i+1} - w_i$ and $\delta_i = h_{i+1} - h_i$ where $h_0 = h_{\ell} = 0$. Then for $i \in \llbracket 0, \ell - 1 \rrbracket$,
    \[
    z_{i+1}-z_i = md_i + \delta_i, \qquad |\delta_i| \leq 2hd_i^{2/3}.
    \]
    For $n_0$ sufficiently large, we have
    \[
    m + \frac{\delta_i}{d_i} \in \left[(2\psi)^{-1}, 2\psi\right].
    \]
    Recall that $d(w,z) = d(z,w) =- w \phi(z/w)$ and $\phi$ is twice continuously differentiable on the compact interval. Taylor's theorem gives
\[
d(d_i, z_{i+1}-z_i) = d(z_{i+1}-z_i, d_i) = -d_i\phi\left(m+\frac{\delta_i}{d_i}\right) = -d_i\phi(m) - \phi'(m)\delta_i + R_i
\]
where $|R_i| \leq C \delta_i^2d_i^{-1} \leq 4Ch^2d_i^{1/3}$ for some constant $C >0$ that depends only on $\psi$. Finally, we conclude the proof with
\[
\sum_{i=0}^{\ell - 1} d_i = w_\ell - w_0, \qquad \sum_{i = 0}^{\ell - 1} \delta_i = h_{\ell} - h_0 = 0.
\]
\end{proof}

\section{Tail estimates for log-gamma polymer}\label{sec: tail}
\subsection{Proof of Proposition~\ref{prop:loggamma-uniform-upper}}
In this section, we adapt the multiscale covering argument of
\cite[Proposition~10.5]{basu2014last} to the positive-temperature
log-gamma polymer. The main ingredients are the uniform
free-energy tail estimates below and the approximate additivity of the
log-gamma shape function established above.

We first recall the uniform upper-tail and lower-tail estimates for the
full-space log-gamma free energy from
\cite[Theorems~3.6 and~3.8]{basu2024temporal}. The upper-tail estimate
combines an integrable moderate-deviation bound
\cite{barraquand2021fluctuations} with a large-deviation estimate
\cite{Georgiou_2013}, while the lower-tail estimate is
proved in \cite{basu2024temporal} by a coupling argument
inspired by
\cite{Emrah_2025} and \cite{landon2024tail}.
\begin{theorem}\label{thm:full-homogeneous}
    Fix $\psi > 0$. There exists some positive constants $C, c, n_0 > 0$ depending on $\psi$ such that for all $m/n \in [\psi^{-1}, \psi]$, $n \geq n_0$, and $x \geq 0$, we have
        \begin{equation}
        \Prob\left(\log \Zfull(1,1;n,m) + n\phi\left(m/n\right) \geq  xn^{1/3}\right) \leq Ce^{-c \min \{x^{3/2}, xn^{1/3}\}},
    \end{equation} 
    \begin{equation}
        \Prob\left(\log \Zfull(1,1;n,m) + n\phi\left(m/n\right) \leq - xn^{1/3}\right) \leq Ce^{-c \min \{x^{3/2}, xn^{1/3}\}}.
    \end{equation} 
By modifying the constants, we can extend the simple upper bound $Ce^{-cx}$ for all
$n\geq m\geq1$ where $m/n \in [\psi^{-1}, \psi]$, since only finitely many pairs with $n<n_0$ remain
and the inverse-gamma weights admit suitable positive and negative
moments.
\end{theorem}

Recall the definition of $U_{n,m,\ell}$ in Definition \ref{def:parallelogram} and let
\begin{equation}
L(U_{n,m,\ell})
=
U_{n,m,\ell}\cap\{(w,z): w\le n/8\},
\qquad
R(U_{n,m,\ell})
=
U_{n,m,\ell}\cap\{(w,z): w\ge 7n/8\}.
\end{equation}

We first prove the uniform lower-tail bound by concatenating polymers and showing that the free energy of each segment is unlikely to be too small. The upper-tail bound is more delicate because it requires controlling the full partition function; we obtain it by combining the uniform lower-tail bound with a one-point upper-tail estimate via FKG inequality. A multiscale covering then extends this control over endpoint pairs with slopes bounded away from zero and infinity.

\begin{lemma}\label{lem:loggammalower-left-right}
Fix $\psi>1$ and let $m \in (\psi^{-1}, \psi)$ and $U = U_{n,m,1}$. There exist constants $n_1,c, C>0$, depending only on
$\alpha$ and $\psi$, such that, for every $n\ge n_1$ and
$x \geq 0$,
\begin{equation}
\mathbb P\left(
\min_{\substack{
(u,v)\in L(U)\\(u', v')\in R(U)
}}
\left\{\log \Zfull(u,v;u',v') -d(u'-u, v'-v) \right\} 
\le-x n^{1/3} 
\right)
\le Ce^{-cx}.
\end{equation}
\end{lemma}

\begin{proof}
Let $r = (r_1, r_2) = (n/2, mn/2)$. By Lemma \ref{lem:shape-additivity}, for any $p = (p_1,p_2) \in L(U)$ and $q = (q_1, q_2) \in R(U)$, we could chose $n_0$ large enough such that
\begin{equation}\label{eq:shape_near_linearity}
\left| d(q_1-p_1,q_2-p_2) - d(q_1 - r_1, q_2 - r_2) - d(r_1-p_1, r_2 - p_2) \right| \leq Cn^{1/3}.
\end{equation}
Let $\FF(p;q) = \log \Zfull(p;q) - d(q_1-p_1, q_2-p_2)$. Since $\log \Zfull(p;q) \geq \log \Zfull(p;r) + \log \Zfull(r; q)$, it follows from \eqref{eq:shape_near_linearity} that
\[
\FF(p;q) \geq \FF(p;r) + \FF(r;q) - Cn^{1/3}.
\]
By symmetry and union bound, it therefore suffices to prove 
\begin{equation}\label{eq:tree}
    \Prob\left( \min_{p \in L(U)} \FF(p; r) \leq -xn^{1/3} \right) \leq Ce^{-cx}.
\end{equation}
We prove \eqref{eq:tree} by a multiscale tree construction. Let $n_0$ be the maximum of the constants from Lemma \ref{lem:shape-additivity} and Theorem \ref{thm:full-homogeneous}. Choose $a_0 \gg n_0$, depending only on $\alpha$ and
$\psi$. For each sufficiently large $n$ ($n^{1/3} \gg a_0$), choose $K\ge1$ and
$a\in[a_0,8a_0)$ such that $n = 2a8^K$.
For $0\le k\le K$, define
\begin{equation}
S_k
=
\left\{
aj8^k:
j=0,1,\ldots,8^{K-k}
\right\},\qquad
T_k
=
\left\{
a^{2/3}j4^k:
j = -2\cdot4^{K-k}, \cdots, 2\cdot4^{K-k}
\right\}.
\end{equation}
Let $V_k$ be the corresponding lattice net in the tilted coordinates:
\begin{equation}
V_k
=
\left\{
( s,ms+t)\in U:
s\in S_k,\ t\in T_k
\right\}.
\end{equation}
For $0\le k\le K$, let $G_k$ be the graph with vertices $V_k$ and edges $(p,q) \in V_k \times V_k$ satisfying
\begin{equation}\label{eq:tree-edge-condition}
a8^k
\le q_1-p_1
\le 10a8^k, \qquad 
\left|
(q_2-p_2)-m(q_1-p_1)
\right|
\le 20a^{2/3}4^k.
\end{equation}
For each $p\in L(U)$, choose $p_1 \in V_1$ by following the line
segment from $p$ to $r$ to the first vertical net line whose
horizontal distance from $p$ lies between $a$ and $9a$, and then
rounding the transverse coordinate to the nearest point of $V_1$.
Set
\begin{equation}
G_{\mathrm{in}}
=
\left\{
(p,p_1):
p\in L(U)
\right\}.
\end{equation}
By taking $a_0$ sufficiently large, every edge in
$G_{\mathrm{in}}$ and $G_k$ for $0 \leq k \leq K$ has slope in $\left[(2\psi)^{-1}, 2\psi\right]$.

We will apply Theorem~\ref{thm:full-homogeneous} with $2\psi$ in
place of $\psi$. For $0 \leq k \leq K$, define
\begin{equation}
\mathcal E_k
=
\left\{
\widetilde{\mathcal F}(p;q)\ge-\frac{x n^{1/3}}{10}
\left(\frac32\right)^{k-K} \text{ for all } (p,q)\in G_k\right\},
\end{equation}

\begin{equation}
\mathcal E_{\mathrm{in}}
=
\left\{
\widetilde{\mathcal F}(p;q)\ge-\frac{xn^{1/3}}{10} \text{ for all }(p,q)\in  G_{\mathrm{in}}
\right\}.
\end{equation}
Because $n^{1/3} > a$, we apply Theorem~\ref{thm:full-homogeneous} with the simple upper bound $Ce^{-cx}$
\begin{equation}\label{eq:all_tree}
    \begin{aligned}
        \Prob\left( \mathcal E^c_{\mathrm{in}} \cup \bigcup_{k=0}^K \mathcal E_k^c\right) \leq \sum_{k=0}^K |G_k|Ce^{-cx(4/3)^{K-k}} + |G_\mathrm{in}|Ce^{-cxn^{2/9}}
    \end{aligned}
\end{equation}
Because $|G_k| \leq |V_k|^2 \leq 8^{4(K-k)}$ and $|G_\mathrm{in}| \leq n^2$, we can upper bound \eqref{eq:all_tree} by $Ce^{-cx}$. 

It remains to show 
\begin{equation}\label{eq:tree-implication}
\mathcal E_{\mathrm{in}}
\cap
\bigcap_{k=0}^{K}\mathcal E_k
\subseteq
\left\{
\min_{p\in L(U)}
\widetilde{\mathcal F}(p;r)
\ge
-x n^{1/3}
\right\}.
\end{equation}
Fix $p = p_{0}\in L(U)$, and let $p_1\in V_1$ be the point chosen
above. We recursively construct
\begin{equation}
p=p_{0}\prec p_1\prec\cdots\prec p_K=r
\end{equation}
as follows. Suppose that $p_k$ has been constructed. Consider the
first vertical line in $S_{k+1}$ lying strictly to the right of
$p_k$. Let $p_{k+1}\in V_{k+1} \subset V_k$ be the point obtained by intersecting
this vertical line with the line segment joining $p$ and $r$, and
then rounding the transverse coordinate to the nearest point in
$T_{k+1}$. By construction,
\begin{equation}
a8^k
\le
p_{k+1,1}-p_{k,1}
\le
8a8^k.
\end{equation}
Moreover, each $p_k$ is within transverse distance
$a^{2/3}4^k$ of the line segment joining $p$ and $r$. Consequently, by setting $n_1$ large enough, 
\begin{equation}
\left|
(p_{k+1,2}-p_{k,2})
-
m(p_{k+1,1}-p_{k,1})
\right|
\le
20a^{2/3}4^k.
\end{equation}
Thus, we see that for $1 \leq k \leq K-1$, $(p_k,p_{k+1})\in G_k$. 

For $0 \leq k \leq K-1$, let $d_i = p_{i+1,1} - p_{i,1}$. Apply Lemma \ref{lem:shape-additivity},
\begin{equation}
\left|
\sum_{k=0}^{K-1}
d(p_{k+1,1}-p_{k,1},p_{k+1,2}-p_{k,2})
-
d(r_1-p_{0,1},r_2-p_{0,2})
\right|
\le
C\sum_{k=0}^{K-1}d_k^{1/3} \leq Cn^{1/3}.
\end{equation}
Therefore, by taking $x$ large enough,
\[
\FF(p; r) \geq \sum_{k=0}^{K-1} \FF(p_k; p_{k+1}) -Cn^{1/3} \geq \sum_{k=0}^{K-1} \FF(p_k; p_{k+1}) -\frac{xn^{1/3}}2 .
\]
On the event $\mathcal E_{\mathrm{in}}
\cap
\bigcap_{k=0}^{K-1}\mathcal E_k,
$
\[
\FF(p_0;p_1) \geq -\frac{xn^{1/3}}{10}, \qquad \FF(p_k; p_{k+1}) \geq
-\frac{x n^{1/3}}{10}
\left(\frac32\right)^{k-K} \text{ for }1 \leq k \leq K-1.
\]
Thus, 
\[
\FF(p; r) \geq -\frac{xn^{1/3}}{10} -\frac{xn^{1/3}}{5} - \frac{xn^{1/3}}{2} \geq -xn^{1/3}.
\]
\end{proof}

\begin{lemma}\label{lem:loggammaupper-left-right}
Fix $\psi>1$ and let $m \in (\psi^{-1}, \psi)$ and $U = U_{n,m,1}$. There exist constants $n_1,c, C>0$, depending only on
$\alpha$ and $\psi$, such that, for every $n\ge n_1$ and
$x \geq 0$,
\begin{equation}
\mathbb P\left(
\max_{\substack{
(u,v)\in L(U)\\(u', v')\in R(U)
}}
\left\{\log \Zfull(u,v;u',v') -d(u'-u, v'-v) \right\}
\ge x n^{1/3}
\right)
\le Ce^{-cx}.
\end{equation}
\end{lemma}

\begin{proof}
    Let
\begin{equation}
p^-=
\left(-\frac{7n}{8},-\frac{7mn}{8}\right),
\qquad
p^+=
\left(\frac{15n}{8},\frac{15mn}{8}\right).
\end{equation}
For any $p\in L(U)$ and $q\in R(U)$, the sequence $p^-\prec p\prec q\prec p^+$ satisfies the assumptions of Lemma~\ref{lem:shape-additivity}. Thus,
\begin{equation}\label{eq:upper-shape-additivity}
\left|
d(p_1-p_1^-, p_2-p_2^-)+d(q_1-p_1, q_2-p_2)+d(p_1^+-q_1, p_2^+-q_2)-d(p_1^+-p_1^-, p_2^+- p_2^-)
\right|
\le
Cn^{1/3}.
\end{equation}
On the other hand,
\begin{equation}\label{eq:upper-inf-sup}
\log\Zfull(p^-;p^+)
\ge
\log\Zfull(p^-;p)
+
\log\Zfull(p;q)
+
\log\Zfull(q;p^+).
\end{equation}
Combining this with \eqref{eq:upper-shape-additivity}  and taking $x$ large enough gives
\begin{equation}\label{eq:upper-concatenation}
\begin{aligned}
\FF(p^-;p^+)
&\ge
\FF(p^-;p)
+
\FF(p;q)
+
\FF(q;p^+)
-x n^{1/3}.
\end{aligned}
\end{equation}
Define the events
\begin{equation}
\mathcal A_1
=
\left\{
\min_{p\in L(U)}
\FF(p^-;p)
\ge
-{x n^{1/3}}
\right\}, \quad
\mathcal A_2
=
\left\{
\min_{q\in R(U)}
\FF(q;p^+)
\ge
-{x n^{1/3}}
\right\},
\end{equation}
and
\begin{equation}
\mathcal A_3
=
\left\{
\max_{\substack{p\in L(U)\\q\in R(U)}}
\FF(p;q)
\ge
4x n^{1/3}
\right\}.
\end{equation}
By Lemma \ref{lem:loggammalower-left-right}, we can take $x$ sufficiently large so that $
\mathbb P(\mathcal A_1), \mathbb P(\mathcal A_2)\ge\frac12.$
Since the events $\mathcal A_1, \mathcal A_2$ and $\mathcal A_3$ are all increasing with respect to the weight and
\begin{equation}\label{eq:upper-event-containment}
\left\{
\FF(p^-;p^+)
\ge
{x n^{1/3}}
\right\}
\supseteq
\mathcal A_1\cap\mathcal A_2\cap\mathcal A_3,
\end{equation}
we apply the FKG inequality and get
\begin{equation}\label{eq:FKG_full}
\begin{aligned}
\mathbb P\left(
\FF(p^-;p^+)
\ge
{x n^{1/3}}
\right)
\ge
\mathbb P(\mathcal A_1\cap\mathcal A_2\cap\mathcal A_3) \ge
\mathbb P(\mathcal A_1)
\mathbb P(\mathcal A_2)
\mathbb P(\mathcal A_3)\ge
\frac14\mathbb P(\mathcal A_3).
\end{aligned}
\end{equation}
We conclude the proof by applying Theorem \ref{thm:full-homogeneous} to \eqref{eq:FKG_full}.
\end{proof}

Finally, we are ready to prove Proposition \ref{prop:loggamma-uniform-upper}.
\begin{proof}[Proof of Proposition~\ref{prop:loggamma-uniform-upper}]
    Let $n_1$ be the constant from Lemma \ref{lem:loggammaupper-left-right}. Choose $n_0$ large enough such that $n_0^{1/6} \gg n_1$. For $k \geq 0$, we define
    \begin{equation}
        S_k = \left\{in2^{-k/4}/32: i = -4,  \cdots, 32\times 2^{k/4}+ 4\right\}, \quad T_k = \left\{jn^{2/3}2^{-k/6}: j = -2 \times2^{k/6}, \cdots, 2\times 2^{k/6}\right\},
    \end{equation}
    and 
    \begin{equation}
        V_k = \left\{(x,y): x \in S_k, y-x \in T_k\right\}.
    \end{equation}
    For $i \in \llbracket-4, 32 \times 2^{k/4} -28 \rrbracket$ and $j_1, j_2 \in \llbracket -2\times 2^{k/6}, 2\times 2^{k/6} -2 \rrbracket$ satisfying
    \[
    \frac{1}{2\psi} < 1+ \frac{j_2-j_1}{2^{-k/12}n^{1/3}} < 2\psi
    \]
    we define $U_{k}(i,j_1,j_2)$ to be the parallelogram with vertices in $V_k$ as follows
    \begin{equation}
    \begin{aligned}
        &\left(\frac{i2^{-k/4}n}{32},\frac{i2^{-k/4}n}{32} + j_12^{-k/6}n^{2/3} \right), \quad  \left(\frac{i2^{-k/4}n}{32},\frac{i2^{-k/4}n}{32} + (j_1+2)2^{-k/6}n^{2/3} \right),\\
         &\left(\frac{(i+32)2^{-k/4}n}{32},\frac{(i+32)2^{-k/4}n}{32} + j_22^{-k/6}n^{2/3} \right), \quad  \left(\frac{(i+32)2^{-k/4}n}{32},\frac{(i+32)2^{-k/4}n}{32} + (j_2+2)2^{-k/6}n^{2/3} \right).
    \end{aligned}
    \end{equation}
    Let $\mathcal U_k$ denote the set of all such parallelograms $U_k(i,j_1,j_2)$ with vertices in $V_k$. It is not difficult to see that if $(p,q) \in \mathcal S_\psi(U)$ and $|q_1- p_1| \in \left[\frac{25}{32}2^{-k/4}n, \frac{31}{32}2^{-k/4}n\right]$, there exists some $U_k(i,j_1, j_2) \in \mathcal U_k$ such that $p \in L(U_k(i,j_1, j_2))$ and $q \in R(U_k(i,j_1, j_2)).$ Also, because $25/32 < 31/32 \times 2^{-1/4}$, 
    \[
    \left[\frac{25}{32}2^{-(k+1)/4}n, \frac{31}{32}2^{-(k+1)/4}n\right] \cap \left[\frac{25}{32}2^{-k/4}n, \frac{31}{32}2^{-k/4}n\right] \neq \emptyset.
    \]
    Consider the following events
    \[
    F_k = \left\{\max_{\substack{p \in L(U_k)\\ q\in R(U_k)}} \FF(p;q) \leq xn^{1/3} \text{ for all }U_k \in \mathcal U_k\right\}, \qquad
    F= \left\{\max_{\substack{(p,q) \in \mathcal S_\psi(U)\\|q_1-p_1| \leq n^{1/6}}} \FF(p;q) \leq xn^{1/3}\right\}.
    \]
    It follows that 
    \[
     F \cap \bigcap_{k=0}^{\frac{10 \log n}{3 \log 2}} F_k \subset \left\{
\max_{(p,q)\in\mathcal S_\psi(U)}
\FF(p;q)
\leq xn^{1/3}
\right\}.
    \]
    We apply Theorem \ref{thm:full-homogeneous} to upper bound $\Prob(F^c) \leq Cn^{4}e^{-cxn^{5/18}} \leq Ce^{-cx}$ and apply Lemma \ref{lem:loggammaupper-left-right} to upper bound $\Prob(F_k^c) \leq C|\mathcal U_k|e^{-cx2^{k/12}}$. Since 
    \[
    |\mathcal U_k| \leq C \times 2^{k/4} \times 2^{k/3} \leq C \times 2^{7k/12},
    \]
    we have
    \[
    \sum_{k=0}^\infty \Prob(F_k^c) \leq C \sum_{k=0}^\infty 2^{7k/12}e^{-cx2^{k/12}} \leq Ce^{-cx}.
    \]
    Combining all the estimates above concludes the proof.
    \end{proof}

\subsection{Proof of Proposition~\ref{prop:hslg_moderate}}
The proof of Proposition~\ref{prop:hslg_moderate} relies on the following two inputs, together with the FKG inequality. We first recall the
following constrained point-to-box lower-tail bound from
\cite[Theorem~3.18]{basu2024temporal}. 

\begin{theorem}\label{thm:min_box_to_point}
For $a,b\in\mathbb Z^2$ and $h>0$, let $R^h(a;b)$ be the
parallelogram with vertices $a\pm(-h,h), b\pm(-h,h).$ Let $U=R^{n^{2/3}}(0,0;n,n), V=R^{n^{2/3}}{(0,0;9n/10,9n/10)}.$ There exists $C, c, n_0>0$ such that for all $n \geq n_0$ and $t >0$, we have
\begin{equation}
    \Prob \left( \min_{(x,y) \in V} \left\{ \log \Zfull_{U}(x,y; n,n) + (2n-x-y)\Psi_0(\alpha) \right\} \leq -tn^{1/3}\right) \leq Ce^{-ct}
\end{equation}
\end{theorem}

Our second input is a one-point upper-tail estimate for the half-space free energy between two points on the diagonal. We prove it using the symmetry under exchanging the diagonal and first-row parameters, followed by a comparison with a stationary full-space polymer.

\begin{proposition}\label{prop:hs_one_point}
Fix $\alpha > 0$. Let $\theta_n$ be defined as in Theorem \ref{thm_half_loggamma_to_DL}. For any fixed $A > 0$ and $x > 0$
\begin{equation}\label{eq:half-one-point}
\Prob\left(
\log \Zhalf(0,0;An,An)+2An\Psi_0(\alpha)
\geq x f(n)^2 n^{1/3}
\right) \rightarrow 0 \qquad \text{ as }n \rightarrow \infty.
\end{equation}
\end{proposition}
\begin{proof}

By stochastic monotonicity in the boundary parameter \cite[Lemma 13.1]{stat}, it suffices to consider the case for \(\theta_n= g(n) n^{-1/3}\) where $g(n) > 0$ and $g(n)= o(f(n))$. By
\cite[Proposition 8.1]{barraquand2020half}, the law of
\(\log\Zhalf(0,0;An,An)\) is invariant under exchanging the diagonal parameter $-\theta_n$ and
first-row parameter $\alpha$. After the exchange, the weights on the first row are independent $\text{Gamma}^{-1}(\alpha - \theta_n)$ random variables and the rest weights are independent  $\text{Gamma}^{-1}(2\alpha)$ random variables. The resulting partition function is stochastically dominated by the full-space stationary log-gamma polymer with independent $\operatorname{Gamma}^{-1}(2\alpha)$ bulk weights, $\operatorname{Gamma}^{-1}(\alpha-\theta_n)$ horizontal boundary weights, and $\operatorname{Gamma}^{-1}(\alpha+\theta_n)$ vertical boundary weights. Let $Z^{\mathrm{stat}}$ denote the full-space stationary partition function defined above and let
\[
m_n = \left\lfloor An \frac{\Psi_1(\alpha - \theta_n)}{\Psi_1(\alpha+\theta_n)}\right\rfloor.
\]
Since $\theta_n\to0$, the parameters
$\alpha\pm\theta_n$ remain in a fixed compact subset of $(0,\infty)$. Hence we can apply \cite[Theorem~1.1]{Landon2022UpperTB} with constants $C,c>0$ chosen uniformly in $n$ such that for sufficiently large $n$ and $0 < x \leq (An)^{2/3}$, we have
\begin{equation}\label{eq:half-stat-one-point}
\Prob\left(
\log Z^{\mathrm{stat}}(0,0;An,m_n)+ m_n\Psi_0(\alpha + \theta_n) + An\Psi_0(\alpha- \theta_n)
\geq x n^{1/3}
\right)
\leq C e^{-cx^{3/2}}.
\end{equation}
Let $R_k = Z^{\mathrm{stat}}(0,0;An,k)/Z^{\mathrm{stat}}(0,0;An,k-1)$. By stationarity, we know that $R_k$ are i.i.d. $\mathrm{Gamma}^{-1}(\alpha+\theta_n)$ random variables and
\[
\log Z^{\mathrm{stat}}(0,0;An,An)   \stackrel{d}{=} \log Z^{\mathrm{stat}}(0,0;An,m_n) - \sum_{k=An+1}^{m_n} \log R_k.
\]
Thus,
\[
\begin{split}
    &\Prob\left(
\log Z^{\mathrm{stat}}(0,0;An,An)+ An\Psi_0(\alpha + \theta_n) + An\Psi_0(\alpha- \theta_n)
\geq x f(n)^2 n^{1/3}
\right)\\
&\leq \Prob\left(\sum_{k=An+1}^{m_n} \log R_k + (m_n - An)\Psi_0(\alpha + \theta_n) \leq -\frac{1}2 xf(n)^2n^{1/3} \right)\\& \quad + \Prob\left(
\log Z^{\mathrm{stat}}(0,0;An,m_n)+ m_n\Psi_0(\alpha +\theta_n) + An\Psi_0(\alpha- \theta_n)
\geq \frac12 x f(n)^2n^{1/3}
\right)
\end{split}
\]
It is not difficult to check that there exists some $C_A>0$ such that for all $n$ large enough
\begin{equation}
m_n-An\leq C_Ag(n)n^{2/3},
\end{equation}
A standard Chernoff bound and \eqref{eq:half-stat-one-point} give
\begin{equation}
\begin{aligned}
&\Prob\left(
\log Z^{\mathrm{stat}}(0,0;An,An)
+An\Psi_0(\alpha+\theta_n)
+An\Psi_0(\alpha-\theta_n)
\geq xf(n)^2n^{1/3}
\right)\\
&\leq
C\exp\left(
-c\min\left\{
x^{3/2}f(n)^3,n
\right\}
\right)+
C\exp\left(
-c\min\left\{
\frac{x^2f(n)^4}{g(n)},
xf(n)^2n^{1/3}
\right\}
\right).
\end{aligned}
\end{equation}
Moreover, Taylor expansion gives
\begin{equation}
\left|
An\Psi_0(\alpha+\theta_n)
+An\Psi_0(\alpha-\theta_n)
-2An\Psi_0(\alpha)
\right|
\leq
C_Ag(n)^2n^{1/3}
=
o\left(f(n)^2n^{1/3}\right).
\end{equation}
Combining the preceding estimates with stochastic domination, we obtain,
for every fixed \(x>0\)
\begin{equation}
\Prob\left(
\log\Zhalf(0,0;An,An)+2An\Psi_0(\alpha)
\geq xf(n)^2n^{1/3}
\right)
\rightarrow0 \qquad \text{ as }n \rightarrow \infty.
\end{equation}

\end{proof}

We now combine these two inputs with the FKG inequality to prove
Proposition~\ref{prop:hslg_moderate}.

\begin{proof}[Proof of Proposition~\ref{prop:hslg_moderate}]
By stochastic monotonicity in the boundary parameter
\cite[Lemma~13.1]{stat}, it suffices to consider $\theta_n=g(n)n^{-1/3}$ where $g(n)>0$ and $g(n)=o(f(n))$. Set $N=An$. Let $E(x)$ denote the event in
Proposition~\ref{prop:hslg_moderate} and define
\[
\begin{aligned}
G_1(x)&=\left\{\log \Zhalf(-3N,-3N;-N+(2N)^{2/3},-N-(2N)^{2/3})+4N\Psi_0(\alpha)
\geq-xn^{1/3}\right\},\\
G_2(x)&=\left\{\min_{0\leq i\leq N}
\left\{\log \Zhalf(-N+(2N)^{2/3},-N-(2N)^{2/3};i,i)+(2N+2i)\Psi_0(\alpha)\right\}
\geq-xn^{1/3}\right\},\\
G_3(x)&=\left\{\min_{0\leq j\leq N}
\left\{\log\Zhalf(j,j;2N+(2N)^{2/3},2N-(2N)^{2/3})+(4N-2j)\Psi_0(\alpha)\right\}
\geq-xn^{1/3}\right\},\\
G_4(x)&=\left\{\log\Zhalf(2N+(2N)^{2/3},2N-(2N)^{2/3};4N,4N)+4N\Psi_0(\alpha)
\geq-xn^{1/3}\right\}.
\end{aligned}
\]
Let $U = R^{(2N)^{2/3}}(-3N+(2N)^{2/3},
-3N-(2N)^{2/3};-N+(2N)^{2/3},
-N-(2N)^{2/3})
\subseteq\mathbb H$. By \cite[Lemma 13.1]{stat}, replacing the diagonal weights by
\(\operatorname{Gamma}^{-1}(2\alpha)\) bulk weights and then restricting
paths to \(U\) gives
\[
\log\Zhalf(-3N,-3N;-N+(2N)^{2/3},-N-(2N)^{2/3})
\succeq_{\mathrm{st}}
\log \Zfull_{U}
(-3N,-3N;-N+(2N)^{2/3},-N-(2N)^{2/3}).
\]
The same comparison applies to the
partition functions defining \(G_2(x),G_3(x)\), and \(G_4(x)\). Thus, by Theorem
\ref{thm:min_box_to_point} and translation invariance of the partition function, we may choose \(x_0>0\) such that for all $n$ large enough, 
\[
\Prob(G_\ell(x_0))\geq\frac12,\qquad 1\leq\ell\leq4.
\]
All the events above are increasing in the polymer weights. Thus, by FKG inequality,
\[
\Prob\left( E(x)\cap\bigcap_{\ell=1}^4G_\ell(x_0)\right)
\geq\frac1{16}\Prob(E(x)).
\]
We conclude the proof by noticing 
\[
E(x)\cap\bigcap_{\ell=1}^4G_\ell(x_0) \subset \left\{\log \Zhalf(-3N,-3N;4N,4N) + 14N\Psi_0(\alpha) \geq (xf(n)^2-4x_0)n^{1/3}\right\}
\]
and applying Proposition \ref{prop:hs_one_point}.
\end{proof}

\subsection{Proof of Proposition~\ref{prop:thin-rect}}

\begin{proof}
For $\lambda \in (0,1)$, we have
\begin{equation}\label{eq:ineq}
\Zfull(1,1;n,m)^\lambda =\left( \sum_{\pi \in \Pi[(1,1) \rightarrow (n,m)]} \prod_{v \in \pi \setminus (1,1)} \omega_v\right)^\lambda \leq \sum_{\pi \in \Pi[(1,1) \rightarrow (n,m)]} \prod_{v \in \pi \setminus (1,1)} \omega_v^\lambda.
\end{equation}
Recall that if $\omega \sim \mathrm{Gamma}^{-1}(\beta)$ and $\beta-\lambda > 0$, $\E[\omega^\lambda] = {\Gamma(\beta - \lambda)}\Gamma(\beta)^{-1}.$
For $\lambda < \min\{1, 2\alpha\}$, we take expectation on both sides of \eqref{eq:ineq} and use independence along each path
\[
\E[\Zfull(1,1;n,m)^\lambda] \leq \binom{n+m-2}{n-1}\left( \frac{\Gamma(2\alpha -\lambda)}{\Gamma(2\alpha)}\right)^{n+m-2}.
\]
By Markov's inequality, we know that
\[
\begin{split}
    \Prob\left( \log \Zfull(1,1;n,m) \geq u \right)\leq \exp\left(-\lambda u + \log \binom{n+m-2}{n-1}+ (n+m-2)\log \left(\frac{\Gamma(2\alpha - \lambda)}{\Gamma(2\alpha)} \right) \right).
\end{split}
\]
Thus,
\[
\begin{split}
    &\Prob\left( \log \Zfull(1,1;n,m) - \frac{1}{\lambda} \log \binom{n+m-2}{n-1}-\frac{(n+m-2)}{\lambda} \log \left(\frac{\Gamma(2\alpha - \lambda)}{\Gamma(2\alpha)} \right) \geq x \right) \leq e^{-\lambda x}.
\end{split}
\]
Notice that for $\lambda$ small enough, we have
\[
\frac{1}{\lambda}\log\frac{\Gamma(2\alpha - \lambda)}{\Gamma(2\alpha)} \approx -\Psi_0(2\alpha).
\]
Recall that $-\Psi_0$ is a strictly decreasing function. Let $L = -\Psi_0(2\alpha-\varepsilon) + \Psi_0(2\alpha) > 0$. Let us choose $\lambda$ small such that 
\[
\left|\frac{1}{\lambda}\log\frac{\Gamma(2\alpha - \lambda)}{\Gamma(2\alpha)} + \Psi_0(2\alpha)\right| < \frac{L}{2}.
\]
Thus,
\[
\begin{aligned}
 -\frac{(n+m-2)}{\lambda}\log\frac{\Gamma(2\alpha - \lambda)}{\Gamma(2\alpha)} &\geq (n+m-2)\Psi_0(2\alpha) - (n+m-2)L/2\\
&= (n+m-2)\Psi_0(2\alpha-\varepsilon) + (n+m-2)L/2
\end{aligned}
\]
By Stirling's formula, we see that
\[
\log \binom{n+m-2}{n-1} =\left( \left(1+\frac{m}{n}\right)\log \left(1+\frac{m}{n}\right) -  \left(\frac{m}{n}\right)\log\left(\frac{m}{n}\right) \right) n + \mathcal{O}(\log n)
\]
Let us choose $\xi$ small enough such that for any $r \in (0, \xi)$
\[
(1+r)\log(1+r)-r\log r < L\lambda/4.
\]
Thus, 
\[
- \frac{1}{\lambda} \log \binom{n+m-2}{n-1}-\frac{(n+m-2)}{\lambda} \log \left(\frac{\Gamma(2\alpha - \lambda)}{\Gamma(2\alpha)} \right) \geq (n+m-2)\Psi_0(2\alpha-\varepsilon).
\] 
We can then conclude
\[
\begin{split}
&\Prob\left( \log \Zfull(1,1;n,m) +(n+m-2)\Psi_0(2\alpha-\varepsilon) \geq x \right)\\
    &\leq \Prob\left( \log \Zfull(1,1;n,m) - \frac{1}{\lambda} \log \binom{n+m-2}{n-1} -\frac{(n+m-2)}{\lambda} \log \left(\frac{\Gamma(2\alpha - \lambda)}{\Gamma(2\alpha)} \right) \geq x \right)\\
    &\leq e^{-\lambda x}.
\end{split}
\]
\end{proof}

\section{Sharpness of the boundary parameter}
\label{sec:sharpness}

We now show that the condition on the boundary parameters in our main
results is sharp at the level of scales.  Recall that the observation
window is at distance \(f(n)n^{2/3}\) from the boundary, where $f(n)\rightarrow\infty$ and $f(n)=O(n^{1/3})$.
Consider boundary parameters
\[
    \rho_n=g(n)n^{-1/3} \leq \rho < 1/2,
    \qquad
    \theta_n=h(n)n^{-1/3} \leq \theta < \alpha,
\]
for exponential LPP and the log-gamma polymer, respectively. We assume
that $g, h, f$ are all strictly positive and of the same order. In this regime, both the gain obtained by reaching the boundary and the
bulk shape cost of traveling from the observation window to the boundary
are of order $f(n)^2n^{1/3}.$

For a sufficiently large fixed temporal separation, the boundary gain
dominates the bulk shape loss. We can therefore find a fixed
space-time point at which the half-space last passage time and the
half-space free energy, after full-space centering and normalization,
diverge to \(+\infty\) in probability.  Since the corresponding full-space
directed landscape marginal is almost surely finite, this rules out
convergence to the directed landscape.

Lastly, it is enough to establish this divergence along a subsequence, hence we will assume the following without changing notations
\begin{equation}\label{eq:assumption}
    \frac{f(n)}{n^{1/3}}\rightarrow c_f\in[0,\infty),
    \qquad
    \frac{g(n)}{f(n)}\longrightarrow \lambda_{\mathrm{LPP}}\in(0,\infty),
    \qquad
    \frac{h(n)}{f(n)}\longrightarrow \lambda_{\mathrm{LG}}\in(0,\infty).
\end{equation}

\subsection{Exponential last passage percolation}
\label{subsec:lpp-sharpness}

\begin{proposition}\label{prop:lpp-sharpness}
Under \eqref{eq:assumption}, there exist constants $T_0$ such that for all $T\geq T_0$
\begin{equation}
L^{n,f}_{\mathrm{half}}(0,0;0,T)
\xlongrightarrow{\Prob}+\infty.
\end{equation}
\end{proposition}

\begin{proof}
Let $u_n = (f(n)n^{2/3}, 0)$ and $v_n = (f(n)n^{2/3} + Tn, Tn)$. 
For some $c > 0$ to be chosen later, notice that 
    \begin{equation}
        \begin{split}
            &\Prob\left(
\Lhalf(u_n;v_n)-\Lfull(u_n;v_n)
<
cf(n)^2n^{1/3}
\right)\\
& \leq \Prob\left(
\Lfull(u_n;v_n)-4Tn
\geq
\frac{c}2f(n)^2n^{1/3}
\right)+ \Prob\left(
\Lhalf(u_n;v_n)-4Tn
\leq
 \frac{3c}2f(n)^2n^{1/3}\right)
        \end{split}
    \end{equation}
By the standard one-point upper-tail estimate for exponential LPP
\cite{Johansson_2000,Small_Deviation}, we can upper bound the first term by
$Ce^{-cf(n)^3}$. For the second term, let $(i,i) = (\frac{Tn}4, \frac{Tn}4)$ and $(j,j) = (\frac{3Tn}4, \frac{3Tn}4)$. For $T$ and $n$ large enough, we have
\[
\begin{aligned}
g(u_n;i,i)+g(j,j;v_n)&=
Tn+\frac{Tn}{2}
\left(
\sqrt{1-\frac{4f(n)}{Tn^{1/3}}}
+
\sqrt{1+\frac{4f(n)}{Tn^{1/3}}}
\right)\\
&\geq
Tn+\frac{Tn}{2}
\left(
2-\frac{16f(n)^2}{T^2n^{2/3}}
\right)\\
&=
2Tn-\frac{8}{T}f(n)^2n^{1/3},
\end{aligned}
\]
where we used 
\[
\sqrt{1-x}+\sqrt{1+x}\geq2-x^2,
\qquad 0\leq x<1.
\]
On the other hand, we have
\[
\begin{aligned}
g^{\mathrm{half}}(i,i;j,j):= \frac{2Tn}{1-4g(n)^2n^{-2/3}}
\geq 2Tn+8Tg(n)^2n^{1/3}\geq
2Tn+2T\lambda_{\mathrm{LPP}}^2f(n)^2n^{1/3},
\end{aligned}
\]
where the last inequality holds for all sufficiently large $n$ by
\eqref{eq:assumption}. Therefore,
\[
\begin{aligned}
g^{\mathrm{half}}(i,i;j,j)
+g(u_n;i,i)
+g(j,j;v_n)\geq
4Tn+
\left(
2T\lambda_{\mathrm{LPP}}^2-\frac{8}{T}
\right)f(n)^2n^{1/3}.
\end{aligned}
\]
holds for all $T$ and $n$ large enough. Choose $T_0$ sufficiently large such that $2c := 2T_0\lambda_{\mathrm{LPP}}^2 - 8/T_0 > 0$.
\begin{equation}
    \begin{split}
                    &\Prob\left(
\Lhalf(u_n;v_n)- 4Tn - 2cf(n)^2n^{1/3}
<
-\frac{c}2f(n)^2n^{1/3}
\right)\\
& \quad \leq \Prob\left(
\Lhalf(i,i;j,j)-g^{\mathrm{half}}(i,i;j,j)
\leq
 -\frac{c}6f(n)^2n^{1/3}\right) \\
&\qquad  + \Prob\left(
\Lhalf(u_n;i,i)-g(u_n;i,i)
\leq
-\frac{c}6f(n)^2n^{1/3}
\right)\\
& \qquad + \Prob\left(
\Lhalf(j,j;v_n)-g(j,j;v_n)
\leq
-\frac{c}6f(n)^2n^{1/3}
\right)
    \end{split}
\end{equation}
For the first term, we apply \cite[Theorem 1.12(2)]{dauvergne2026half} to bound it by $Ce^{-c'\frac{f(n)^4}{g(n)}}$. For the last two terms, we can
apply \cite[Theorem 3.13]{dauvergne2026half} to upper bound them by $Ce^{-c'f(n)^6}$.

Consequently, as $n \rightarrow \infty$, we have
\begin{equation}
\Prob\left(
\Lhalf(u_n;v_n)-\Lfull(u_n;v_n)
\geq
cf(n)^2n^{1/3}
\right)
\rightarrow1.
\end{equation} Since $f(n) \rightarrow \infty$, 
\[
L_{\mathrm{half}}^{n,f}(0,0;0,T) - L_{\mathrm{full}}^{n,f}(0,0;0,T)  \xlongrightarrow{\Prob}+\infty.
\]
Lastly, as $L_{\mathrm{full}}^{n,f}(0,0;0,T)
$ converges to $
\mathcal L(0,0;0,T)$ in distribution and $ \mathcal L(0,0;0,T)$ is almost surely finite, 
it follows that
\[
L_{\mathrm{half}}^{n,f}(0,0;0,T)
\xlongrightarrow{\Prob}+\infty.
\]
\end{proof}
\subsection{Log-gamma polymer}
For the log-gamma polymer, the analogues of the two lower-tail bounds in
\cite[Theorem~1.12(2)]{dauvergne2026half} and \cite[Theorem~3.13]{dauvergne2026half} are not available in the literature. One could in principle
adapt their argument to the positive-temperature setting. For our purposes, it is simpler to work directly with the full-space
constrained point-to-box lower-tail bound
\cite[Theorem~3.18]{basu2024temporal}, which we restate as
Theorem~\ref{thm:min_box_to_point}. From this bound, we derive lower-tail estimates for constrained bulk free energy and for half-space free energy with endpoints on the diagonal. These estimates control the cost of reaching the boundary and retain a positive gain from the boundary attraction, allowing us to implement the same sharpness strategy as in LPP.

\begin{proposition}\label{prop:full-space-constrained-lower-tail}
    For $a, b \in \Z^2$ and $h > 0$, let $R^h(a;b)$ be the parallelogram with vertices $a \pm (-h,h)$, $b \pm (-h,h)$. Fix any function $k: \N \rightarrow (0,\infty)$ such that $k(n) \rightarrow \infty$ and $k(n) = o(n^{1/3})$. Let $U = R^{k(n)n^{2/3}}(0,0;n,n)$. Then, there exist some constants $n_0, c, C > 0$ such that for any $p_n = (p_n^1, p_n^2)\in \{(-i,i): |i| \leq k(n)n^{2/3}\}$, $q_n = (q_n^1, q_n^2) \in \{(n-i,n+i): |i| \leq k(n)n^{2/3}\}$, $n \geq n_0$, and $ x \geq 0$, we have
    \begin{equation}
\label{eq:full-space-constrained-lower-tail}
\mathbb P\left(
\log Z_{U}^{\mathrm{full}}(p_n;q_n)
+(q_n^1+q_n^2 - p_n^1 - p_n^2 )\Psi_0(\alpha)
\le -(C+x) k(n)^2n^{1/3}
\right)
\le e^{-cxk(n)^3}.
\end{equation}
\end{proposition}

\begin{proof}
Let $K_n = 27k(n)^3$ and $\ell_n=\frac{n}{27k(n)^3}.$ Let $z_0=p_n$ and $z_{K_n}=q_n$. Choose $\sigma, \sigma'\in\{-1,1\}$ and set
\[
z_1=p_n+(\ell_n,\ell_n)
    +\sigma(-\ell_n^{2/3},\ell_n^{2/3}), \qquad z_{K_n-1}=q_n-(\ell_n,\ell_n)
    +\sigma'(-\ell_n^{2/3},\ell_n^{2/3})
\]
so that
\[
U_0:=R^{\ell_n^{2/3}}
\left(
p_n+\sigma(-\ell_n^{2/3},\ell_n^{2/3});z_1
\right), \quad  U_{K_n-1}:=R^{\ell_n^{2/3}}
\left(z_{K_n-1};
q_n+\sigma'(-\ell_n^{2/3},\ell_n^{2/3})
\right)
\subseteq U.
\]
Connect $z_1$ and $z_{K_n-1}$ by a line segment. For each
$2\le r\le K_n-2$, let $z_r$ be the unique intersection of this
segment with the anti-diagonal
\[
\{(x,y)\in\mathbb R^2:x+y=2r\ell_n\}.
\]
Then
\[
z_{r+1}-z_r
=(\ell_n,\ell_n)+a_r(-1,1),
\qquad 1\le r\le K_n-2
\]
where
\[
|a_r|
\le \frac{2k(n)n^{2/3}}{K_n-2}
\le 3(27)^{-1/3}\ell_n^{2/3} \leq \ell_n^{2/3}
\]
for all sufficiently large $n$. For $1\le r\le K_n-2$, define
\[
U_r
=
R^{\ell_n^{2/3}}
\bigl(z_{r+1}-(\ell_n,\ell_n);z_{r+1}\bigr).
\]
The segment from $z_1$ to $z_{K_n-1}$ lies at transverse distance at least
$\ell_n^{2/3}$ from the lateral sides of $U$. Consequently,
$U_r\subseteq U$ for every $0\le r\le K_n-1$. For each $r$, let $z_r = (z^1_r, z^2_r)$ and
\[
X_{r,n}
=
\log Z_{U_r}^{\mathrm{full}}(z_r;z_{r+1})
+(z_{r+1}^1 + z_{r+1}^2-z_r^1 -z_r^2 )\Psi_0(\alpha).
\]
Because $\ell_n \rightarrow \infty$ as $n \rightarrow \infty$, we can apply Theorem~\ref{thm:min_box_to_point} to each $X_{r,n}$
\[
\Prob\left(X_{r,n}\le -t\ell_n^{1/3}\right)
\le Ce^{-ct}.
\]
Because parallelograms $U_r$ lie
in disjoint temporal slabs and the weight at the initial point is
excluded, the variables $X_{r,n}$ are independent. Hence an exponential Markov inequality shows that there exists some $C>0$ such that 
\[
\Prob\left(
\sum_{r=0}^{K_n-1}X_{r,n}
\le -(C+x)K_n\ell_n^{1/3}
\right)
\le e^{-cxK_n}, \quad x \geq 0.
\]
We conclude the proof by concatenation,
\[
\log Z_U^{\mathrm{full}}(p_n;q_n) +(q_n^1+q_n^2-p_n^1-p_n^2)\Psi_0(\alpha)
\ge
\sum_{r=0}^{K_n-1}
\left[\log Z_{U_r}^{\mathrm{full}}(z_r;z_{r+1})+
(z_{r+1}^1+z_{r+1}^2-z_r^1-z_r^2)
\Psi_0(\alpha)\right].
\]

\end{proof}
Now we are ready to prove the lower tail of the half-space free energy for points on the diagonal.

\begin{proposition}\label{prop:half-space-diagonal-lower-tail}
Fix $\alpha>0$. Let $\theta_n = -h(n)n^{-1/3}\in(-\alpha,0)$ satisfy $h(n) \rightarrow \infty$ and $h(n) = o(n^{1/3})$. Then there exist constants $c,C,n_0>0$ and $-\Psi_2(\alpha) > b >0$ such that, for all
$n\ge n_0$
\[
\Prob\left(
\log Z^{\mathrm{half}}(0,0;n,n)
+n\Psi_0(\alpha+\theta_n) + n\Psi_0(\alpha-\theta_n)
\leq -b|\theta_n|^2n
\right)
\le C e^{-c|\theta_n|^3n}.
\]
\end{proposition}

\begin{proof}
Let $\delta > 0$ be a small constant to be chosen later and let $m_n = \delta n |\theta_n|$. By
\cite[Proposition~8.1]{barraquand2020half}, the law of
$Z^{\mathrm{half}}(0,0;n,n)$ is unchanged if the diagonal parameter
$\theta_n$ is exchanged with the first row parameter $\alpha$. In the resulting environment, the first-row weights are independent
$\text{Gamma}^{-1}(\alpha+\theta_n)$ random variables, while the remaining
weights are independent
$\text{Gamma}^{-1}(2\alpha)$ random variables. Let $\omega_{1},\ldots,\omega_{m_n}$ be the first $m_n$ weights on the first row and let
\[
S_n=\sum_{i=1}^{m_n}\log \omega_i.
\]
Moreover, let $U = R^{\frac{1}2m_n}(m_n,0;n+m_n/2,n-{m_n}/2)$. By \cite[Lemma 13.1]{stat}, we know that $\log Z^{\mathrm{half}}(0,0;n,n)$ stochastically dominates
\begin{equation}\label{eq:disect}
S_n  +
\log Z_{U}^{\mathrm{full}}\bigl(m_n,0;n,n\bigr).
\end{equation}
Since the centered variables
$\log\omega_i+\Psi_0(\alpha+\theta_n)$ have uniformly bounded
exponential moments, Bernstein's inequality gives
\begin{equation}\label{eq:RW}
\Prob\left(
S_n+m_n\Psi_0(\alpha+\theta_n)
\le -\frac{\Psi_1(\alpha)}{2} m_n|\theta_n|
\right)
\le C e^{-c\delta|\theta_n|^3n}
\end{equation}
for some $C, c>0$ and for all $n$ large enough. 
Moreover, Proposition~\ref{prop:full-space-constrained-lower-tail}
implies there exists some $C,c>0$ such that for all $n$ large enough
\begin{equation}\label{eq:Z_restircted}
\Prob\left(
\log Z_{U}^{\mathrm{full}}\bigl(m_n,0;n,n\bigr)
+(2n-m_n)\Psi_0(\alpha)
\le -Cm_n^2n^{-1}
\right)
\le
Ce^{-c\delta^3|\theta_n|^3n}.
\end{equation}
Consequently, with probability at least
\(1- Ce^{-c\delta^3|\theta_n|^3n}\),
\[
\begin{aligned}
\log Z^{\mathrm{half}}(0,0;n,n)
+2n\Psi_0(\alpha)
&\geq
m_n\bigl[
\Psi_0(\alpha)-\Psi_0(\alpha+\theta_n)
\bigr]
-\frac{\Psi_1(\alpha)}2m_n|\theta_n|
-C\frac{m_n^2}{n}.
\end{aligned}
\]
Since $\Psi_0(\alpha)-\Psi_0(\alpha+\theta_n)
\geq
{\Psi_1(\alpha)}|\theta_n|$ for all sufficiently large \(n\), we obtain
\[
\log Z^{\mathrm{half}}(0,0;n,n)
+2n\Psi_0(\alpha)
\geq
\left(
\frac{\Psi_1(\alpha)}2\delta-C\delta^2
\right)n\theta_n^2.
\]
Choose \(\delta>0\) sufficiently small that
\[
\beta:=
\frac{\Psi_1(\alpha)}2\delta-C\delta^2
\in(0,-\Psi_2(\alpha)).
\]
Finally,
\[
2\Psi_0(\alpha)
-\Psi_0(\alpha+\theta_n)
-\Psi_0(\alpha-\theta_n)
=
-\Psi_2(\alpha)\theta_n^2
+O(\theta_n^4).
\]
It follows that with probability at least
\(1- Ce^{-c\delta^3|\theta_n|^3n}\),
\[
\begin{aligned}
&\log Z^{\mathrm{half}}(0,0;n,n)
+n\Psi_0(\alpha+\theta_n)
+n\Psi_0(\alpha-\theta_n)\geq
-\left(
-\Psi_2(\alpha)-\beta+o(1)
\right)n\theta_n^2.
\end{aligned}
\]
Taking $b=-\Psi_2(\alpha)-\frac{\beta}{2}
\in(0,-\Psi_2(\alpha))$ and increasing $n_0$ completes the proof.
\end{proof}

\begin{proposition}\label{prop:log-gamma-sharpness}
Under \eqref{eq:assumption}, there exists \(T_0>0\) such that, for every
fixed \(T\geq T_0\),
\[
h_{\mathrm{half}}^{n,f}(0,0;0,T)
\xlongrightarrow{\Prob}+\infty.
\]
\end{proposition}

\begin{proof}
Let $u_n=(f(n)n^{2/3}, 0)$ and $v_n = (f(n)n^{2/3}+Tn,Tn)$. We want to show that there exists some $c>0$ such that the following probability vanishes as $n \rightarrow \infty$:
 \begin{equation}\label{eq:constantc}
        \begin{split}
            &\Prob\left(\log
\Zhalf(u_n;v_n)-\log\Zfull(u_n;v_n)
<
cf(n)^2n^{1/3}
\right)\\
& \leq \Prob\left(
\log \Zfull(u_n;v_n)+2Tn\Psi_0(\alpha)
\geq
\frac{c}2f(n)^2n^{1/3}
\right)\\ &
\quad + \Prob\left(
\log \Zhalf(u_n;v_n)+2Tn\Psi_0(\alpha)
\leq
 \frac{3c}2f(n)^2n^{1/3}\right)
        \end{split}
    \end{equation}
For the first term, we can apply \cite[Proposition 3.8]{basu2024temporal} and upper bound it by $Ce^{-cf(n)^3}$. For the second term, we distinguish two cases.

\medskip
\noindent\textit{Case 1: \(f(n)/n^{1/3}\to0\).}
Let $(i,i) = (\frac{Tn}{4}, \frac{Tn}{4})$ and $(j,j) = (\frac{3Tn}{4}, \frac{3Tn}{4})$. In the proof of Proposition \ref{prop:half-space-diagonal-lower-tail}, we have shown that there exists some $\beta, C,c > 0$ such that for all sufficiently large $n$ 
\begin{equation}\label{eq:diagonal-positive-gain}
\Prob\left(
\log Z^{\mathrm{half}}(i,i;j,j)
+Tn\Psi_0(\alpha)
\leq \beta Th(n)^2n^{1/3}
\right) \leq Ce^{-ch(n)^3T}.
\end{equation}
Let \(a_n=f(n)n^{2/3}\). Define
\[
U_1
=
R^{a_n/2}\left(
u_n;
\left(i+\frac{a_n}{2},i-\frac{a_n}{2}\right)
\right), \quad U_2
=
R^{a_n/2}\left(
\left(j+\frac{a_n}{2},j-\frac{a_n}{2}\right);
v_n
\right).
\]
We apply Proposition \ref{prop:full-space-constrained-lower-tail} to upper bound the probability of the following events by $Ce^{-cf(n)^3/T^2}$
\[
\begin{aligned}
\log Z_{U_1}^{\mathrm{full}}(u_n;i,i) + \left(\frac{Tn}{2} - a_n\right)\Psi_0(\alpha)
&\leq
-\frac{C}{T}f(n)^2n^{1/3},\\
\log Z_{U_2}^{\mathrm{full}}(j,j;v_n) + \left(\frac{Tn}{2} + a_n\right)\Psi_0(\alpha)
&\leq
-\frac{C}{T}f(n)^2n^{1/3}.
\end{aligned}
\]
By \cite[Lemma~13.1]{stat} and concatenation,
\[
\log Z^{\mathrm{half}}(u_n;v_n)
\succeq_{\mathrm{st}}
\log Z_{U_1}^{\mathrm{full}}(u_n;i,i)
+\log Z^{\mathrm{half}}(i,i;j,j)
+\log Z_{U_2}^{\mathrm{full}}(j,j;v_n).
\]
A union bound gives
\[
\begin{aligned}
&\Prob\left(
\log Z^{\mathrm{half}}(u_n;v_n)
+2Tn\Psi_0(\alpha)
\leq
\left(
\beta T h(n)^2-\frac{C}{T}f(n)^2
\right)n^{1/3}
\right)\leq
Ce^{-cT h(n)^3}
+
Ce^{-cf(n)^3/T^2}.
\end{aligned}
\]
Since ${h(n)}/{f(n)}
\longrightarrow\lambda_{\mathrm{LG}}>0,$
for all sufficiently large \(n\),
\[
\beta T h(n)^2-\frac{C}{T}f(n)^2
\geq
\left(
\frac{\beta T\lambda_{\mathrm{LG}}^2}{2}
-\frac{C}{T}
\right)f(n)^2.
\]
We now choose \(T_0\) sufficiently large and \(c>0\) sufficiently small
so that, for every \(T\geq T_0\),
\[
\frac{\beta T_0\lambda_{\mathrm{LG}}^2}{2}
-\frac{C}{T_0}
>
\frac{3c}{2}.
\]
\medskip
\noindent\textit{Case 2: \(f(n)/n^{1/3}\to c_f>0\).}
In this case $\theta_n \rightarrow \theta_* = -c_f\lambda_{\mathrm{LG}}\in(-\alpha,0).$ Fix \(\overline\theta\in(\theta_*,0)\). For all sufficiently large \(n\),
the half-space model with parameter \(\theta_n\) stochastically dominates
the model with the fixed parameter \(\overline\theta\). Let $\log \Zhalf_{\overline{\theta}}$ denote the free energy of the half-space model with boundary parameter $\overline\theta$. Let $a_n = f(n)n^{2/3}$. Since $f(n)/n^{1/3} \rightarrow c_f$, $a_n \leq 2c_fn$ for all $n$ large enough. For $T > 4c_f$, by \cite[Theorem~6.9]{imamura2022solvablemodelskpzclass},
\[
\Prob\left(\log Z_{\overline\theta}^{\mathrm{half}}(a_n,a_n;Tn,Tn)
+(Tn-a_n) 
\left[
\Psi_0(\alpha+\overline\theta)
+\Psi_0(\alpha-\overline\theta)\right]
\leq -n^{2/3} \right) = o(1).
\]
On the other hand,
\[
\begin{aligned}
   \Prob\left(\log Z^{\mathrm{full}}(a_n,0;a_n,a_n)
+a_n \Psi_0(2\alpha) \leq -n^{2/3} \right) &= o(1), \\
\Prob\left(\log Z^{\mathrm{full}}(Tn,Tn;Tn+a_n;Tn)
+a_n \Psi_0(2\alpha) \leq -n^{2/3} \right) &= o(1).
\end{aligned}
\]
By stochastic domination and concatenation, 
\[
\Prob\left(\log \Zhalf(u_n;v_n) +(Tn-a_n) 
\left[
\Psi_0(\alpha+\overline\theta)
+\Psi_0(\alpha-\overline\theta)\right] + 2a_n\Psi_0(2\alpha) \leq -3n^{2/3} \right) = o(1).
\]
Take $T_0$ large enough so that for all sufficiently large $n$,
\[
T_0\left[ 2\Psi_0(\alpha)-
\Psi_0(\alpha+\overline\theta)
-\Psi_0(\alpha-\overline\theta)\right] + \frac{a_n}{n} [\Psi_0(\alpha+\overline\theta)
+\Psi_0(\alpha-\overline\theta) - 2\Psi_0(2\alpha)] \geq \eta > 0.
\]
This is possible because $\Psi_0$ is strictly concave and $a_n/n \rightarrow c_f$. Since $f(n)^2n^{1/3}\leq 2c_f^2n$ for all $n$ sufficiently large, we may decrease the constant \(c>0\) in \eqref{eq:constantc} such that ${3c}c_f^2<\eta/2.$
It follows that
\[
\Prob\left(
\log Z^{\mathrm{half}}(u_n;v_n)
+2Tn\Psi_0(\alpha)
\leq
\frac{3c}{2}f(n)^2n^{1/3}
\right) = o(1).
\]
Combining the two cases with the full-space upper-tail estimate, the
remainder of the proof is identical to that of
Proposition~\ref{prop:lpp-sharpness}, and we conclude that
\[
h_{\mathrm{half}}^{n,f}(0,0;0,T/2)
\xlongrightarrow{\Prob}+\infty.
\]
\end{proof}

\section{Phase transition of the exponential last passage percolation}\label{sec:9}

In this section, we prove Theorem~\ref{thm:phase diagram} by comparing the boundary-touching and bulk parts of the last passage time. Half-space shift invariance \cite[Corollary~1.1]{he2026shift} identifies the distribution of the boundary-touching part with a last passage time from the origin to an off-diagonal endpoint, for which an exact Fredholm Pfaffian formula suitable for asymptotic analysis is available \cite[Proposition~1.6]{Pfaffian_Schur}. Heuristically, the boundary-touching part enters the Gaussian regime before the gain from boundary attraction fully compensates for the shape cost of reaching the boundary. A shape gap with the bulk part therefore persists beyond the onset of Gaussian fluctuations. The following Gaussian limit theorems confirm this picture. Their proofs, based on asymptotic analysis of the Fredholm Pfaffian formula, are deferred to Appendix~\ref{appendix}.

\begin{theorem}\label{thm: Gaussian}
    Fix $T>0$. Let $M_n = Tn + f(n)n^{2/3}$ and $K_n = Tn - f(n)n^{2/3}$. Assume that
    \[
    f(n), g(n) \rightarrow \infty, \qquad f(n),g(n) = o(n^{1/3}), \qquad \frac{g(n)}{f(n)} \rightarrow \lambda \in \left(\frac{1}{4T}, \infty\right].
    \]
    Moreover, let $\alpha_n = 1/2 - g(n)n^{-1/3}$, $\theta_n = g(n)n^{-1/3}$, and
    \[
    h_n = \frac{M_n}{1/2+\theta_n} + \frac{K_n}{1/2 - \theta_n}, \qquad \tau_n^2 = \frac{K_n}{(1/2-\theta_n)^2} - \frac{M_n}{(1/2+\theta_n)^2}.
    \]
For the half-space exponential LPP model with diagonal weights of rate $\alpha_n$, we have
    \[
\lim_{n \rightarrow \infty} \Prob \left(\frac{\Lhalf(0,0;M_n, K_n) - h_n}{\tau_n} \leq x\right) = \int_{-\infty}^x \frac{1}{\sqrt{2\pi}} e^{-y^2/2}
dy\]
\end{theorem}

\begin{theorem}\label{thm: Gaussian2}
    Fix $T > 0$. Let $M_n = Tn + f(n)n^{2/3}$ and $K_n = Tn - f(n)n^{2/3}$. Assume that 
    \[
    \frac{f(n)}{n^{1/3}} \rightarrow c_f  \in (0, T), \qquad \frac{g(n)}{n^{1/3}} \rightarrow c_g \in \left(0,\frac12\right), \qquad \lambda :=\frac{c_g}{c_f}  \in \left(\frac{1}{2\left(T + \sqrt{T^2-c_f^2}\right)},\infty\right).
    \]
    let $\alpha_n = 1/2 - g(n)n^{-1/3}$, $\theta_n = g(n)n^{-1/3}$, and
    \[
    h_n = \frac{M_n}{1/2+\theta_n} + \frac{K_n}{1/2 - \theta_n}, \qquad \tau_n^2 = \frac{K_n}{(1/2-\theta_n)^2} - \frac{M_n}{(1/2+\theta_n)^2}.
    \]
For the half-space exponential LPP model with diagonal weights of rate $\alpha_n$, we have
    \[
\lim_{n \rightarrow \infty} \Prob \left(\frac{\Lhalf(0,0;M_n, K_n) - h_n}{\tau_n} \leq x\right) = \int_{-\infty}^x \frac{1}{\sqrt{2\pi}} e^{-y^2/2}
dy\]
\end{theorem}
 
 We now prove Theorem~\ref{thm:phase diagram} by combining the preceding Gaussian limits with estimates for the bulk part. Although shift invariance gives only the marginal law of the boundary-touching part, not its joint law with the bulk part, this will suffice because the bulk part converges to a deterministic limit under the relevant centering and scaling. In the boundary-dominated regime, we show that the bulk part converges to $-\infty$ under the boundary centering and scaling. In the critical window, we center by the bulk shape and show that the bulk part converges to zero on the boundary fluctuation scale. These deterministic limits will give the joint convergence needed to take the maximum.

\begin{proof}[Proof of Theorem \ref{thm:phase diagram}]
    For $p \prec q$ in $\HH$, let us define
\[
L^{\mathrm{b}}(p;q) = \max_{\substack{\pi \in \Pi_{\HH}[p \rightarrow q]\\ \pi \cap \D \neq \emptyset}} \sum_{(i,j) \in \pi \setminus p} \nu^{\mathrm{half}}_{i,j}, \qquad L^{\mathrm{in}}(p;q) = \max_{\substack{\pi \in \Pi_{\HH}[p \rightarrow q]\\ \pi \cap \D = \emptyset}} \sum_{(i,j) \in \pi \setminus p} \nu^{\mathrm{half}}_{i,j}
\]
where 
\[
\Lhalf(p;q) = \max\{L^{\mathrm{b}}(p;q), L^{\mathrm{in}}(p;q)\}.
\]
Let $u_n = (f(n)n^{2/3}, 0)$ and $v_n = (f(n)n^{2/3}+Tn, Tn)$. By the shift invariance of the half-space exponential LPP \cite[Corollary 1.1]{he2026shift}, we know that 
    \[
   L^{\mathrm{b}}(u_n; v_n) + \nu^{\mathrm{half}}_{f(n)n^{2/3}, 0}\stackrel{d}{=} \Lhalf(0,0;Tn+f(n)n^{2/3}, Tn - f(n)n^{2/3}) + \nu^{\mathrm{half}}_{0,0}.
    \]
    Since the
starting weights will vanish under our scaling, we ignore them in the following proof. By Theorem \ref{thm: Gaussian} and Theorem \ref{thm: Gaussian2}, we know that there exists some $\varepsilon >0$ for $\lambda \in \left(\frac{1}{2T}-\varepsilon, \infty \right]$,
\[
\lim_{n \rightarrow \infty}\Prob\left( \frac{\Lhalf(0,0;Tn+f(n)n^{2/3}, Tn - f(n)n^{2/3}) -h_n}{\tau_n}  \leq x\right) = \int_{-\infty}^x \frac{1}{\sqrt{2\pi}}e^{-y^2/2}dy.
\]
where 
\[
h_n -4Tn =  \frac{8g(n)n^{1/3}(2Tg(n) - f(n))}{1-4g(n)^2n^{-2/3}}, \qquad 
\tau_n=
\frac{2^{3/2}n^{1/3}
\sqrt{4Tg(n)-f(n)\bigl(1+4g(n)^2n^{-2/3}\bigr)}}
{1-4g(n)^2n^{-2/3}}.
\]
When $\lambda > \frac{1}{2T}$, we know that there exists some $C_1, C_2 > 0$ such that for all $n$ large enough,
\[
h_n -4Tn \geq C_1g(n)^2n^{1/3}, \qquad \tau_n \leq C_2g(n)^{1/2}n^{1/3}
\]
Notice that for any fixed $x > 0$
\[
\begin{split}
&\Prob\left( L^\mathrm{in}(u_n; v_n) -h_n \leq -x\tau_n \right) \geq \Prob\left( \Lfull(u_n; v_n) -4Tn \leq (C_1g(n)^{3/2} -C_2x )g(n)^{1/2}n^{1/3} \right) \rightarrow 1
\end{split}
\]
by the tightness of full-space
exponential LPP \cite[Theorem~1.6]{Johansson_2000}. Together with the shift-invariance, we have
\[
\frac{L^\mathrm{in}(u_n; v_n) -h_n}{\tau_n} \xlongrightarrow{\Prob} -\infty, \quad \frac{L^\mathrm{b}(u_n; v_n) -h_n}{\tau_n} \xlongrightarrow{d} \mathcal{N}(0,1)
\]
Thus, by continuous mapping theorem, we have
\[
\frac{\Lhalf(u_n; v_n) -h_n}{\tau_n}\xlongrightarrow{d} \mathcal{N}(0,1).
\]
When $g(n) = \frac{1}{2T}f(n) + \xi f(n)^{-1/2}$, we can still apply Theorem \ref{thm: Gaussian}, Theorem  \ref{thm: Gaussian2}, and shift-invariance to get
\[
\frac{L^\mathrm{b}(u_n; v_n) -h_n}{\tau_n} \xlongrightarrow{d} \mathcal{N}(0,1).
\]
When $f(n), g(n) = o(n^{1/3})$, we have 
\[
h_n = 4Tn + 8\xi\sqrt{f(n)}n^{1/3} + o\left(\sqrt{f(n)}n^{1/3}\right), \qquad \tau_n = \sqrt{8f(n)}n^{1/3}(1+o(1)).
\]
\[
\frac{L^\mathrm{b}(u_n; v_n) -4Tn}{\sqrt{f(n)}n^{1/3}} \xlongrightarrow{d} 8\xi + \sqrt{8}\mathcal{N}(0,1).
\]
When $f(n), g(n)$ are both of order $n^{1/3}$, we have
\[
h_n = 4Tn + \frac{8\xi\sqrt{c_f}}{1-c_f^2/T^2}\sqrt{n}+o(\sqrt{n}), \qquad \tau_n = \sqrt{\frac{8c_f}{1-c_f^2/T^2}} \sqrt{n}+o(\sqrt{n}).
\]
\[
\frac{L^\mathrm{b}(u_n; v_n) -4Tn}{\sqrt{f(n)}n^{1/3}} \xlongrightarrow{d} \frac{8\xi}{1-c_f^2/T^2} + \sqrt{\frac{8}{1-c_f^2/T^2}}\mathcal{N}(0,1).
\]

Fix $\varepsilon > 0$. The following probability is bounded by $Ce^{-c\varepsilon\sqrt{f(n)}}$, which goes to zero as $n \rightarrow \infty$, by Proposition \ref{prop:lpp-constrained-moderate} and Proposition \ref{prop:tree_sup}.
\[
\begin{split}
&\Prob\left(
\left|L^{\mathrm{in}}(u_n; v_n) - 4Tn \right| \geq \varepsilon \sqrt{f(n)}n^{1/3}
\right)\\
&\leq \Prob\left(
L^{\mathrm{in}, (Tn)^{2/3}}(u_n; v_n) - 4Tn  \leq -\varepsilon \sqrt{f(n)}n^{1/3} 
\right) +\Prob\left(
\Lfull (u_n; v_n) - 4Tn \geq \varepsilon \sqrt{f(n)}n^{1/3}
\right)
\end{split}
\]
Thus, 
\[
\frac{
L^{\mathrm{in}}(f(n)n^{2/3},0;Tn+f(n)n^{2/3},Tn)-4Tn
}{\sqrt{f(n)}\,n^{1/3}}
\xlongrightarrow{\Prob}0.
\]
Again, by continuous mapping theorem, we have
\[
\frac{\Lhalf(u_n; v_n) -4Tn}{\sqrt{f(n)}n^{1/3}}\xlongrightarrow{d} \max\left\{\frac{8\xi}{1-c_f^2/T^2} + \sqrt{\frac{8}{1-c_f^2/T^2}}\mathcal{N}(0,1), 0\right\}.
\]
Lastly, for $\lambda < \frac{1}{2T}$, we follow the strategy used to prove Proposition \ref{prop:lpp-one-pair-goal}. Notice that
\begin{equation}
    \begin{split}
       &\Prob\left( L_{\mathrm{half}}^{n,f}(0,0;0,T) \neq  L_{\mathrm{full}}^{n,f}(0,0;0,T)\right)\\
       &\leq \Prob\left(L^\mathrm{b}(u_n; v_n) > L^\mathrm{in}(u_n; v_n)\right) +  \Prob\left(L^{\mathrm{in}, (Tn)^{2/3}}(u_n; v_n) < L^{\mathrm{exit}, \frac{1}{4}f(n)n^{2/3}}(u_n; v_n)\right) 
    \end{split}
\end{equation}
The same proof from Proposition \ref{prop:lpp-one-pair-goal} implies that the second term is bounded by $Ce^{-cf(n)^2}$. For the first term, fix some $a > 0$ that will be chosen later.
\begin{equation}
    \begin{split}
       &\Prob\left(L^\mathrm{b}(u_n; v_n) > L^\mathrm{in}(u_n; v_n)\right)\\
       &\leq  \Prob\left(L^\mathrm{b}(u_n; v_n) -4Tn> -af(n)^2n^{1/3}\right) +\Prob\left(L^\mathrm{in}(u_n; v_n) -4Tn < - af(n)^2n^{1/3}\right).
    \end{split}
\end{equation}
Again, we can apply the same proof from Proposition \ref{prop:lpp-one-pair-goal} and prove that the second term is bounded by $Ce^{-cf(n)^2}$. The first term is equal to
\begin{equation}\label{eq:half-upper}
\Prob\left(L^\mathrm{half}(0,0; Tn+f(n)n^{2/3},Tn-f(n)n^{2/3}) -4Tn> -af(n)^2n^{1/3}\right) 
\end{equation}
by the half-space shift-invariance. We then use the same strategy as in the proof of Proposition \ref{prop:hs_one_point}. By \cite[Corollary~3.2]{dauvergne2026half}, which is the zero-temperature degeneration of \cite[Theorem 1.4]{barraquand2023identity}, $L^\mathrm{half}(0,0; Tn+f(n)n^{2/3},Tn-f(n)n^{2/3}) $ is stochastically dominated by the corresponding last passage time in the full-space stationary model with independent $\operatorname{Exp}(1)$ weights in the bulk and independent $\operatorname{Exp}(1/2+g(n)n^{-1/3})$ and $\operatorname{Exp}(1/2-g(n)n^{-1/3})$ weights on the horizontal and vertical axes, respectively. Let $L^{\mathrm{stat}}$ denote the full-space stationary last passage time defined above. Let 
\[
h_t = \frac{Tn + f(n)n^{2/3}}{t} + \frac{Tn - f(n)n^{2/3}}{1-t}, \qquad \theta_n =g(n)n^{-1/3},  \qquad z_n = 1/2+\theta_n.
\]
By \cite[Lemma 4.6(c)]{emrah2020righttailmoderatedeviationsexponential}, we know that for any $\delta > 0$ and small $\varepsilon >0$
\[
\begin{split}
&\Prob\left(L^{\mathrm{stat}}(0,0; Tn+f(n)n^{2/3}, Tn-f(n)n^{2/3}) - h_n \geq \delta f(n)^2n^{1/3}\right)\\
&\leq 2\exp\left(- \min \left\{\int^{z_n}_{z_n-\varepsilon\theta_n}(\delta f(n)^2n^{1/3} + h_n -h_t)  dt , \int^{z_n + \varepsilon\theta_n}_{z_n}(\delta f(n)^2n^{1/3} + h_n -h_t)  dt\right\}\right).
\end{split}
\]
Since $g(n)=\Theta(f(n))$, the mean value theorem gives
\[
|h_t-h_n|\le C\varepsilon f(n)^2n^{1/3}
\]
uniformly over both integration intervals and for all large enough $n$. Choosing
$C\varepsilon\le\delta/2$, we get
\begin{equation}\label{eq:stat-upper}
\Prob\left(
L^{\mathrm{stat}}(0,0;Tn+f(n)n^{2/3},Tn-f(n)n^{2/3})
-h_n\ge \delta f(n)^2n^{1/3}
\right)
\le 2e^{-c\varepsilon\delta f(n)^3}.
\end{equation} 
Because $\lambda < \frac{1}{2T}$, we know that there exists some $\Delta > 0$ such that for all $n$ large enough
\[
h_n - 4Tn < -\Delta f(n)^2n^{1/3}.
\]
Let $a = \Delta/2$. We upper bound \eqref{eq:half-upper} by $Ce^{-cf(n)^3}$
using stochastic domination and \eqref{eq:stat-upper}.

\end{proof}

\section{Acknowledgments}
The author sincerely thanks their advisor, Ivan Corwin, for his continued
support and guidance. The author is especially grateful to
Sayan Das for suggesting this problem and for many valuable
discussions, to Dominik Schmid for pointing out that the separation
scale could be improved from $n^{2/3+\delta}$ to $n^{2/3}f(n)$ and for
directing the author to the relevant references, and to Milind Hegde
for carefully reading an earlier draft. The author thanks Jiyue Zeng and
Alan Zhao for valuable conversations. The author also thanks the organizers and participants of the SCGP
workshop ``Algebraic Methods in Probability'' for valuable discussions
and feedback. This research was partially
supported by Ivan Corwin's National Science Foundation grant
DMS-2246576 and Simons Investigator in Mathematics Award
MPS-SIM-00929852. 

\section{AI Statement}
The author used AI to assist with English-language editing, figure-drawing, and derivative calculations in the steepest-descent analysis. All suggestions and calculations incorporated into the manuscript were reviewed and verified by the author, who takes full responsibility for the content of this paper.

\begin{appendices}\label{appendix}
\section{Proof of Theorem \ref{thm: Gaussian} and \ref{thm: Gaussian2}}    
The asymptotic analysis of the half-space exponential last passage time relies on the following Fredholm Pfaffian expression \cite[Proposition 1.6]{Pfaffian_Schur}. The original formula involves joint distribution of the passage times, but for simplicity, we only state the proposition in terms of the one-point distribution.

\begin{proposition}
    For any $h>0$ and any positive integers $n > m$, the half space exponential LPP model with diagonal weights of rate $\alpha \in (0, 1/2)$ satisfies the following
\[
\Prob(\Lhalf(0,0;n,m) + \nu^{\mathrm{half}}_{0,0} < h) = \operatorname{Pf}(J - \Kexp)_{\mathbb{L}^2(h, \infty)}.
\]
where $\Kexp$ is a $2 \times 2$-matrix valued skew-symmetric kernel
\begin{equation}
    \Kexp(x,y) = \begin{pmatrix}
        \Kexp_{11}(x,y) & \Kexp_{12}(x,y)\\
        -\Kexp_{12}(y,x) & \Kexp_{22}(x,y)
    \end{pmatrix}
\end{equation}
whose Fredholm Pfaffian is defined by the series expansion 
\[
\operatorname{Pf}(J-\Kexp)_{\mathbb L^2(h, \infty)} = 1+ \sum_{k=1}^\infty \frac{(-1)^k}{k!} \int_{h}^\infty \cdots\int_{h}^\infty \operatorname{Pf}(\Kexp (x_i,x_j))_{i,j = 1}^k dx_1\cdots dx_k.
\]

\end{proposition}

Let $C_a^\theta = \{a+re^{\pm i \theta}: r \geq 0 \}$ and assume that the contour is oriented from $a + \infty e^{-i\theta}$ to $a + \infty e^{i \theta}.$ Each entry of the matrix kernel is defined as the following:
\begin{equation}\label{eq:deformed_kernel}
    \begin{aligned}
        \Kexp_{11}(x,y) &= \int_{C_{1/4}^{\pi/3}} \frac{dz}{2\pi \I} \int_{C_{1/4}^{\pi/3}} \frac{dw}{2\pi\I} \frac{(z-w)e^{-xz-yw}}{(z+w)4zw}(2z+2\alpha - 1)(2w+2\alpha - 1)\frac{(1+2z)^n(1+2w)^n}{(1-2z)^m(1-2w)^m}\\
        \Kexp_{12}(x,y) &= \int_{C_{a_z}^{\pi/3}} \frac{dz}{2\pi \I} \int_{C_{a_w}^{\pi/3}} \frac{dw}{2\pi \I}\frac{(z-w)e^{-xz-yw}}{(z+w)2z}\frac{2z+2\alpha - 1}{2\alpha - 1-2w}\frac{(1+2z)^n(1+2w)^m}{(1-2z)^m(1-2w)^n}\\
        \Kexp_{22}(x,y) &= \int_{C_{c}^{\pi/2}} \frac{dz}{2\pi \I} \int_{C_{c}^{\pi/2}} \frac{dw}{2\pi \I}\frac{(z-w)e^{-xz-yw}}{(z+w)}\frac{1}{2\alpha - 1 -2z}\frac{1}{2\alpha - 1-2w}\frac{(1+2z)^m(1+2w)^m}{(1-2z)^n(1-2w)^n}
    \end{aligned}
\end{equation}
where $a_z \in (0, 1/2), a_z+a_w > 0$, $a_w < \frac{2\alpha-1}{2}$, and $c \in (-\frac12, \frac{2\alpha-1}{2})$. 
\begin{remark}
    The original formula for $\Kexp_{22}$ in \cite{Pfaffian_Schur} is 
\begin{equation}\label{eq:original_contour}
\begin{split}
  &\int_{C_{b_z}^{\pi/3}} \frac{dz}{2\pi \I} \int_{C_{b_w}^{\pi/3}} \frac{dw}{2\pi \I}\frac{(z-w)e^{-xz-yw}}{(z+w)}\frac{1}{2\alpha - 1 -2z}\frac{1}{2\alpha - 1-2w}\frac{(1+2z)^m(1+2w)^m}{(1-2z)^n(1-2w)^n} \\
  &- \frac{e^{\frac{1-2\alpha}{2}y}}{2}\int_{C_{c_z}^{\pi/3}} \frac{dz}{2\pi \I} \frac{(1+2z)^m(2\alpha)^m}{(1-2z)^n(2-2\alpha)^n}\frac{e^{-xz}}{2\alpha-1+2z} + \frac{e^{\frac{1-2\alpha}{2}x}}{2}\int_{C_{c_z}^{\pi/3}} \frac{dz}{2\pi \I} \frac{(1+2z)^m(2\alpha)^m}{(1-2z)^n(2-2\alpha)^n}\frac{e^{-yz}}{2\alpha-1+2z}\\
  & -  \operatorname{sign}(x-y)\int_{C_{c_z}^{\pi/3}} \frac{dz}{2\pi \I} \frac{(1+2z)^m(1-2z)^m}{(1-2z)^n(1+2z)^n}\frac{2ze^{-|x-y|z}}{(2\alpha-1+2z)(2\alpha-1-2z)}
\end{split}
\end{equation}
for some $b_z, b_w \in (0, 1/2-\alpha)$ and $c_z \in (\alpha-1/2, 1/2-\alpha)$. Because $n > m$, we can straighten all the contour from $C^{\pi/3}_{\cdot}$ to $C^{\pi/2}_{\cdot}$. For $x>y$, we first shift the $w$-contour from
$C_{b_w}^{\pi/2}$ to $C_c^{\pi/2}$, crossing the poles at $w=-z$ and $w=\alpha-1/2$, and then shift the
$z$-contour to $C_c^{\pi/2}$, crossing the pole at
$z=\alpha-1/2$. The resulting residue contributions
cancel the three single-integral terms in
\eqref{eq:original_contour}, yielding the expression
for $\Kexp_{22}$ in \eqref{eq:deformed_kernel}.
The case $x<y$ follows shifting the $z$-contour first and then the $w$-contour. For $x=y$, both expressions vanish.
\end{remark}

\begin{proof}[Proof of Theorem \ref{thm: Gaussian}]
    Let $\alpha_n = 1/2 - g(n)n^{-1/3}$ and $\theta_n = g(n)n^{-1/3}$. Let
    \[
    h_n = \frac{M_n}{1/2+\theta_n} + \frac{K_n}{1/2 - \theta_n}, \qquad \tau_n^2 = \frac{K_n}{(1/2-\theta_n)^2} - \frac{M_n}{(1/2+\theta_n)^2}.
    \]
Taylor expansion gives us $\tau_n = \sqrt{8(4Tg(n) - f(n))}n^{1/3}(1+o(1))$ and the requirement for $\frac{g(n)}{f(n)} \rightarrow \lambda \in \left(\frac{1}{4T}, \infty\right]$ tells us that $\tau_n $ is of order $g(n)^{1/2}n^{1/3}$. Let $(u,v) = (h_n + \tau_n x, h_n + \tau_n y)$. Then
    \begin{equation}\label{eq:K11}
    \begin{split}
        \Kexp_{11}(u, v)= & \int_{C_{1/4}^{\pi/3}} \frac{dz}{2\pi \I} \int_{C_{1/4}^{\pi/3}} \frac{dw}{2\pi\I} \frac{(z-w)}{(z+w)4zw}(2z+2\alpha_n - 1)(2w+2\alpha_n - 1)\\
        & \times \exp(n[\psi(z)+\psi(w)] - \tau_n(xz+ yw))
    \end{split}
    \end{equation}
    where 
    \begin{equation}\label{eq: exponent function}
    \psi(z) = -\frac{h_n}{n}z + \frac{M_n}{n} \log(1+2z) - \frac{K_n}{n}\log(1-2z).
    \end{equation}
    One can check that $\psi'(\theta_n) = 0$ and $\psi''(\theta_n) = \tau_n^2/n$. We rescale the kernel as
    \begin{equation}\label{eq:rescaled kernel}
    \widetilde{K}^{\mathrm{exp}}(x,y) = \begin{pmatrix}
        \widetilde{K}^{\mathrm{exp}}_{11}(x,y) & \widetilde{K}^{\mathrm{exp}}_{12}(x,y)\\
         \widetilde{K}^{\mathrm{exp}}_{21}(x,y) & \widetilde{K}^{\mathrm{exp}}_{22}(x,y)\\ 
    \end{pmatrix} := \begin{pmatrix}
        \frac{\tau_n^2 {K}^{\mathrm{exp}}_{11}(u,v)}{\exp(2n\psi(\theta_n) - \tau_n\theta_n(x+y))} & \frac{\tau_n{K}^{\mathrm{exp}}_{12}(u,v)}{\exp(\tau_n\theta_n(y-x))}\\
        \frac{\tau_n{K}^{\mathrm{exp}}_{21}(u,v)}{\exp(\tau_n\theta_n(x-y))} & \frac{{K}^{\mathrm{exp}}_{22}(u,v)}{\exp(-2n\psi(\theta_n) + \tau_n\theta_n(x+y))} 
    \end{pmatrix}.
    \end{equation}
    It is not difficult to check that conjugation will not change the Pfaffian and thus the exponential factors will cancel each other. Moreover, the factor $\tau_n$ is precisely the Jacobian of the change of variables in the Fredholm Pfaffian expansion. Thus,
    \[
    \operatorname{Pf}(J - \widetilde{K}^{\mathrm{exp}})_{\mathbb L^2(h, \infty)} = \operatorname{Pf}(J - {K}^{\mathrm{exp}})_{\mathbb L^2(\tau_nh + h_n, \infty)} = \Prob\left(\frac{\Lhalf(0,0;M_n,K_n) + \nu_{0,0}^{\mathrm{half}}-h_n}{\tau_n} < h \right).
    \]
    By Cauchy's theorem, we can shift the contour of $\widetilde{K}^{\mathrm{exp}}_{11}$ to $C^{\pi/3}_{\theta_n + 1/\tau_n}$ and perform the change of variable $z = \theta_n + \frac{\tilde{z}}{\tau_n}$ and $w = \theta_n + \frac{\tilde{w}}{\tau_n}$.
    \[
    \begin{split}
        \widetilde{K}^{\mathrm{exp}}_{11}(x,y)
        & =  \int_{C_{1}^{\pi/3}} \frac{d\tilde{z}}{2\pi \I} \int_{C_{1}^{\pi/3}} \frac{d\tilde{w}}{2\pi\I} \frac{(\tilde{z}-\tilde{w})}{\tau_n^3(2\theta_n + \frac{\tilde{z}+\tilde{w}}{\tau_n})(\theta_n + \frac{\tilde{z}}{\tau_n})(\theta_n + \frac{\tilde{w}}{\tau_n})}\tilde{z}\tilde{w}\\
        & \qquad  \times \exp\left(n\left[\psi\left(\theta_n + \frac{\tilde{z}}{\tau_n}\right)-\psi(\theta_n)+\psi\left(\theta_n + \frac{\tilde{w}}{\tau_n}\right)-\psi(\theta_n)\right] - (x\tilde{z}+ y\tilde{w})\right)
    \end{split}
    \]
    The claim is that there exists some $\varepsilon, c, C>0$ such that for all sufficiently large $n$, 
    \[
    \mathfrak{R}\left[n\psi\left(\theta_n + \frac{1}{\tau_n}+ te^{\pm \I\pi/3}\right) - n\psi(\theta_n)\right] \leq 
    \begin{cases}
     C-c\tau_n^2t^2,&0\le t\le\varepsilon,\\ C-cn t,&t\ge\varepsilon. 
    \end{cases}
    \]
    For $ 0 \leq t \leq \varepsilon$, notice that 
    \[
    \begin{split}
    &\mathfrak{R}[\psi(\theta_n + \tau_n^{-1} + te^{\pm \I\pi/3 }) - \psi(\theta_n)] \\
    &= \psi(\theta_n+\tau_n^{-1}) - \psi(\theta_n) + \frac{t}{2}\psi'(\theta_n +\tau_n^{-1}) -\frac{\psi''(\theta_n + \tau_n^{-1})}{4}t^2 - \frac{t^3}{6}\mathfrak{R}[\psi'''(\theta_n + \tau_n^{-1} + t_0e^{\pm \I\pi/3 })]
    \end{split}
    \]
    where $t_0 \in (0, t)$. Since $\psi'(\theta_n) = 0, n\psi''(\theta_n) = \tau_n^2,$ and $n/\tau_n^3 \rightarrow 0$, Taylor expansion at $\theta_n$ gives
    \[
    n[\psi(\theta_n + \tau_n^{-1}) - \psi(\theta_n)] = 1/2 + O(n/\tau_n^3) \leq C, \qquad  n\psi'(\theta_n + \tau_n^{-1}) = \tau_n + O(n/\tau_n^2) \leq C\tau_n
    \]
    for some $C>0$ and all $n$ sufficiently large. Because
    \[
    \psi'''(z) = \frac{16M_n/n}{(1+2z)^3} + \frac{16K_n/n}{(1-2z)^3}
    \]
    and $M_n/n, K_n/n \rightarrow T >0$, $\theta_n \rightarrow 0, \tau_n \rightarrow \infty$, we can choose $\varepsilon > 0$ so that for all $n$ large enough and $t_0 \in (0, \varepsilon)$
    \[
    \mathfrak{R}\psi'''(\theta_n +\tau_n^{-1}+ t_0e^{\pm \I \pi/3}) > 0 .
    \]
Since $\psi'''$ is positive on the real interval $(-1/2,1/2)$, $n\psi''(\theta_n+\tau_n^{-1}) \geq n\psi''(\theta_n)=\tau_n^2.$ Therefore, for all $0 \leq t \leq \varepsilon$ and $n$ large enough
    \[
    \mathfrak{R}[n\psi(\theta_n + \tau_n^{-1} + te^{\pm \I\pi/3}) - n\psi(\theta_n)] \leq  C+ C\tau_nt-\frac{\tau_n^2}{4}t^2 \leq C - \frac{\tau_n^2}{8}t^2.
    \]
    On the other hand,
    \[
    \frac{d}{dt} \mathfrak R \psi(\theta_n + \tau_n^{-1} + te^{\pm \I \pi/3}) = \mathfrak R \left[e^{\pm \I \pi/3}\left(-\frac{h_n}{n} + \frac{2M_n/n}{1+2(\theta_n+\tau_n^{-1}) + 2te^{\pm \I \pi/3}}+ \frac{2K_n/n}{1-2(\theta_n+\tau_n^{-1}) - 2te^{\pm \I \pi/3}}\right)\right].
    \]
    Because $|1-2te^{\pm \I \pi/3}| \geq \frac{\sqrt{3}}{2}$ and $|1-2(\theta_n+\tau_n^{-1}) -2te^{\pm \I \pi/3}| \geq \frac{\sqrt{3}}{2}(1-2(\theta_n+\tau_n^{-1}))$ for all $t \geq 0$, we have
    \[
    \mathfrak R \left[e^{\pm \I \pi/3}\left(-\frac{h_n}{n} + \frac{2M_n/n}{1+2(\theta_n+\tau_n^{-1}) + 2te^{\pm \I \pi/3}}+ \frac{2K_n/n}{1-2(\theta_n+\tau_n^{-1}) - 2te^{\pm \I \pi/3}}\right)\right]  \rightarrow  -\frac{16Tt^2(1+2t^2)}{1+4t^2+16t^4}
    \]
    uniformly for all $t \geq 0$ when $n \rightarrow \infty$. Thus, for $t \geq \varepsilon/2$, there exists some $c > 0$ such that for all $n$ sufficiently large, 
    \[
    \frac{d}{dt} \mathfrak R \psi(\theta_n + \tau_n^{-1} + te^{\pm \I \pi/3}) \leq -c.
    \]
    Since we have already shown $\mathfrak R [\psi(\theta_n + \tau_n^{-1}+\frac{\varepsilon}{2}e^{\pm \I \pi/3}) - \psi(\theta_n) ]\leq 0$ for $n$ large enough, for $t \geq \varepsilon$, we have
    \[
    \mathfrak R [\psi(\theta_n + \tau_n^{-1} + t e^{\pm \I \pi/3}) - \psi(\theta_n) ] \leq -c \left(t-\frac{\varepsilon}{2}\right) \leq -\frac{c}{2}t.
    \]

Let $\tilde z=1+r_1e^{\pm\I\pi/3}$ and
$\tilde w=1+r_2e^{\pm\I\pi/3}$, where $r_1,r_2\geq0$.
Fix $L\in\mathbb R$. Using the preceding bounds and $n/\tau_n\to\infty$, for $x,y\geq L$
the absolute value of the integrand of
$(\tau_n\theta_n)^3\widetilde K_{11}^{\mathrm{exp}}(x,y)$
is bounded by the following for $n$ large enough
\[
\begin{split}
&Ce^{-x-y}(r_1+r_2)(1+r_1)(1+r_2)
\left(e^{-cr_1^2}+e^{-cr_1n/\tau_n}\right)\left(e^{-cr_2^2}+e^{-cr_2n/\tau_n}\right)e^{-xr_1/2}e^{-yr_2/2}\\
&\leq Ce^{-x-y}(r_1+r_2)(1+r_1)(1+r_2)
\left(e^{-cr_1^2 + |L|r_1/2}+e^{-r_1}\right)\left(e^{-cr_2^2 + |L|r_2/2}+e^{-r_2}\right)
\end{split}
\]
which is integrable on $[0,\infty)^2$. Consequently, there exists $C_L>0$ such that, for all sufficiently
large $n$ and all $x,y\geq L$,
\[
\left|\widetilde K_{11}^{\mathrm{exp}}(x,y)\right|
\leq \frac{C_L}{(\tau_n\theta_n)^3}e^{-x-y}.
\]
Since $\tau_n\theta_n\to\infty$, this yields
$\widetilde K_{11}^{\mathrm{exp}}(x,y)\to0$.

For $\Kexp_{22}$, we have
\[
    \begin{split}
        \Kexp_{22}(u, v)= & \int_{C_{c}^{\pi/2}} \frac{dz}{2\pi \I} \int_{C_{c}^{\pi/2}} \frac{dw}{2\pi\I} \frac{(z-w)}{(z+w)}\frac{1}{2\alpha_n - 1-2z}\frac{1}{2\alpha_n-2w - 1}\\
        & \times \exp(n[\kappa(z)+\kappa(w)] - \tau_n(xz+ yw))
    \end{split}
    \]
where $c \in (-\frac{1}{2}, -\theta_n)$ and 
\[
\kappa(z) = -\frac{h_n}{n}z + \frac{K_n}{n} \log (1+2z) - \frac{M_n}{n}\log(1-2z) = -\psi(-z).
\]
Thus, $\kappa'(-\theta_n) = 0$ and $\kappa''(-\theta_n) = -\tau_n^2/n$. Let $\eta_n = \frac{f(n)}{8Tg(n)}$, which is less than $1/2$ for all $n$ sufficiently large. Let $t_n = \frac{2\sqrt{3}}{3}[(1-\eta_n)\theta_n + (2\tau_n)^{-1}]$ and consider the contour
\[
\Gamma_n := \left\{-\theta_n - (2\tau_n)^{-1} + te^{\pm \I \pi/6} : 0 \leq t \leq t_n \right\} \cup {\Gamma}_n^{\mathrm{vert}}, \qquad {\Gamma}_n^{\mathrm{vert}} := \{-\eta_n \theta_n \pm is : s \geq t_n/2\}.
\]
By Cauchy's theorem, we deform the contour from $C_c^{\pi/2}$ to $\Gamma_n$, and perform the change of variable $z = -\theta_n + \frac{\tilde{z}}{\tau_n}$ and $w = -\theta_n + \frac{\tilde{w}}{\tau_n}$. Let 
\[
\widetilde{\Gamma}_n = \left\{-1/2 + te^{\pm \I \pi/6} : 0 \leq t \leq \tau_nt_n \right\} \cup \widetilde{\Gamma}_n^{\mathrm{vert}}, \qquad \widetilde{\Gamma}_n^{\mathrm{vert}} =  \{(1-\eta_n)\tau_n \theta_n \pm is : s \geq \tau_nt_n/2\}.
\]
Then,
\begin{equation}\label{eq:K22}
    \begin{split}
        \widetilde{K}^{\mathrm{exp}}_{22}(x,y)&=\int_{\widetilde{\Gamma}_n} \frac{d\tilde{z}}{2\pi \I} \int_{\widetilde{\Gamma}_n} \frac{d\tilde{w}}{2\pi\I} \frac{( \tilde{z}-\tilde{w})}{(-2\theta_n\tau_n+\tilde{z}+\tilde{w})} \frac{1}{4\tilde{z}\tilde{w}}\\
        & \qquad  \times \exp\left(n\left[\kappa\left(-\theta_n + \frac{\tilde{z}}{\tau_n}\right)-\kappa(-\theta_n)+\kappa\left(-\theta_n + \frac{\tilde{w}}{\tau_n}\right)-\kappa(-\theta_n)\right] - (x\tilde{z}+ y\tilde{w})\right).
    \end{split}
    \end{equation}
Similar as above,
    \[
    \begin{split}
    &\mathfrak{R}[\kappa(-\theta_n -(2\tau_n)^{-1} + te^{\pm \I\pi/6 }) - \kappa(-\theta_n)] =\mathfrak{R}[\psi(\theta_n) - \psi(\theta_n +(2\tau_n)^{-1} - te^{\pm \I\pi/6 })] \\
    &= \psi(\theta_n) - \psi(\theta_n+(2\tau_n)^{-1}) + \frac{\sqrt{3}}{2}t\psi'(\theta_n+(2\tau_n)^{-1}) - \frac{t^2}{4}\psi''(\theta_n+(2\tau_n)^{-1}) \\&
    \qquad - \frac{t^4}{24}\mathfrak R[{e^{\pm \I 2\pi/3}}\psi''''(\theta_n+(2\tau_n)^{-1} - t_0e^{\pm \I \pi/6})]
    \end{split}
    \]
    where $t_0 \in (0, t)$. Since $\psi'(\theta_n) = 0, n\psi''(\theta_n) = \tau_n^2,$ and $n/\tau_n^3 \rightarrow 0$, Taylor expansion at $\theta_n$ gives
    \[
    n[ \psi(\theta_n)- \psi(\theta_n + (2\tau_n)^{-1})] = -1/8 + O(n/\tau_n^3) \leq C, \quad  n\psi'(\theta_n + (2\tau_n)^{-1}) = \tau_n/2 + O(n/\tau_n^2) \leq C\tau_n.
    \]
    for some $C>0$ and all $n$ sufficiently large. Since $\psi'''$ is positive on the real interval $(-1/2,1/2)$, $n\psi''(\theta_n+(2\tau_n)^{-1}) \geq n\psi''(\theta_n)=\tau_n^2.$ Lastly, because $t_n \rightarrow 0$, $|\psi''''(\theta_n+(2\tau_n)^{-1} - t_0e^{\pm \I \pi/6})|$ is bounded by some constant $C$ for all $n$ large enough and 
    \[
    Cnt^4 = C \tau_n^2t^2 \frac{nt^2}{\tau_n^2} \leq C \tau_n^2t^2 \frac{nt_n^2}{\tau_n^2}, \qquad \frac{nt_n^2}{\tau_n^2} = O(\theta_n) \rightarrow 0.
    \]
    Therefore, we can absorb the quartic term into the quadratic term and conclude that for all $n$ large enough and $0 \leq t \leq t_n$,
    \begin{equation}\label{eq:K22 local bound}
    \mathfrak{R}[n\kappa(-\theta_n -(2 \tau_n)^{-1} + te^{\pm \I\pi/6}) - n\kappa(-\theta_n)] \leq  C+ C\tau_nt-\frac{\tau_n^2}{4}t^2 + C\tau_n^2t^2 \frac{nt_n^2}{\tau_n^2} \leq C - \frac{\tau_n^2}{8}t^2.
    \end{equation}
On the other hand, for $s \geq 0$,
\[
\begin{split}
\frac{d}{ds}\mathfrak{R}[n\kappa(-\eta_n \theta_n \pm is)] &= 4s\left[\frac{K_n}{(1-2\eta_n\theta_n)^2 + 4s^2}-\frac{M_n}{(1+2\eta_n\theta_n)^2 + 4s^2}\right]\\& = -4s \frac{2f(n)n^{2/3}(1/2+4\eta_n^2 \theta_n^2 + 4s^2)}{((1-2\eta_n\theta_n)^2 + 4s^2)((1+2\eta_n\theta_n)^2 + 4s^2)}.
\end{split}
\]
As $\theta_n \rightarrow 0$, we know that there exist some $c, C> 0$ such that for all $n$ large enough, 
\[
1/2+4\eta_n^2 \theta_n^2 + 4s^2\geq c (1+s^2), \quad  ((1-2\eta_n\theta_n)^2 + 4s^2)((1+2\eta_n\theta_n)^2 + 4s^2) \leq C(1+s^2)^2.
\]
Thus, for some $c>0$ and all $n$ large enough,
\[
\frac{d}{ds}\mathfrak{R}[n\kappa(-\eta_n \theta_n \pm is)] \leq -cf(n)n^{2/3}\frac{s}{1+s^2}.
\]
Because 
\[
    \mathfrak{R}[n\kappa(-\theta_n -(2 \tau_n)^{-1} + t_ne^{\pm \I\pi/6}) - n\kappa(-\theta_n)] \leq C - \frac{\tau_n^2}{8}t_n^2,
    \]
we have for $s \geq t_n/2$,
   \[
    \mathfrak{R}[n\kappa(-\eta_n \theta_n \pm is) - n\kappa(-\theta_n)] \leq C - c \tau_n^2 \theta_n^2 - c f(n)n^{2/3} \log\left(\frac{1+s^2}{1+t_n^2/4}\right).
    \] 
Fix $L \in \mathbb R$. Thus, for $\tilde{z} = (1-\eta_n)\tau_n\theta_n \pm i s$ for $s \geq t_n\tau_n/2$ and $\tilde{w} \in \widetilde{\Gamma}_n$, there exists some $C_L, C, c > 0$ such that for all $n$ large enough and $x, y \geq L$,
\begin{equation}\label{eq:K_22bound}
\begin{aligned}
\left|\frac{(\tilde{z}-\tilde{w})}{4\tilde{z}\tilde{w}(-2\tau_n\theta_n + \tilde{z}+\tilde{w})} \right| &\leq  \frac{C}{\tau_n\theta_n},\\
\left|e^{n \kappa(-\theta_n + \tilde{z}\tau_n^{-1})-n\kappa(-\theta_n) - x\tilde z}\right| &\leq C\frac{e^{x(\eta_n-1)\tau_n\theta_n}}{e^{  c(\tau_n\theta_n)^2}}\left(\frac{1+s^2/\tau_n^2}{1+t_n^2/4}\right)^{-cf(n)n^{2/3}} \\&\leq C_L e^{x/2 - c(\tau_n\theta_n)^2}\left(\frac{1+s^2/\tau_n^2}{1+t_n^2/4}\right)^{-cf(n)n^{2/3}}.
\end{aligned}
\end{equation}
It is not difficult to check that  
\[
\int_{\tau_nt_n/2}^\infty \left(\frac{1+s^2/\tau_n^2}{1+t_n^2/4}\right)^{-cf(n)n^{2/3}} ds \leq C
\]
for some constant $C>0$ for all $n$ large enough. Combining this with upper bounds from \eqref{eq:K_22bound} and \eqref{eq:K22 local bound}, the contribution from $\widetilde\Gamma_n^{\mathrm{vert}}\times\widetilde\Gamma_n$ is bounded in absolute value by
\[
\frac{C_L}{\tau_n\theta_n}e^{\frac{x+y}{2}-c(\tau_n\theta_n)^2}.
\]
Since $\tau_n\theta_n\to\infty$, this contribution vanishes for every fixed $x,y \geq L$. Similarly, the contribution from $\widetilde\Gamma_n \times \widetilde\Gamma_n^{\mathrm{vert}}$ vanishes. Let $\tilde z=-1/2+r_1e^{\pm\I\pi/6}$ and
$\tilde w=-1/2+r_2e^{\pm\I\pi/6}$, where $r_1,r_2\geq0$.
For $x,y\geq L$,
the absolute value of the integrand of
$\tau_n\theta_n\widetilde K_{22}^{\mathrm{exp}}(x,y)$
is bounded by the following for all $n$ large enough
\[
\begin{split}
&C_Le^{(x+y)/2 - cr_1^2 -cr_2^2}
\end{split}
\]
which is integrable on $[0,\infty)^2$. Consequently, 
\[
\left|\widetilde K_{22}^{\mathrm{exp}}(x,y)\right|
\leq \frac{C_L}{\tau_n\theta_n}e^{(x+y)/2} \rightarrow 0.
\]
For $\Kexp_{12}$, we have
\[
    \begin{split}
        \Kexp_{12}(u, v)= & \int_{C_{a_z}^{\pi/3}} \frac{dz}{2\pi \I} \int_{C_{a_w}^{\pi/3}} \frac{dw}{2\pi\I} \frac{(z-w)}{2z(z+w)}\frac{2z+2\alpha_n - 1}{2\alpha_n-2w - 1}\exp(n[\psi(z)+\kappa(w)] - \tau_n(xz+ yw))
    \end{split}
    \]
where $a_z \in (0, 1/2)$, $a_z + a_w > 0$ and $a_w < -\theta_n$. Due to the same reason as for $\Kexp_{11}(u, v)$ and $\Kexp_{22}(u, v)$, we can deform the contour $C_{a_z}^{\pi/3}$ to $C_{\theta_n + 1/\tau_n}^{\pi/3}$ and $C_{a_w}^{\pi/3}$ to $\Gamma_n$ and do the same change of variable $z = \theta_n + \frac{\tilde{z}}{\tau_n}$ and $w = -\theta_n + \frac{\tilde{w}}{\tau_n}$. Hence
\begin{equation}\label{eq:K12}
    \begin{split}
        \widetilde{K}^{\mathrm{exp}}_{12}(x,y)
        &=  -\int_{C_1^{\pi/3}} \frac{d\tilde{z}}{2\pi \I} \int_{\widetilde{\Gamma}_n} \frac{d\tilde{w}}{2\pi\I}  \frac{( 2\theta_n\tau_n + \tilde{z}-\tilde{w})}{(\tilde{z}+\tilde{w})(\theta_n\tau_n + \tilde{z})} \frac{\tilde{z}}{2\tilde{w}}\\
        & \qquad  \times \exp\left(n\left[\psi\left(\theta_n + \frac{\tilde{z}}{\tau_n}\right)-\psi(\theta_n)+\kappa\left(-\theta_n + \frac{\tilde{w}}{\tau_n}\right)-\kappa(-\theta_n)\right] - (x\tilde{z}+ y\tilde{w})\right).
    \end{split}
    \end{equation}
There exists some $C > 0$ such that for all $\tilde{z} \in C_1^{\pi/3}$ and $\tilde{w} \in \widetilde{\Gamma}_n$ and $n$ large enough
\begin{equation}\label{eq:K12prefactor}
\left| \frac{\tilde{z}(2\theta_n\tau_n + \tilde{z}-\tilde{w})}{2\tilde{w}(\tilde{z}+\tilde{w})(\theta_n\tau_n + \tilde{z})}\right| = \left| \frac{1}{\tilde{z} + \tilde{w}} - \frac{1}{\tilde{w}} + \frac{\tilde{z}}{2(\theta_n\tau_n + \tilde{z})\tilde{w}}\right| \leq C.
\end{equation}
For each fixed $L\in\mathbb R$ and $x,y>L$, the contribution from $\widetilde\Gamma_n^{\mathrm{vert}}$ vanishes as $n\to\infty$,
by the estimates established for $\widetilde K_{22}^{\mathrm{exp}}$ in \eqref{eq:K_22bound}.
On the remaining contours, the integrand bounds used for
$\widetilde K_{11}^{\mathrm{exp}}$ and
$\widetilde K_{22}^{\mathrm{exp}}$, together with \eqref{eq:K12prefactor}, justify dominated convergence
of \eqref{eq:K12}:
\[
\begin{split}
\lim_{n \rightarrow \infty} \widetilde K_{12}^{\mathrm{exp}}(x,y) &= - 
\int_{C_{1}^{\pi/3}} \frac{d\tilde{z}}{2\pi \I} \int_{C_{-1/2}^{\pi/6}} \frac{d\tilde{w}}{2\pi\I} \frac{\tilde{z}}{\tilde{w}(\tilde{z}+\tilde{w})}e^{\tilde{z}^2/2 - \tilde{w}^2/2-x\tilde{z} - y\tilde{w}}\\
&= 
\int_{C_{1}^{\pi/3}} \frac{d\tilde{z}}{2\pi \I} e^{\tilde{z}^2/2-x\tilde{z} }\\
&=  \frac{1}{\sqrt{2\pi}}e^{-x^2/2}
\end{split}
\]
In the second equality, we close the $C_{-1/2}^{\pi/6}$ contour to the right and the integral on the closing arc vanishes due to the Gaussian decay. Since the resulting contour is clockwise and encloses only one pole $\tilde{w} = 0$, we can replace the inner integral by the residue computation.

Moreover, the same estimates give
\begin{equation}
\bigl|\widetilde K_{12}^{\mathrm{exp}}(x,y)\bigr|
\le C_L e^{-x+y/2},
\qquad \text{ for all }x,y \geq L,
\end{equation}
uniformly for all sufficiently large $n$. 

Finally, by \cite[Lemma 2.5]{Pfaffian_Schur}, we know that the kernel convergence implies the Pfaffain convergence and hence
\[
\lim_{n \rightarrow \infty} \Prob \left(\frac{\Lhalf(0,0;M_n, K_n) - h_n}{\tau_n} \leq x\right) =  \int_{-\infty}^x \frac{1}{\sqrt{2\pi}} e^{-y^2/2}
dy\]
as all the terms of order at least two in the Fredholm Pfaffian expansion vanish.
\end{proof}

\begin{proof}[Proof of Theorem \ref{thm: Gaussian2}]
The proof structure is similar to the proof of Theorem \ref{thm: Gaussian} but the choices for steep descent contours are different. Notice that $\lim_{n \rightarrow \infty}\tau_n^2/n > 0$ is equivalent to
\begin{equation}\label{eq:lambda}
\frac{T - \sqrt{T^2-c_f^2}}{2c_f^2} < \lambda <\frac{T + \sqrt{T^2-c_f^2}}{2c_f^2}.
\end{equation}
The condition that $c_g  \in (0, 1/2)$ guarantees $\lambda \in \left(0,\frac{1}{2 c_f}\right)$ and thus the upper bound in \eqref{eq:lambda}. Hence, we need
\[
\lambda > \frac{1}{2\left(T + \sqrt{T^2 - c_f^2}\right)}.
\]
In this case, $\tau_n$ is of order $n^{1/2}$. 

We deform the contours of $\Kexp_{11}$ from $C_{1/4}^{\pi/3}$ to $C_{1/4}^{\phi}$ where 
\[
\phi \in (\pi/4, \pi/3), \qquad 
0<-\cos(2\phi)
<\frac{4Tc_g-c_f(1+4c_g^2)}{4c_g(T-2c_fc_g)}.
\]
We perform the same change of variables $(u,v) = (h_n + \tau_n x , h_n+\tau_n y)$ and the formulas for $\Kexp_{11}$ and $\psi$ remain the same as in \eqref{eq:K11} and \eqref{eq: exponent function} except the contours are now $C_{1/4}^{\phi} \times C_{1/4}^{\phi}$. We also rescale the kernel the same way as in \eqref{eq:rescaled kernel} and shift the contour of $\widetilde{K}_{11}^{\mathrm{exp}}$ to $C^{\phi}_{\theta_n + 1/\tau_n}$ and do the same change of variables $z = \theta_n + \frac{\tilde{z}}{\tau_n}$ and $w = \theta_n + \frac{\tilde{w}}{\tau_n}$. We still like to prove that there exists some $\varepsilon, c, C > 0$ such that for all sufficiently large $n$,
\[    \mathfrak{R}\left[n\psi\left(\theta_n + \frac{1}{\tau_n}+ te^{\pm \I\phi}\right) - n\psi(\theta_n)\right] \leq
    \begin{cases}
     C-c\tau_n^2t^2,&0\le t\le\varepsilon,\\ C-cn t,&t\ge\varepsilon. 
    \end{cases}
\]
 For $ 0 \leq t \leq \varepsilon$, notice that 
    \[
    \begin{split}
    &\mathfrak{R}[\psi(\theta_n + \tau_n^{-1} + te^{\pm \I\phi }) - \psi(\theta_n)] \\
    &= \psi(\theta_n+\tau_n^{-1}) - \psi(\theta_n) + t\cos(\phi)\psi'(\theta_n +\tau_n^{-1}) +\frac{\psi''(\theta_n + \tau_n^{-1})}{2}t^2\cos(2\phi) \\
    & \qquad + \frac{t^3}{6}\mathfrak{R}[e^{\pm 3\I \phi}\psi'''(\theta_n + \tau_n^{-1} + t_0e^{\pm \I\phi })]
    \end{split}
    \]
    where $t_0 \in (0, t)$. Since $\psi'(\theta_n) = 0, n\psi''(\theta_n) = \tau_n^2,$ and $n/\tau_n^3 \rightarrow 0$, Taylor expansion at $\theta_n$ gives
    \[
    n[\psi(\theta_n + \tau_n^{-1}) - \psi(\theta_n)] = 1/2 + O(n/\tau_n^3) \leq C, \qquad  n\psi'(\theta_n + \tau_n^{-1}) = \tau_n + O(n/\tau_n^2) \leq C\tau_n
    \]
    for some $C>0$ and all $n$ sufficiently large. Because $\tau_n^{-1} \rightarrow 0$ and 
    \[
    \psi'''(\theta_n) \rightarrow \frac{16(T+c_f)}{(1+2c_g)^3} + \frac{16(T-c_f)}{(1-2c_g)^3}>0,
    \]
    we can choose $\varepsilon > 0$ so that for all $n$ large enough and $t_0 \in (0, \varepsilon)$
    \[
    \mathfrak{R}[e^{\pm 3\I \phi}\psi'''(\theta_n +\tau_n^{-1}+ t_0e^{\pm \I \phi})] < 0 .
    \]
    Since $\psi'''$ is positive on the real interval $(-1/2,1/2)$, $n\psi''(\theta_n+\tau_n^{-1}) \geq n\psi''(\theta_n)=\tau_n^2$. Since $\cos(2\phi) <0$, for all $0 \leq t \leq \varepsilon$ and $n$ large enough
    \[
    \mathfrak{R}[n\psi(\theta_n + \tau_n^{-1} + te^{\pm \I\phi}) - n\psi(\theta_n)] \leq  C+ C\tau_nt-2c{\tau_n^2}t^2 \leq C - c{\tau_n^2}t^2.
    \]
    On the other hand,
    \[
    \frac{d}{dt} \mathfrak R \psi(\theta_n + \tau_n^{-1} + te^{\pm \I \phi}) = \mathfrak R \left[e^{\pm \I \phi}\left(-\frac{h_n}{n} + \frac{2M_n/n}{1+2(\theta_n+\tau_n^{-1}) + 2te^{\pm \I \phi}}+ \frac{2K_n/n}{1-2(\theta_n+\tau_n^{-1}) - 2te^{\pm \I \phi}}\right)\right].
    \]
    Because $|1-2te^{\pm \I \phi}| \geq \sin(\phi)$ and $|1-2(\theta_n+\tau_n^{-1}) -2te^{\pm \I \phi}| \geq \sin(\phi)(1-2(\theta_n+\tau_n^{-1}))$ for all $t \geq 0$, we have
    \begin{equation}\label{eq:Dt}
    \begin{split}
    &\mathfrak R \left[e^{\pm \I \phi}\left(-\frac{h_n}{n} + \frac{2M_n/n}{1+2(\theta_n+\tau_n^{-1}) + 2te^{\pm \I \phi}}+ \frac{2K_n/n}{1-2(\theta_n+\tau_n^{-1}) - 2te^{\pm \I \phi}}\right)\right] \rightarrow\\
    & - \cos(\phi) \frac{4T - 8c_gc_f}{1-4c_g^2} + \frac{(2T + 2c_f)((1+2c_g)\cos(\phi) + 2t)}{(1+2c_g)^2 + 4(1+2c_g)t\cos(\phi) + 4t^2}+ \frac{(2T -2c_f)((1-2c_g)\cos(\phi) -2t)}{(1-2c_g)^2 - 4(1-2c_g)t\cos(\phi) + 4t^2}
    \end{split}
    \end{equation}
    uniformly for all $t \geq 0$ when $n \rightarrow \infty$. Let us denote the function on the RHS of \eqref{eq:Dt} as $D(t)$. The choice of $\phi$ guarantees that $D(t) < 0$ for all $t > 0$. It is also not hard to check that when 
    \[
    \lim_{t \rightarrow \infty} D(t) = - \cos(\phi) \frac{4T - 8c_gc_f}{1-4c_g^2} < 0.
    \]
    Thus, for $t \geq \varepsilon/2$, there exists some $c > 0$ such that for all $n$ sufficiently large, 
    \[
    \frac{d}{dt} \mathfrak R \psi(\theta_n + \tau_n^{-1} + te^{\pm \I \phi}) \leq -c.
    \]
    Then for all $n$ large enough and $t \geq \varepsilon$, we have
    \[
    \mathfrak R [\psi(\theta_n + \tau_n^{-1} + t e^{\pm \I\phi}) - \psi(\theta_n) ]  \leq -\frac{c}{2}t.
    \]
    Let $\tilde z=1+r_1e^{\pm\I\phi}$ and
$\tilde w=1+r_2e^{\pm\I\phi}$, where $r_1,r_2\geq0$.
Fix $L\in\mathbb R$. By the preceding bounds, for $x,y\geq L$ and all $n$ sufficiently large,
the absolute value of the integrand of
$(\tau_n\theta_n)^3\widetilde K_{11}^{\mathrm{exp}}(x,y)$
is bounded
\[
\begin{split}
&Ce^{-x-y}(r_1+r_2)(1+r_1)(1+r_2)
\left(e^{-cr_1^2}+e^{-cr_1n/\tau_n}\right)\left(e^{-cr_2^2}+e^{-cr_2n/\tau_n}\right)e^{-xr_1\cos(\phi)}e^{-yr_2\cos(\phi)}\\
&\leq Ce^{-x-y}(r_1+r_2)(1+r_1)(1+r_2)
\left(e^{-cr_1^2 + |L|r_1\cos(\phi)}+e^{-r_1}\right)\left(e^{-cr_2^2 + |L|r_2\cos(\phi)}+e^{-r_2}\right)
\end{split}
\]
which is integrable on $[0,\infty)^2$. The rest of the estimates for $\widetilde{K}_{11}^{\mathrm{exp}}$ follow exactly from the proof of Theorem \ref{thm: Gaussian}.

Let 
\[
t_n = \frac{2\sqrt{3}}{3}[\theta_n - \chi_n + (2\tau_n)^{-1}],  \qquad \chi_n = \frac{M_n-K_n}{2(\sqrt{M_n} +\sqrt{K_n})^2}.
\]
For $\Kexp_{22}$, we change the contour to be 
\[
\Gamma_n := \left\{-\theta_n - (2\tau_n)^{-1} + te^{\pm \I \pi/6} : 0 \leq t \leq t_n \right\} \cup {\Gamma}_n^{\mathrm{vert}}, \qquad {\Gamma}_n^{\mathrm{vert}} := \{-\chi_n \pm is : s \geq t_n/2\}.
\]
In this case,
\begin{equation}\label{eq:K_22secondderivative}
\begin{split}
&\frac{d^2}{dt^2}\mathfrak{R}[\kappa(-\theta_n -(2 \tau_n)^{-1} + te^{\pm \I\pi/6}) ]\\ &= \frac{4M_n}{n}\mathfrak{R}\left(\frac{e^{i\pi/3}}{(1+2\theta_n+1/\tau_n -2te^{i\pi/6})^2}\right) -\frac{4K_n}{n}\mathfrak{R}\left(\frac{e^{i\pi/3}}{(1-2\theta_n-1/\tau_n +2te^{i\pi/6})^2}\right)\\
&\rightarrow \frac{2(T+c_f)((1+2c_g)^2 - 4\sqrt{3}t(1+2c_g) + 8t^2)}{((1+2c_g)^2-2\sqrt{3}t(1+2c_g) + 4t^2)^2}- \frac{2(T-c_f)((1-2c_g)^2 + 4\sqrt{3}t(1-2c_g) + 8t^2)}{((1-2c_g)^2+2\sqrt{3}t(1-2c_g) + 4t^2)^2}
\end{split} 
\end{equation}
uniformly over all $t \geq 0$. Let 
\[
\chi_* = \lim_{n \rightarrow \infty} \chi_n = \frac{c_f}{2\left(T + \sqrt{T^2 - c_f^2}\right)}, \qquad t_* = \lim_{n \rightarrow \infty}t_n =  \frac{2\sqrt{3}}{3} (c_g - \chi_*).
\]
The conditions on $\lambda$ implies $c_g - \chi_* >0$. Let $D(t)$ denote the limit on the RHS of \eqref{eq:K_22secondderivative}. For all $0 < t \leq t_{*}$, we have
\[
D(t)  < \frac{2(T+c_f)}{(1+2c_g-\sqrt{3}t)^2} - \frac{2(T-c_f)}{(1-2c_g+\sqrt{3}t)^2} \leq \frac{2(T+c_f)}{(1+2\chi_*)^2} - \frac{2(T-c_f)}{(1-2\chi_*)^2} = 0.
\]
Also, 
\[
D(0) = \frac{2(T+c_f)}{(1+2c_g)^2}- \frac{2(T-c_f)}{(1-2c_g)^2} < 0.
\]
By continuity and compactness, we know that there exists some $\delta > 0$ and $c> 0 $ such that 
for $t \in [0, t_* + \delta]$, $D(t) \leq -2c$. Thus, for all large enough $n$ and $t \in [0, t_n]$,
\[
\frac{d^2}{dt^2}\mathfrak{R}[\kappa(-\theta_n -(2 \tau_n)^{-1} + te^{\pm \I\pi/6}) ] \leq -c.
\]
Consequently, for $0 \leq t \leq t_n$ and $n$ large enough
\[
\begin{split}
&\mathfrak{R}[n\kappa(-\theta_n -(2 \tau_n)^{-1} + te^{\pm \I\pi/6}) - n\kappa(-\theta_n)]\\
&= n\psi(\theta_n) - n\psi(\theta_n + (2\tau_n)^{-1}) + \frac{nt\sqrt{3}}{2}\psi'(\theta_n + (2\tau_n)^{-1}) + \frac{n t^2}{2} \frac{d^2}{dt^2}\mathfrak{R}[\kappa(-\theta_n -(2 \tau_n)^{-1} + t_0e^{\pm \I\pi/6}) ]\\
&\leq C+ Cn^{1/2}t - \frac{c}{2}nt^2 \leq C - c'nt^2.
\end{split}
\]
where $t_0 \in (0, t)$. For $s \geq t_n/2$, we have 
\[
\frac{d}{ds} \mathfrak{R}[n \kappa(-\chi_n \pm is)] = -\frac{16(M_n-K_n)s^3}{((1-2\chi_n)^2 + 4s^2)((1+2\chi_n)^2 + 4s^2)} \leq -cn \frac{s^3}{(1+s^2)^2} \leq -c'n \frac{s}{1+s^2}
\]
for some $c, c' > 0$. The remaining estimates for $\widetilde K_{22}^{\mathrm{exp}}$ and the analysis of $\widetilde K_{12}^{\mathrm{exp}}$ proceed as in the proof of Theorem~\ref{thm: Gaussian}, using the new contours and the estimates established above. We obtain the same pointwise kernel limits and uniform bounds, which yield convergence of the Fredholm Pfaffians by the same argument.
\end{proof}
\end{appendices}

\bibliographystyle{amsalpha}
\bibliography{bib.bib}

@article{basu2024temporal,
  title={Temporal correlation in the inverse-gamma polymer},
  author={Basu, Riddhipratim and Sepp{\"a}l{\"a}inen, Timo and Shen, Xiao},
  journal={Communications in Mathematical Physics},
  volume={405},
  number={7},
  pages={163},
  year={2024},
  publisher={Springer}
}

@article{he2026shift,
  title={Shift invariance of half space integrable models},
  author={He, Jimmy},
  journal={Probability Theory and Related Fields},
  volume={194},
  number={1},
  pages={541--611},
  year={2026},
  publisher={Springer}
}

@misc{emrah2020righttailmoderatedeviationsexponential,
      title={Right-tail moderate deviations in the exponential last-passage percolation}, 
      author={Elnur Emrah and Chris Janjigian and Timo Sepp{\"a}l{\"a}inen},
      year={2020},
      eprint={2004.04285},
      archivePrefix={arXiv},
      primaryClass={math.PR},
      url={https://arxiv.org/abs/2004.04285}, 
}

@article{bhatia2020moderate,
  title={Moderate Deviation and Exit Time Estimates for Stationary Last Passage Percolation: M. Bhatia},
  author={Bhatia, Manan},
  journal={Journal of Statistical Physics},
  volume={181},
  number={4},
  pages={1410--1432},
  year={2020},
  publisher={Springer}
}

@article{Small_Deviation,
author = {Michel Ledoux and Brian Rider},
title = {{Small Deviations for Beta Ensembles}},
volume = {15},
journal = {Electronic Journal of Probability},
number = {none},
publisher = {Institute of Mathematical Statistics and Bernoulli Society},
pages = {1319 -- 1343},
year = {2010},
doi = {10.1214/EJP.v15-798},
URL = {https://doi.org/10.1214/EJP.v15-798}
}

@article{basu2022interlacing,
  title={Interlacing and scaling exponents for the geodesic watermelon in last passage percolation},
  author={Basu, Riddhipratim and Ganguly, Shirshendu and Hammond, Alan and Hegde, Milind},
  journal={Communications in Mathematical Physics},
  volume={393},
  number={3},
  pages={1241--1309},
  year={2022},
  publisher={Springer}
}

@article{BR01a,
  author  = {Baik, Jinho and Rains, Eric M.},
  title   = {Algebraic aspects of increasing subsequences},
  journal = {Duke Mathematical Journal},
  volume  = {109},
  number  = {1},
  pages   = {1--65},
  year    = {2001}
}

@article{Martin2004,
  author  = {Martin, James B.},
  title   = {Limiting Shape for Directed Percolation Models},
  journal = {The Annals of Probability},
  volume  = {32},
  number  = {4},
  pages   = {2908--2937},
  year    = {2004},
  doi     = {10.1214/009117904000000838},
  url     = {https://doi.org/10.1214/009117904000000838}
}

@article{ginsburg2024pinning,
  title={Pinning, diffusive fluctuations, and Gaussian limits for half-space directed polymer models},
  author={Ginsburg, Victor},
  journal={Electronic Journal of Probability},
  volume={29},
  pages={1--34},
  year={2024},
  publisher={The Institute of Mathematical Statistics and the Bernoulli Society}
}

@article{BBP05,
  author  = {Baik, Jinho and Ben Arous, G{\'e}rard and P{\'e}ch{\'e}, Sandrine},
  title   = {Phase transition of the largest eigenvalue for nonnull complex sample covariance matrices},
  journal = {The Annals of Probability},
  volume  = {33},
  number  = {5},
  pages   = {1643--1697},
  year    = {2005},
  url     = {https://www.jstor.org/stable/3481698}
}

@article{BR01b,
  author  = {Baik, Jinho and Rains, Eric M.},
  title   = {The asymptotics of monotone subsequences of involutions},
  journal = {Duke Mathematical Journal},
  volume  = {109},
  number  = {2},
  pages   = {205--281},
  year    = {2001}
}

@incollection{BR01c,
  author    = {Baik, Jinho and Rains, Eric M.},
  title     = {Symmetrized random permutations},
  booktitle = {Random matrix models and their applications},
  series    = {Mathematical Sciences Research Institute Publications},
  volume    = {40},
  pages     = {1--19},
  publisher = {Cambridge University Press},
  address   = {Cambridge},
  year      = {2001}
}

@article{Johansson_2000,
   title={Shape Fluctuations and Random Matrices},
   volume={209},
   ISSN={1432-0916},
   url={http://dx.doi.org/10.1007/s002200050027},
   DOI={10.1007/s002200050027},
   number={2},
   journal={Communications in Mathematical Physics},
   publisher={Springer Science and Business Media LLC},
   author={Johansson, Kurt},
   year={2000},
   month=feb, pages={437–476} }

@misc{dauvergne2026half,
  title={The directed landscape in half-space},
  author={Dauvergne, Duncan and Zhang, Lingfu},
  year={2026},
  eprint={2604.10020},
  archivePrefix={arXiv},
  primaryClass={math.PR},
  url={https://arxiv.org/abs/2604.10020}
}

@article{Pfaffian_Schur,
author = {Jinho Baik and Guillaume Barraquand and Ivan Corwin and Toufic Suidan},
title = {{Pfaffian Schur processes and last passage percolation in a half-quadrant}},
volume = {46},
journal = {The Annals of Probability},
number = {6},
publisher = {Institute of Mathematical Statistics},
pages = {3015 -- 3089},
year = {2018},
doi = {10.1214/17-AOP1226},
URL = {https://doi.org/10.1214/17-AOP1226}
}

@article{basu2019temporal,
  title={Temporal correlation in last passage percolation with flat initial condition via Brownian comparison},
  author={Basu, Riddhipratim and Ganguly, Shirshendu and Zhang, Lingfu},
  journal={arXiv preprint arXiv:1912.04891},
  year={2019}
}

@Inbook{Basu2021,
author="Basu, Riddhipratim
and Ganguly, Shirshendu",
editor="Vares, Maria Eul{\'a}lia
and Fern{\'a}ndez, Roberto
and Fontes, Luiz Renato
and Newman, Charles M.",
title="Time Correlation Exponents in Last Passage Percolation",
bookTitle="In and Out of Equilibrium 3: Celebrating Vladas Sidoravicius",
year="2021",
publisher="Springer International Publishing",
address="Cham",
pages="101--123",
isbn="978-3-030-60754-8",
doi="10.1007/978-3-030-60754-8_5",
url="https://doi.org/10.1007/978-3-030-60754-8_5"
}

@article{landon2024tail,
  title={Tail bounds for the O’Connell-Yor polymer},
  author={Landon, Benjamin and Sosoe, Philippe},
  journal={Electronic Journal of Probability},
  volume={29},
  pages={1--47},
  year={2024},
  publisher={The Institute of Mathematical Statistics and the Bernoulli Society}
}

@article{Emrah_2025,
   title={Coupling Derivation of Optimal-Order Central Moment Bounds in Exponential Last-Passage Percolation},
   volume={192},
   ISSN={1572-9613},
   url={http://dx.doi.org/10.1007/s10955-025-03402-3},
   DOI={10.1007/s10955-025-03402-3},
   number={2},
   journal={Journal of Statistical Physics},
   publisher={Springer Science and Business Media LLC},
   author={Emrah, Elnur and Georgiou, Nicos and Ortmann, Janosch},
   year={2025},
   month=Jan }

@article{Georgiou_2013,
   title={Large deviation rate functions for the partition function in a log-gamma distributed random potential},
   volume={41},
   ISSN={0091-1798},
   url={http://dx.doi.org/10.1214/12-AOP768},
   DOI={10.1214/12-aop768},
   number={6},
   journal={The Annals of Probability},
   publisher={Institute of Mathematical Statistics},
   author={Georgiou, Nicos and Seppäläinen, Timo},
   year={2013},
   month=Nov }

@article{Landon2022UpperTB,
  title={Upper Tail Bounds for Stationary KPZ Models},
  author={Benjamin R. Landon and Philippe Sosoe},
  journal={Communications in Mathematical Physics},
  year={2022},
  pages={1-25},
  url={https://api.semanticscholar.org/CorpusID:251253393}
}

@article{schmid2023mixing,
  title={Mixing times for the TASEP in the maximal current phase},
  author={Schmid, Dominik},
  journal={The Annals of Probability},
  volume={51},
  number={4},
  pages={1342--1379},
  year={2023},
  publisher={Institute of Mathematical Statistics}
}

@inproceedings{barraquand2020half,
  title={Half-space Macdonald processes},
  author={Barraquand, Guillaume and Borodin, Alexei and Corwin, Ivan},
  booktitle={Forum of Mathematics, Pi},
  volume={8},
  pages={e11},
  year={2020},
  organization={Cambridge University Press}
}

@article{basu2014last,
  title={Last passage percolation with a defect line and the solution of the slow bond problem},
  author={Basu, Riddhipratim and Sidoravicius, Vladas and Sly, Allan},
  journal={arXiv preprint arXiv:1408.3464},
  year={2014}
}

@article{zhang2025convergence,
  title={Convergence from the log-gamma polymer to the directed landscape},
  author={Zhang, Xinyi},
  journal={arXiv preprint arXiv:2505.05685},
  year={2025}
}

@article{dauvergne2021scaling,
  title={The scaling limit of the longest increasing subsequence},
  author={Dauvergne, Duncan and Vir{\'a}g, B{\'a}lint},
  journal={arXiv preprint arXiv:2104.08210},
  year={2021}
}

@inproceedings{dauvergne2023uniform,
  title={Uniform convergence to the Airy line ensemble},
  author={Dauvergne, Duncan and Nica, Mihai and Vir{\'a}g, B{\'a}lint},
  booktitle={Annales de l'Institut Henri Poincare (B) Probabilites et statistiques},
  volume={59},
  number={4},
  pages={2220--2256},
  year={2023},
  organization={Institut Henri Poincar{\'e}}
}

@article{aggarwal2024scaling,
  title={Scaling limit of the colored ASEP and stochastic six-vertex models},
  author={Aggarwal, Amol and Corwin, Ivan and Hegde, Milind},
  journal={arXiv preprint arXiv:2403.01341},
  year={2024}
}

@article{barraquand2023identity,
  title={An identity in distribution between full-space and half-space log-gamma polymers},
  author={Barraquand, Guillaume and Wang, Shouda},
  journal={International Mathematics Research Notices},
  volume={2023},
  number={14},
  pages={11877--11929},
  year={2023},
  publisher={Oxford University Press}
}

@article{stat,
  title={Stationary Log-Gamma Polymer in Half-Space},
  author={Zeng, Jiyue and Zhang, Xinyi},
  journal={arXiv preprint arXiv:2602.19500},
  year={2026}
}

@article{dauvergne2022directed,
  title={The directed landscape},
  author={Dauvergne, Duncan and Ortmann, Janosch and Vir{\'a}g, B{\'a}lint},
  journal={Acta Mathematica},
  volume={229},
  number={2},
  pages={201--285},
  year={2022},
  publisher={Lehigh University Bethlehem, Penn., USA}
}

@article{wu2023kpz,
  title={The KPZ equation and the directed landscape},
  author={Wu, Xuan},
  journal={arXiv preprint arXiv:2301.00547},
  year={2023}
}

@misc{imamura2022solvablemodelskpzclass,
      title={Solvable models in the KPZ class: approach through periodic and free boundary Schur measures}, 
      author={Takashi Imamura and Matteo Mucciconi and Tomohiro Sasamoto},
      year={2022},
      eprint={2204.08420},
      archivePrefix={arXiv},
      primaryClass={math.PR},
      url={https://arxiv.org/abs/2204.08420}, 
}

@article{barraquand2021fluctuations,
  title={Fluctuations of the log-gamma polymer free energy with general parameters and slopes},
  author={Barraquand, Guillaume and Corwin, Ivan and Dimitrov, Evgeni},
  journal={Probability Theory and Related Fields},
  volume={181},
  number={1},
  pages={113--195},
  year={2021},
  publisher={Springer}
}
\end{document}